\documentclass[11pt,a4paper]{amsart}
\usepackage{amssymb,amsthm}
\usepackage{bm}

\usepackage[truedimen,margin=30truemm]{geometry}

\usepackage[pdftex,colorlinks=true,bookmarks=true,citecolor=blue,linkcolor=blue,pdfauthor={},pdftitle={}]{hyperref}

\theoremstyle{plain}
\newtheorem{theorem}{Theorem}[section]
\newtheorem{lemma}[theorem]{Lemma}
\newtheorem{proposition}[theorem]{Proposition}

\theoremstyle{definition}
\newtheorem{definition}[theorem]{Definition}

\theoremstyle{remark}
\newtheorem{remark}[theorem]{Remark}

\numberwithin{equation}{section}

\begin{document}

\title[Navier--Stokes equations in a thin spherical shell]{Difference estimate for weak solutions to the Navier--Stokes equations in a thin spherical shell and on the unit sphere}

\author[T.-H. Miura]{Tatsu-Hiko Miura}
\address{Graduate School of Science and Technology, Hirosaki University, 3, Bunkyo-cho, Hirosaki-shi, Aomori, 036-8561, Japan}
\email{thmiura623@hirosaki-u.ac.jp}



\begin{abstract}
  We consider the Navier--Stokes equations in a three-dimensional thin spherical shell and on the two-dimensional unit sphere, and estimate the difference of weak solutions on the thin spherical shell and the unit sphere.
  Assuming that the weak solution on the thin spherical shell is a Leray--Hopf weak solution satisfying the energy inequality, we derive difference estimates for the two weak solutions by the weak-strong uniqueness argument.
  The main idea is to extend the weak solution on the unit sphere properly to an approximate solution on the thin spherical shell, and to use the extension as a strong solution in the weak-strong uniqueness argument.
\end{abstract}

\maketitle

\section{Introduction} \label{S:Intro}

\subsection{Problem setting and main results} \label{SS:Int_Main}
For a sufficiently small $\varepsilon\in(0,1)$, let
\begin{align*}
  \Omega_\varepsilon := \{x\in\mathbb{R}^3 \mid 1<|x|<1+\varepsilon\}
\end{align*}
be a thin spherical shell in $\mathbb{R}^3$.
Also, let $\mathbf{n}_\varepsilon$ be the unit outward normal vector field of $\partial\Omega_\varepsilon$ and $\mathbf{P}_\varepsilon:=\mathbf{I}_3-\mathbf{n}_\varepsilon\otimes\mathbf{n}_\varepsilon$ be the orthogonal projection onto the tangent plane of $\partial\Omega_\varepsilon$, where $\mathbf{I}_3$ is the $3\times3$ identity matrix and $\mathbf{a}\otimes\mathbf{b}$ is the tensor product of $\mathbf{a},\mathbf{b}\in\mathbb{R}^3$.

In this paper, we consider the three-dimensional (3D) Navier--Stokes equations under the Navier boundary conditions
\begin{align} \label{E:NS_TSS}
  \left\{
  \begin{aligned}
    &\partial_t\mathbf{u}^\varepsilon+(\mathbf{u}^\varepsilon\cdot\nabla)\mathbf{u}^\varepsilon-\nu\Delta\mathbf{u}^\varepsilon+\nabla p^\varepsilon = \mathbf{f}^\varepsilon, \quad \mathrm{div}\,\mathbf{u}^\varepsilon = 0 \quad\text{in}\quad \Omega_\varepsilon\times(0,\infty), \\
    &\mathbf{u}^\varepsilon\cdot\mathbf{n}_\varepsilon = 0, \quad \mathbf{P}_\varepsilon\mathbf{D}(\mathbf{u}^\varepsilon)\mathbf{n}_\varepsilon = \mathbf{0} \quad\text{on}\quad \partial\Omega_\varepsilon\times(0,\infty), \\
    &\mathbf{u}^\varepsilon|_{t=0} = \mathbf{u}_0^\varepsilon \quad\text{in}\quad \Omega_\varepsilon.
  \end{aligned}
  \right.
\end{align}
Here, $\mathbf{f}^\varepsilon$ and $\mathbf{u}_0^\varepsilon$ are a given external force and initial velocity, respectively, and $\nu>0$ is the viscosity coefficient.
We assume that $\nu$ is independent of $\varepsilon$.
Also, $\mathbf{D}(\mathbf{u}^\varepsilon)$ is the strain rate tensor given as the symmetric part of $\nabla\mathbf{u}^\varepsilon$.
The first boundary condition states that the fluid does not pass through the boundary.
The second one describes the perfect slip, which means that the fluid slips on the boundary without friction.

The purpose of this paper is to compare solutions to \eqref{E:NS_TSS} with solutions to the two-dimensional (2D) Navier--Stokes equations on the unit sphere $S^2$ in $\mathbb{R}^3$ of the form
\begin{align} \label{E:NS_S2}
  \left\{
  \begin{aligned}
    &\partial_t\mathbf{v}+\nabla_{\mathbf{v}}\mathbf{v}-2\nu\mathbf{P}\,\mathrm{div}_{S^2}[\mathbf{D}_{S^2}(\mathbf{v})]+\nabla_{S^2}q = \mathbf{f}, \quad \mathrm{div}_{S^2}\mathbf{v} = 0 \quad\text{on}\quad S^2\times(0,\infty), \\
    &\mathbf{v}|_{t=0} = \mathbf{v}_0 \quad\text{on}\quad S^2.
  \end{aligned}
  \right.
\end{align}
Here, $\mathbf{v}$ is the tangential velocity of a fluid, and $\mathbf{f}$ and $\mathbf{v}_0$ are a given external force and initial velocity, respectively.
Also, $\nabla_{\mathbf{v}}\mathbf{v}$ is the covariant derivative of $\mathbf{v}$ along itself, $\nabla_{S^2}$ is the tangential gradient, and $\mathrm{div}_{S^2}$ is the surface divergence.
We write $\mathbf{P}$ for the orthogonal projection onto the tangent plane of $S^2$ and $\mathbf{D}_{S^2}(\mathbf{v})$ for the surface strain rate tensor (see Section \ref{S:Prelim} for details).
The system \eqref{E:NS_S2} can be rewritten as
\begin{align*}
  \partial_t\mathbf{v}+\nabla_{\mathbf{v}}\mathbf{v}-\nu(\Delta_H\mathbf{v}+2\mathbf{v})+\nabla_{S^2}q = \mathbf{f}, \quad \mathrm{div}_{S^2}\mathbf{v} = 0 \quad\text{on}\quad S^2\times(0,\infty),
\end{align*}
where $\Delta_H$ is the Hodge Laplacian (see \cite[Lemma C.11]{Miu20_03}).
This form is intrinsic in the sense that it involves only the first fundamental form of $S^2$ and thus can be considered without the ambient space $\mathbb{R}^3$.
In this paper, we use the extrinsic form \eqref{E:NS_S2} which requires the ambient space $\mathbb{R}^3$, since it is more convenient for our analysis.

The relation between the two systems \eqref{E:NS_TSS} and \eqref{E:NS_S2} was studied in our previous work \cite{Miu20_03} for curved thin domains around general closed surfaces (see also \cite{Miu24_SNS} for the stationary case).
There, we considered a strong solution to \eqref{E:NS_TSS} in the $L^2$ class and proved that the average in the thin direction of a strong solution converges to a (unique) weak solution to \eqref{E:NS_S2} as $\varepsilon\to0$ in an appropriate sense.
Moreover, we derived estimates for the difference of a strong solution to \eqref{E:NS_TSS} and a weak solution to \eqref{E:NS_S2}, which roughly state that the difference of the two solutions is of order $\varepsilon$.
In this paper, we establish a similar difference estimate by assuming only that a solution to \eqref{E:NS_TSS} is a Leray--Hopf weak solution.

To state main results, we introduce some notations (see Section \ref{S:Prelim} for details).
Let
\begin{align*}
  \mathcal{H}_\varepsilon := \{\mathbf{u}\in L^2(\Omega_\varepsilon)^3 \mid \text{$\mathrm{div}\,\mathbf{u}=0$ in $\Omega_\varepsilon$, $\mathbf{u}\cdot\mathbf{n}_\varepsilon=0$ on $\partial\Omega_\varepsilon$}\}, \quad \mathcal{V}_\varepsilon := \mathcal{H}_\varepsilon\cap H^1(\Omega_\varepsilon)^3
\end{align*}
be the spaces of solenoidal vector fields on $\Omega_\varepsilon$.
Also, let
\begin{align*}
  \mathcal{H}_0 := \{\mathbf{v}\in L^2(S^2)^3 \mid \text{$\mathbf{v}\cdot\mathbf{n}=0$, $\mathrm{div}_{S^2}\,\mathbf{v}=0$ on $S^2$}\}, \quad \mathcal{V}_0:=\mathcal{H}_0\cap H^1(S^2)^3
\end{align*}
be the spaces of tangential and solenoidal vector fields on $S^2$, where $\mathbf{n}$ is the unit outward normal vector field of $S^2$.
We denote by $\mathcal{V}_\varepsilon^\ast$ and $\mathcal{V}_0^\ast$ the dual spaces of $\mathcal{V}_\varepsilon$ and $\mathcal{V}_0$, respectively.
When $\eta$ is a function on $S^2$, we set $\bar{\eta}(x):=\eta(x/|x|)$ for $x\in\mathbb{R}^3\setminus\{0\}$, which is the constant extension of $\eta$ in the radial direction.
By duality, we also define the ``constant extension'' $\bar{\mathbf{f}}$ of $\mathbf{f}\in\mathcal{V}_0^\ast$ (see \eqref{E:Def_fext} for the precise definition).
For a vector field $\mathbf{u}$ on $\Omega_\varepsilon$, let $\partial_{\mathbf{n}}\mathbf{u}:=(\bar{\mathbf{n}}\cdot\nabla)\mathbf{u}$ be the derivative of $\mathbf{u}$ in the direction $\bar{\mathbf{n}}$ (i.e. the radial direction).

Now, let us present the main results of this paper.
We refer to Definitions \ref{D:WS_TSS} and \ref{D:WS_S2} for the definitions of weak solutions to \eqref{E:NS_TSS} and \eqref{E:NS_S2}, respectively.

\begin{theorem} \label{T:Diff_Est}
  For given data
  \begin{align*}
    (\mathbf{u}_0^\varepsilon,\mathbf{f}^\varepsilon) \in \mathcal{H}_\varepsilon \times L_{\mathrm{loc}}^2([0,\infty);\mathcal{V}_\varepsilon^\ast), \quad (\mathbf{v}_0,\mathbf{f}) \in \mathcal{H}_0 \times L_{\mathrm{loc}}^2([0,\infty);\mathcal{V}_0^\ast),
  \end{align*}
  let $\mathbf{u}^\varepsilon$ and $\mathbf{v}$ be weak solutions to \eqref{E:NS_TSS} and \eqref{E:NS_S2}, respectively.
  Also, for $t\geq0$, let
  \begin{align} \label{E:Def_Diff}
    \begin{aligned}
      D_{\mathrm{data}}^\varepsilon(t) &:= \frac{1}{\varepsilon}\|\mathbf{u}_0^\varepsilon-\bar{\mathbf{v}}_0\|_{L^2(\Omega_\varepsilon)}^2+\frac{1}{\varepsilon\nu}\int_0^t\|\mathbf{f}^\varepsilon-\bar{\mathbf{f}}\|_{\mathcal{V}_\varepsilon^\ast}^2\,ds, \\
      D_{\mathrm{sol}}^\varepsilon(t) &:= \frac{1}{\varepsilon}\|[\mathbf{u}^\varepsilon-\bar{\mathbf{v}}](t)\|_{L^2(\Omega_\varepsilon)}^2 \\
      &\qquad\qquad+\frac{\nu}{\varepsilon}\int_0^t\Bigl(\Bigl\|\overline{\mathbf{P}}\nabla\mathbf{u}^\varepsilon-\overline{\nabla_{S^2}\mathbf{v}}\Bigr\|_{L^2(\Omega_\varepsilon)}^2+\|\partial_{\mathbf{n}}\mathbf{u}^\varepsilon-\bar{\mathbf{v}}\|_{L^2(\Omega_\varepsilon)}^2\Bigr)\,ds.
    \end{aligned}
  \end{align}
  Suppose that $\mathbf{u}^\varepsilon$ is a Leray--Hopf weak solution, i.e. the energy inequality
  \begin{align} \label{E:TSS_Ener}
    \frac{1}{2}\|\mathbf{u}^\varepsilon(t)\|_{L^2(\Omega_\varepsilon)}^2+2\nu\int_0^t\|\mathbf{D}(\mathbf{u}^\varepsilon)\|_{L^2(\Omega_\varepsilon)}^2\,ds \leq \frac{1}{2}\|\mathbf{u}_0^\varepsilon\|_{L^2(\Omega_\varepsilon)}^2+\int_0^t\langle\mathbf{f}^\varepsilon,\mathbf{u}^\varepsilon\rangle_{\mathcal{V}_\varepsilon}\,ds
  \end{align}
  holds for all $t\geq0$.
  Then, we have
  \begin{align} \label{E:Diff_Est}
    D_{\mathrm{sol}}^\varepsilon(t) \leq C_1F_{\mathbf{v}}(t)\{D_{\mathrm{data}}^\varepsilon(t)+\varepsilon^2G_{\mathbf{v}}(t)\} \quad\text{for all}\quad t\geq0.
  \end{align}
  Here, $C_1>0$ is a constant independent of $\varepsilon$, $t$, and $\nu$.
  Also,
  \begin{align} \label{E:Def_FG}
    \begin{aligned}
      F_{\mathbf{v}}(t) &:= \exp\left(C_2\int_0^t\left\{\nu+\frac{1}{\nu^3}\|\mathbf{v}\|_{L^2(S^2)}^2\|\mathbf{v}\|_{H^1(S^2)}^2\right\}\,ds\right), \\
      G_{\mathbf{v}}(t) &:= \|\mathbf{v}_0\|_{L^2(S^2)}^2+\|\mathbf{v}(t)\|_{L^2(S^2)}^2 \\
      &\qquad +\frac{1}{\nu}\int_0^t\Bigl\{\|\mathbf{f}\|_{\mathcal{V}_0^\ast}^2+\Bigl(\nu^2+\|\mathbf{v}\|_{L^2(S^2)}^2\Bigr)\|\mathbf{v}\|_{H^1(S^2)}^2\Bigr\}\,ds
    \end{aligned}
  \end{align}
  with a constant $C_2>0$ independent of $\varepsilon$, $t$, and $\nu$.
\end{theorem}

In \eqref{E:Def_Diff}, we divide the integrals over $\Omega_\varepsilon$ by the thickness $\varepsilon$ of $\Omega_\varepsilon$.
Thus, taking the square root of \eqref{E:Diff_Est} (and \eqref{E:DEst_Gl} below), we can say that the weak solution $\mathbf{v}$ to \eqref{E:NS_S2} approximates the weak solution $\mathbf{u}^\varepsilon$ to \eqref{E:NS_TSS} of order $\varepsilon$.

Let $\mathcal{M}_\varepsilon^0\mathbf{u}^\varepsilon$ be the average of $\mathbf{u}^\varepsilon$ in the radial direction given by
\begin{align*}
  \mathcal{M}_\varepsilon^0\mathbf{u}^\varepsilon(y,t) := \frac{1}{\varepsilon}\int_1^{1+\varepsilon}\mathbf{u}^\varepsilon(ry,t)\,dr, \quad (y,t)\in S^2\times(0,\infty)
\end{align*}
The difference estimate \eqref{E:Diff_Est} implies the convergence of $\mathcal{M}_\varepsilon^0\mathbf{u}^\varepsilon$ as $\varepsilon\to0$.

\begin{theorem} \label{T:Ave_Conv}
  Under the assumptions of Theorem \ref{T:Diff_Est}, suppose further that
  \begin{align} \label{E:ACL_data}
    \lim_{\varepsilon\to0}\varepsilon^{-1/2}\|\mathbf{u}_0^\varepsilon-\bar{\mathbf{v}}_0\|_{L^2(\Omega_\varepsilon)} = 0, \quad \lim_{\varepsilon\to0}\varepsilon^{-1/2}\|\mathbf{f}^\varepsilon-\bar{\mathbf{f}}\|_{L^2(0,T;\mathcal{V}_\varepsilon^\ast)} = 0
  \end{align}
  for any fixed finite $T>0$.
  Then,
  \begin{align} \label{E:ACL_sol}
    \lim_{\varepsilon\to0}\mathcal{M}_\varepsilon^0\mathbf{u}^\varepsilon = \mathbf{v} \quad\text{strongly in}\quad L^\infty(0,T;L^2(S^2)^3)\cap L^2(0,T;H^1(S^2)^3).
  \end{align}
\end{theorem}

The difference estimate \eqref{E:Diff_Est} is a local-in-time estimate in the sense that the functions $F_{\mathbf{v}}(t)$ and $G_{\mathbf{v}}(t)$ given by \eqref{E:Def_FG} are increasing in $t$.
In particular, $F_{\mathbf{v}}(t)\geq e^{C_2t}$.
If we make a few additional assumptions, we can get a global-in-time difference estimate.

Let $\mathbf{a}\in\mathbb{R}^3$.
We set $\mathbf{r}_{\mathbf{a}}(x):=\mathbf{a}\times x$ for $x\in\mathbb{R}^3$, where $\times$ is the vector product in $\mathbb{R}^3$.
For the sake of simplicity, we use the same notation $\mathbf{r}_{\mathbf{a}}$ for the restriction of $\mathbf{r}_{\mathbf{a}}$ to $\Omega_\varepsilon$ and $S^2$.
Note that the restriction of $\mathbf{r}_{\mathrm{a}}$ is in $\mathcal{V}_\varepsilon$ as well as in $\mathcal{V}_0$ (see Lemmas \ref{L:Rota_TSS} and \ref{L:Rota_S2}).

\begin{theorem} \label{T:DEst_Gl}
  Let $\varepsilon_0\in(0,1)$ be the constant given in Lemma \ref{L:KoH1_TSS} and $\varepsilon\in(0,\varepsilon_0]$.
  Under the assumptions of Theorem \ref{T:Diff_Est}, suppose further that
  \begin{align*}
    (\mathbf{u}_0^\varepsilon,\mathbf{r}_{\mathbf{a}})_{L^2(\Omega_\varepsilon)} = 0, \quad \langle\mathbf{f}^\varepsilon(t),\mathbf{r}_{\mathbf{a}}\rangle_{\mathcal{V}_\varepsilon} = 0, \quad (\mathbf{v}_0,\mathbf{r}_{\mathbf{a}})_{L^2(S^2)} = 0, \quad \langle\mathbf{f}(t),\mathbf{r}_{\mathbf{a}}\rangle_{\mathcal{V}_0} = 0
  \end{align*}
  for all $\mathbf{a}\in\mathbb{R}^3$ and a.a. $t\in(0,\infty)$, and that $\mathbf{f} \in L^2(0,\infty;\mathcal{V}_0^\ast)$.
  Then,
  \begin{align} \label{E:DEst_Gl}
    D_{\mathrm{sol}}^\varepsilon(t)+\frac{\nu}{\varepsilon}\int_0^t\|\mathbf{u}^\varepsilon-\bar{\mathbf{v}}\|_{L^2(\Omega_\varepsilon)}^2\,ds \leq C_3F_0\{D_{\mathrm{data}}^\varepsilon(t)+\varepsilon^2G_0\} \quad\text{for all}\quad t\geq0.
  \end{align}
  Here, $C_3>0$ is a constant independent of $\varepsilon$, $t$, and $\nu$.
  Also,
  \begin{align} \label{E:Def_GlCon}
    \begin{aligned}
      E_0 &:= \|\mathbf{v}_0\|_{L^2(S^2)}^2+\frac{1}{\nu}\int_0^\infty\|\mathbf{f}\|_{\mathcal{V}_0^\ast}^2\,dt, \\
      F_0 &:= \exp\left(\frac{C_4}{\nu^2}E_0^2\right), \quad G_0 := E_0+\frac{E_0^2}{\nu^2}
    \end{aligned}
  \end{align}
  with a constant $C_4>0$ independent of $\varepsilon$, $t$, and $\nu$.
  In particular,
  \begin{align} \label{E:ACG_sol}
    \lim_{\varepsilon\to0}\mathcal{M}_\varepsilon^0\mathbf{u}^\varepsilon = \mathbf{v} \quad \text{strongly in} \quad L^\infty(0,\infty;L^2(S^2)^3)\cap L^2(0,\infty;H^1(S^2)^3)
  \end{align}
  provided that $\mathbf{f}^\varepsilon \in L^2(0,\infty;\mathcal{V}_\varepsilon^\ast)$ and
  \begin{align} \label{E:ACG_data}
    \lim_{\varepsilon\to0}\varepsilon^{-1/2}\|\mathbf{u}_0^\varepsilon-\bar{\mathbf{v}}_0\|_{L^2(\Omega_\varepsilon)} = 0, \quad \lim_{\varepsilon\to0}\varepsilon^{-1/2}\|\mathbf{f}^\varepsilon-\bar{\mathbf{f}}\|_{L^2(0,\infty;\mathcal{V}_\varepsilon^\ast)} = 0.
  \end{align}
\end{theorem}

We note that the local difference estimate \eqref{E:Diff_Est} was derived in \cite[Theorems 7.29 and 7.30]{Miu20_03} but under additional assumptions including
\begin{align*}
  \mathbf{u}_0^\varepsilon \in \mathcal{V}_\varepsilon, \quad \mathbf{f}^\varepsilon\in L^\infty(0,\infty;L^2(\Omega_\varepsilon)^3), \quad (\mathbf{u}_0^\varepsilon,\mathbf{r}_{\mathbf{a}})_{L^2(\Omega_\varepsilon)} = 0 , \quad (\mathbf{f}^\varepsilon(t),\mathbf{r}_{\mathbf{a}})_{L^2(\Omega_\varepsilon)} = 0
\end{align*}
for all $\mathbf{a}\in\mathbb{R}^3$ and a.a. $t\in(0,\infty)$.
Moreover, it seems that the global difference esitmate \eqref{E:DEst_Gl} has not been obtained in the literature, even in the case of a thin product domain $\omega\times(0,\varepsilon)$ around a 2D flat domain $\omega$.

\subsection{Idea of proof} \label{SS:Int_Idea}
Let us explain the ideas of the proofs of the main results (see Sections \ref{S:Diff} and \ref{S:Global} for details).
The convergence of the average $\mathcal{M}_\varepsilon^0\mathbf{u}^\varepsilon$ in Theorems \ref{T:Ave_Conv} and \ref{T:DEst_Gl} easily follows from the difference estimates \eqref{E:Diff_Est} and \eqref{E:DEst_Gl} and a few properties of the average (see Section \ref{SS:AE_Ave}).
Also, the global difference estimate \eqref{E:DEst_Gl} is shown by a refinement of the proof of the local difference estimate \eqref{E:Diff_Est} with the aid of the (uniform) Korn inequalities
\begin{align*}
  \|\mathbf{u}^\varepsilon\|_{H^1(\Omega_\varepsilon)} \leq c\|\mathbf{D}(\mathbf{u}^\varepsilon)\|_{L^2(\Omega_\varepsilon)}, \quad \|\mathbf{v}\|_{H^1(S^2)} \leq c\|\mathbf{D}_{S^2}(\mathbf{v})\|_{L^2(S^2)}
\end{align*}
with a constant $c>0$ independent of $\varepsilon$, $t$, and $\nu$, which are applicable to the weak solutions $\mathbf{u}^\varepsilon$ to \eqref{E:NS_TSS} and $\mathbf{v}$ to \eqref{E:NS_S2} under the assumptions of Theorem \ref{T:DEst_Gl} (see Sections \ref{SS:Gl_Korn} and \ref{SS:Gl_Orth}).
Thus, the main effort is devoted to the proof of the local difference estimate \eqref{E:Diff_Est}.

Since the weak solutions $\mathbf{u}^\varepsilon$ to \eqref{E:NS_TSS} and $\mathbf{v}$ to \eqref{E:NS_S2} are defined on the different domains $\Omega_\varepsilon$ and $S^2$, we need to make their domains the same in order to compare them by using either of the systems \eqref{E:NS_TSS} or \eqref{E:NS_S2}.
In our previous work \cite{Miu20_03}, we considered the average of a strong solution to \eqref{E:NS_TSS} as an approximate solution to \eqref{E:NS_S2} with some error terms, and estimated the error terms by the $L_t^2H_x^2$- and $L_t^\infty H_x^1$-norms of the strong solution.
Here, we only assume that $\mathbf{u}^\varepsilon$ is a Leray--Hopf weak solution to \eqref{E:NS_TSS}, so the approach of \cite{Miu20_03} is not applicable.
Thus, instead, we extend $\mathbf{v}$ to an approximate solution to \eqref{E:NS_TSS}.

Contrary to the 2D problem \eqref{E:NS_S2}, we cannot apply a standard energy method to weak solutions of the 3D problem \eqref{E:NS_TSS} due to the lack of the integrability of the time derivative.
Thus, to derive the difference estimate \eqref{E:Diff_Est}, we use the weak-strong uniqueness argument (see e.g. \cite{Soh01,BoyFab13}) in which we expect an extension of $\mathbf{v}$ to play a role of a strong solution.
However, the constant extension $\bar{\mathbf{v}}(x)=\mathbf{v}(x/|x|)$ of $\mathbf{v}$ is not appropriate for this purpose, since the strain rate tensor of $\bar{\mathbf{v}}$ is of the form
\begin{align*}
  [\mathbf{D}(\bar{\mathbf{v}})](x) = \frac{1}{|x|}\left\{\Bigl[\overline{\mathbf{D}_{S^2}(\mathbf{v})}\Bigr](x)-\frac{1}{2}[\bar{\mathbf{v}}\otimes\bar{\mathbf{n}}](x)-\frac{1}{2}[\bar{\mathbf{n}}\otimes\bar{\mathbf{v}}](x)\right\}, \quad x\in\mathbb{R}^3\setminus\{0\},
\end{align*}
which follows from \eqref{E:TGr_Dec} and \eqref{E:Const}, and the last two additional terms cannot be negligible in a weak form of \eqref{E:NS_TSS}.
To eliminate them, we introduce the weighted extension
\begin{align*}
  \mathbf{v}_{\mathrm{E}}(x) := |x|\,\bar{\mathbf{v}}(x) = |x|\,\mathbf{v}\left(\frac{x}{|x|}\right), \quad x\in\mathbb{R}^3\setminus\{0\}.
\end{align*}
Then, it turns out (see Lemmas \ref{L:Ext_Grad} and \ref{L:Ext_Sol}) that $\mathbf{v}_{\mathrm{E}}$ satisfies
\begin{align*}
  [\mathbf{D}(\mathbf{v}_{\mathrm{E}})](x) = \Bigl[\overline{\mathbf{D}_{S^2}(\mathbf{v})}\Bigr](x), \quad \mathrm{div}\,\mathbf{v}_{\mathrm{E}}(x) = \overline{\mathrm{div}_{S^2}\mathbf{v}}(x) = 0, \quad x\in\mathbb{R}^3\setminus\{0\}
\end{align*}
and $\mathbf{v}_{\mathrm{E}}\cdot\mathbf{n}_\varepsilon=0$ on $\partial\Omega_\varepsilon$.
Based on this observation, we use $\mathbf{v}_{\mathrm{E}}$ as an approximate solution to \eqref{E:NS_TSS} and apply the weak-strong uniqueness argument to $\mathbf{u}^\varepsilon$ and $\mathbf{v}_{\mathrm{E}}$.

In the actual proof, we need to first show that $\mathbf{v}_{\mathrm{E}}$ is indeed an approximate solution to \eqref{E:NS_TSS}, i.e. it satisfies a weak form of \eqref{E:NS_TSS} with small error terms.
To this end, we take a test function $\bm{\psi}$ of a weak form of \eqref{E:NS_TSS} and substitute the weighted average
\begin{align*}
  \mathcal{M}_\varepsilon^3\bm{\psi}(y) := \frac{1}{\varepsilon}\int_1^{1+\varepsilon}\bm{\psi}(ry)r^3\,dr, \quad y\in S^2
\end{align*}
for a test function of a weak form of \eqref{E:NS_S2}.
Then, we ``unfold'' the average like
\begin{align*}
  \varepsilon\int_{S^2}\mathbf{v}(y)\cdot\mathcal{M}_\varepsilon^3\bm{\psi}(y)\,d\mathcal{H}^2(y) &= \int_{S^2}\int_1^{1+\varepsilon}\mathbf{v}(y)\cdot\bm{\psi}(ry)r^3\,dr\,d\mathcal{H}^2(y) \\
  &= \int_{\Omega_\varepsilon}\mathbf{v}_{\mathrm{E}}(x)\cdot\bm{\psi}(x)\,dx
\end{align*}
by using the polar coordinates $dx=r^2dr\,d\mathcal{H}^2(y)$, where $\mathcal{H}^2$ is the 2D Hausdorff measure (a similar argument was used in the study \cite{Miu25_GL} of the Ginzburg--Landau heat flow).
By this procedure, we find that $\mathbf{v}_{\mathrm{E}}$ satisfies a weak form of \eqref{E:NS_TSS} with error terms coming from the exchange of the derivatives with the average and extension.
Moreover, we can estimate the error terms to show that they are actually small by calculating the derivatives of the average and extension directly (see Section \ref{S:AvEx}).

Another issue is in the regularity of $\mathbf{u}^\varepsilon$ and $\mathbf{v}_{\mathrm{E}}$ and the integrability of trilinear terms.
The weak solution $\mathbf{v}$ to the 2D problem \eqref{E:NS_S2} satisfies (see Proposition \ref{P:S2_Cont})
\begin{align*}
  \mathbf{v} \in L_{\mathrm{loc}}^2([0,\infty);\mathcal{V}_0)\cap C([0,\infty);\mathcal{H}_0), \quad \partial_t\mathbf{v} \in L_{\mathrm{loc}}^2([0,\infty);\mathcal{V}_0^\ast),
\end{align*}
and its extension $\mathbf{v}_{\mathrm{E}}$ has the same regularity on $\Omega_\varepsilon$.
Thus, we can take $\mathbf{v}_{\mathrm{E}}$ itself as a test function of a weak form satisfied by $\mathbf{v}_{\mathrm{E}}$ to get an energy equality for $\mathbf{v}_{\mathrm{E}}$ (with error terms).
On the other hand, we have (see Proposition \ref{P:TSS_WeCo})
\begin{align*}
  \mathbf{u}^\varepsilon \in L_{\mathrm{loc}}^2([0,\infty);\mathcal{V}_\varepsilon)\cap C_{\mathrm{weak}}([0,\infty);\mathcal{H}_\varepsilon), \quad \partial_t\mathbf{u}^\varepsilon \in L_{\mathrm{loc}}^{4/3}([0,\infty);\mathcal{V}_\varepsilon^\ast)
\end{align*}
for the weak solution $\mathbf{u}^\varepsilon$ to \eqref{E:NS_TSS}, where $C_{\mathrm{weak}}$ means the weak continuity in time.
Then, at first glance, it seems that the time integrals of the trilinear terms
\begin{align*}
  \int_0^T\bigl((\mathbf{v}_{\mathrm{E}}\cdot\nabla)\mathbf{v}_{\mathrm{E}},\mathbf{u}^\varepsilon\bigr)_{L^2(\Omega_\varepsilon)}\,dt, \quad \int_0^T\bigl((\mathbf{u}^\varepsilon\cdot\nabla)\mathbf{u}^\varepsilon,\mathbf{v}_{\mathrm{E}}\bigr)_{L^2(\Omega_\varepsilon)}\,dt, \quad T>0
\end{align*}
are not well-defined just by the above regularity of $\mathbf{u}^\varepsilon$ and $\mathbf{v}_{\mathrm{E}}$.
It turns out that, however, these integrals actually make sense due to the fact that $\mathbf{v}_{\mathrm{E}}$ is essentially 2D.
Indeed, we have $|\mathbf{v}_{\mathrm{E}}|\leq2|\bar{\mathbf{v}}|$ in $\Omega_\varepsilon$ by $|x|\leq2$, and the following good product estimate
\begin{align*}
  \|\bar{\eta}\varphi\|_{L^2(\Omega_\varepsilon)} \leq c\|\eta\|_{L^2(S^2)}^{1/2}\|\eta\|_{H^1(S^2)}^{1/2}\|\varphi\|_{L^2(\Omega_\varepsilon)}^{1/2}\|\varphi\|_{H^1(\Omega_\varepsilon)}^{1/2}
\end{align*}
holds for $\varphi\in H^1(\Omega_\varepsilon)$ and the constant extension $\bar{\eta}$ of $\eta\in H^1(S^2)$ (see Lemma \ref{L:Quad_Thin}), which is not valid in general if $\bar{\eta}$ is replaced by a function on $\Omega_\varepsilon$.
This estimate ensures that the above regularity of $\mathbf{u}^\varepsilon$ and $\mathbf{v}$ is sufficient for the integrability of the trilinear terms.
Thus, we can proceed the weak-strong uniqueness argument with $\mathbf{u}^\varepsilon$ and $\mathbf{v}_{\mathrm{E}}$.

\begin{remark} \label{R:Int_Idea}
  The above idea also applies to the case of a thin product domain $\omega\times(0,\varepsilon)$ around a 2D domain $\omega$.
  In that case, it is sufficient to take the constant extension of a weak solution on $\omega$ in the vertical direction to proceed the weak-strong uniqueness argument.
\end{remark}

\subsection{Literature overview} \label{SS:Int_Lit}
Fluid equations in thin domains arise in problems of natural sciences like lubrication and geophysical fluid dynamics (see e.g. \cite{Ped87,OckOck95}).
The mathematical study of the (evolutionary) Navier--Stokes equations in 3D thin domains was initiated by Raugel and Sell \cite{RauSel93_03,RauSel93_01,RauSel94_02}, who considered a thin product domain $\omega\times(0,\varepsilon)$ around a 2D domain $\omega$.
Since then, many authors have studied the Navier--Stokes equations in thin product domains (see \cite{TemZia96,MoTeZi97,Ift99,Mon99,IftRau01,KukZia06,Hu07,KukZia07} and the references cited therein).
There are also several works dealing with flat thin domains with nonflat top and bottom boundaries \cite{IfRaSe07,Hoa10,HoaSel10,Hoa13}, thin spherical shells \cite{TemZia97,BrDhLe_21}, and curved thin domains around general closed surfaces \cite{Miu20_03,Miu21_02,Miu22_01}.

In the above works, one of the main topics is the global-in-time existence of a strong solution for large data depending on the thickness of thin domains, since it is expected that the situation is similar to the 2D case for 3D thin domains with very small thickness.
Another main topic is the thin-film limit, i.e. the limit as the thickness of thin domains tends to zero, and some of the above works obtained the convergence of solutions on 3D thin domains to solutions on 2D limit sets in the thin film limit.
However, they assumed solutions on thin domains to be strong ones and used the regularity of strong solutions to prove the convergence results.
This is probably because the thin-film limit problem is usually studied after getting the global existence of strong solutions.

The convergence results of weak solutions on thin domains have been obtained in the compressible case \cite{MalNov14} and the imhomogeneous incompressible case \cite{SunWan20}.
In these works, the authors considered thin product domains and proved the convergence results by showing difference estimates in terms of relative entropy functionals, but under the assumption that solutions on 2D limit domains are sufficiently smooth.
It seems that difference estimates have been first obtained in this paper when both solutions on thin domains and limit sets are assumed to be only weak ones.

The Navier--Stokes equations on the 2D unit sphere have been studied well in view of applications to geophysical fluid dynamics.
We refer to \cite{IliFil88,TemWan93,SkiAde98,CaRaTi99,Ski02,Ili04,ChaYon13,SaTaYa13,SaTaYa15,Wir15,Ski17,Miu22_L2K,MaeMiu22,Miu23_NL2K} and the references cited therein.
Also, fluid equations on more general surfaces appear in various fields such as biology and engineering (see e.g. \cite{SlSaOh07,ArrDeS09,ToMiAr19}), and there is a large literature on the Navier--Stokes equations on surfaces and manifolds (see e.g. \cite{Ili90,Tay92,Prie94,Nag99,MitTay01,DinMit04,KheMis12,ChaCzu13,ChaCzu16,Lic16,ChCzDi17,Pie17,KorWen18,SamTuo20,PrSiWi21}).
Recently, the Navier--Stokes equations on moving surfaces have been also proposed and studied in \cite{KoLiGi17,JaOlRe18,Miu18,OlReZh22}.
It is the topic of future works to extend the results of this paper to curved thin domains around general surfaces, but one of main difficulties is to find an appropriate extension of a solution to the surface problem which solves the thin-domain problem approximately.

\subsection{Organization of the paper} \label{SS:Int_Org}
The rest of this paper is organized as follows.
We fix notations and give basic results in Section \ref{S:Prelim}.
Section \ref{S:Weak} provides the definition and some properties of weak solutions to \eqref{E:NS_TSS} and \eqref{E:NS_S2}.
In Section \ref{S:AvEx}, we study the average of a function on $\Omega_\varepsilon $ and the extension of a function on $S^2$ in the radial direction.
The proofs of Theorems \ref{T:Diff_Est} and \ref{T:Ave_Conv} are given in Section \ref{S:Diff}.
Also, we prove Theorem \ref{T:DEst_Gl} in Section \ref{S:Global}.

\section{Preliminaries} \label{S:Prelim}
This section gives basic notations and results used in the following sections.

\subsection{Basic notations} \label{SS:Pre_Not}
Throughout this paper, the symbol $c$ stands for a general positive constant independent of $\varepsilon$, $t$, and the viscosity coefficient $\nu$.

Vectors $\mathbf{a}\in\mathbb{R}^3$ and matrices $\mathbf{A}\in\mathbb{R}^{3\times 3}$ are written as
\begin{align*}
  \mathbf{a} =
  \begin{pmatrix}
    a_1 \\ a_2 \\ a_3
  \end{pmatrix}
  = (a_1,a_2,a_3)^{\mathrm{T}}, \quad \mathbf{A} = (A_{ij})_{i,j} =
  \begin{pmatrix}
    A_{11} & A_{12} & A_{13} \\
    A_{21} & A_{22} & A_{23} \\
    A_{31} & A_{32} & A_{33}
  \end{pmatrix}.
\end{align*}
Let $\mathbf{0}\in\mathbb{R}^3$ be the zero vector and $\mathbf{O}_3,\mathbf{I}_3\in\mathbb{R}^{3\times3}$ be the zero and identity matrices.
Also, let $\mathbf{A}^{\mathrm{T}}$ be the transpose of $\mathbf{A}\in\mathbb{R}^{3\times3}$.
The symmetric part of $\mathbf{A}$ and the tensor product of $\mathbf{a},\mathbf{b}\in\mathbb{R}^3$ are denoted by
\begin{align*}
  \mathbf{A}_{\mathrm{S}} := \frac{\mathbf{A}+\mathbf{A}^{\mathrm{T}}}{2}, \quad \mathbf{a}\otimes\mathbf{b}:=(a_ib_j)_{i,j}, \quad \mathbf{a} = (a_1,a_2,a_3)^{\mathrm{T}}, \quad \mathbf{b} = (b_1,b_2,b_3)^{\mathrm{T}}.
\end{align*}
Moreover, we define the inner product and the Frobenius norm of matrices by
\begin{align*}
  \mathbf{A}:\mathbf{B} := \mathrm{tr}[\mathbf{A}^{\mathrm{T}}\mathbf{B}], \quad |\mathbf{A}| := \sqrt{\mathbf{A}:\mathbf{A}}, \quad \mathbf{A},\mathbf{B}\in\mathbb{R}^{3\times 3}.
\end{align*}
For $\mathbf{A},\mathbf{B},\mathbf{C}\in\mathbb{R}^{3\times3}$, we easily observe by $\mathbf{A}:\mathbf{B}=\sum_{i,j=1}^3A_{ij}B_{ij}$ that
\begin{align} \label{E:Mat_Inn}
  \mathbf{A}:\mathbf{B} = \mathbf{A}^{\mathrm{T}}:\mathbf{B}^{\mathrm{T}}, \quad [\mathbf{A}\mathbf{B}]:\mathbf{C} = \mathbf{B}:[\mathbf{A}^{\mathrm{T}}\mathbf{C}] = \mathbf{A}:[\mathbf{C}\mathbf{B}^{\mathrm{T}}].
\end{align}
When $\mathbf{u}=(u_1,u_2,u_3)^{\mathrm{T}}$ and $\bm{\psi}$ are mappings from a subset of $\mathbb{R}^3$ into $\mathbb{R}^3$, we write
\begin{align*}
  \nabla\mathbf{u} := (\partial_iu_j)_{i,j}, \quad \mathbf{D}(\mathbf{u}) := (\nabla\mathbf{u})_{\mathrm{S}}, \quad (\bm{\psi}\cdot\nabla)\mathbf{u} := (\bm{\psi}\cdot\nabla u_1,\bm{\psi}\cdot\nabla u_2,\bm{\psi}\cdot\nabla u_3)^{\mathrm{T}}.
\end{align*}
Note that the indices of $\partial_iu_j$ are reversed in some literature.
Our choice makes
\begin{align*}
  \nabla[\varphi\mathbf{u}] = (\nabla\varphi)\otimes\mathbf{u}+\varphi\nabla\mathbf{u}, \quad \nabla[\mathbf{u}\circ\Phi] = [\nabla\Phi][(\nabla\mathbf{u})\circ\Phi]
\end{align*}
for $\varphi\colon\mathbb{R}^3\to\mathbb{R}$ and $\Phi\colon\mathbb{R}^3\to\mathbb{R}^3$.
Also, we set
\begin{align*}
  \mathrm{div}\,\mathbf{A} := ([\mathrm{div}\,\mathbf{A}]_1,[\mathrm{div}\,\mathbf{A}]_2,[\mathrm{div}\,\mathbf{A}]_3)^{\mathrm{T}}, \quad [\mathrm{div}\,\mathbf{A}]_j := \textstyle\sum_{i=1}^3\partial_iA_{ij}
\end{align*}
for an $\mathbb{R}^{3\times3}$-valued function $\mathbf{A}=(A_{ij})_{i,j}$.

\subsection{Bochner spaces} \label{SS:Pre_Boch}
Let us recall some results on Bochner spaces.

For a Banach space $\mathcal{X}$, let $\mathcal{X}^\ast$ be the dual space of $\mathcal{X}$ and $\langle\cdot,\cdot\rangle_{\mathcal{X}}$ be the duality product between $\mathcal{X}^\ast$ and $\mathcal{X}$.
Also, for an interval $I$ in $\mathbb{R}$, let $C_{\mathrm{weak}}(I;\mathcal{X})$ be the space of weakly continuous functions on $I$ with values in $\mathcal{X}$.
Clearly, $C(I;\mathcal{X}) \subset C_{\mathrm{weak}}(I;\mathcal{X})$.

\begin{lemma} \label{L:BS_WeCo}
  Let $\mathcal{X}$ and $\mathcal{Y}$ be Banach spaces such that $\mathcal{X}$ is reflexive and continuously embedded into $\mathcal{Y}$.
  Then,
  \begin{align*}
    L^\infty(0,T;\mathcal{X})\cap C_{\mathrm{weak}}([0,T];\mathcal{Y})\subset C_{\mathrm{weak}}([0,T];\mathcal{X}).
  \end{align*}
\end{lemma}

\begin{proof}
  We refer to \cite[Chapter 3, Lemma 1.4]{Tem01} and \cite[Lemma II.5.9]{BoyFab13} for the proof.
\end{proof}

Let $u,v\in L^1(0,T;\mathcal{X})$.
We write $v=\partial_tu$ if
\begin{align*}
  \int_0^T\langle\varphi(t),v(t)\rangle_{\mathcal{X}}\,dt = -\int_0^T\langle\partial_t\varphi(t),u(t)\rangle_{\mathcal{X}}\,dt
\end{align*}
for all $\varphi\in C_c^1(0,T;\mathcal{X}^\ast)$.
In this case, $u\in C([0,T];\mathcal{X})$.

\begin{lemma} \label{L:BS_Ipl}
  Let $\mathcal{H}$ and $\mathcal{V}$ be Hilbert spaces such that $\mathcal{V}$ is continuously and densely embedded into $\mathcal{H}$.
  We consider $\mathcal{V}\hookrightarrow\mathcal{H}\hookrightarrow\mathcal{V}^\ast$ by
  \begin{align*}
    \langle u,v\rangle_{\mathcal{V}} := (u,v)_{\mathcal{H}}, \quad u\in\mathcal{H}, \, v\in\mathcal{V},
  \end{align*}
  where $(\cdot,\cdot)_{\mathcal{H}}$ is the inner product of $\mathcal{H}$.
  For $T>0$, we define
  \begin{align} \label{E:Def_ETV}
    \mathbb{E}_T(\mathcal{V}) := \{u\in L^2(0,T;\mathcal{V}) \mid \partial_tu\in L^2(0,T;\mathcal{V}^\ast)\}.
  \end{align}
  Then, $\mathbb{E}_T(\mathcal{V})\subset C([0,T];\mathcal{H})$.
  Moreover,
  \begin{align} \label{E:Dt_IbP}
    \frac{d}{dt}\bigl(u_1(t),u_2(t)\bigr)_{\mathcal{H}} = \langle\partial_tu_1(t),u_2(t)\rangle_{\mathcal{V}}+\langle\partial_tu_2(t),u_1(t)\rangle_{\mathcal{V}}
  \end{align}
  for all $u_1,u_2\in\mathbb{E}_T(\mathcal{V})$ and a.a. $t\in(0,T)$.
\end{lemma}

\begin{proof}
  We refer to \cite[Chapter 3, Lemma 1.2]{Tem01} and \cite[Theorems II.5.12 and II.5.13]{BoyFab13} for the proof.
\end{proof}

\subsection{Unit sphere} \label{SS:Pre_Sph}
Let $S^2$ be the unit sphere in $\mathbb{R}^3$ and $\mathbf{n}(y):=y$ be the unit outward normal vector field of $S^2$.
Also, let $\mathbf{P}:=\mathbf{I}_3-\mathbf{n}\otimes\mathbf{n}$ be the orthogonal projection onto the tangent plane of $S^2$.
Then, we easily find that
\begin{align*}
  \mathbf{P}^{\mathrm{T}} = \mathbf{P}^2 = \mathbf{P}, \quad |\mathbf{P}|^2 = \mathrm{tr}[\mathbf{P}^{\mathrm{T}}\mathbf{P}] = \mathrm{tr}[\mathbf{P}] = 2 \quad\text{on}\quad S^2.
\end{align*}
In what follows, we frequently use these relations (sometimes without mention).

We define the tangential gradient of a function $\eta$ on $S^2$ by
\begin{align*}
  \nabla_{S^2}\eta = (\underline{D}_1\eta,\underline{D}_2\eta,\underline{D}_3\eta)^{\mathrm{T}} := \mathbf{P}\nabla\tilde{\eta} \quad\text{on}\quad S^2.
\end{align*}
Here, $\tilde{\eta}$ is an extension of $\eta$ to an open neighborhood of $S^2$ in $\mathbb{R}^3$ and the value of $\nabla_{S^2}\eta$ does not depend on the choice of $\tilde{\eta}$.
Next, we define the tangential gradient matrix, the surface divergence, and the surface strain rate tensor of a vector field $\mathbf{v}$ on $S^2$ by
\begin{align*}
  \nabla_{S^2}\mathbf{v} := (\underline{D}_iv_j)_{i,j}, \quad \mathrm{div}_{S^2}\mathbf{v} := \mathrm{tr}[\nabla_{S^2}\mathbf{v}], \quad \mathbf{D}_{S^2}(\mathbf{v}) := \mathbf{P}(\nabla_{S^2}\mathbf{v})_{\mathrm{S}}\mathbf{P} \quad\text{on}\quad S^2,
\end{align*}
respectively, where $\mathbf{v}=(v_1,v_2,v_3)^{\mathrm{T}}$.
The indices of $\underline{D}_iv_j$ are reversed in some literature.
Our notation yields $\nabla_{S^2}\mathbf{v}=\mathbf{P}\nabla\tilde{\mathbf{v}}$ for any extension $\tilde{\mathbf{v}}$ of $\mathbf{v}$.
In particular, we have
\begin{align} \label{E:TGr_Nor}
  \nabla_{S^2}\mathbf{n} = \mathbf{P}, \quad \mathrm{div}_{S^2}\mathbf{n} = \mathrm{tr}[\mathbf{P}] = 2 \quad\text{on}\quad S^2
\end{align}
by taking the extension $\tilde{\mathbf{n}}(x):=x$.
Note that $H:=-\mathrm{div}_{S^2}\mathbf{n}=-2$ is the mean curvature of $S^2$.
We define the Laplace--Beltrami operator on $S^2$ by
\begin{align*}
  \Delta_{S^2}\eta := \mathrm{div}_{S^2}(\nabla_{S^2}\eta) \quad\text{on}\quad S^2, \quad \eta\colon S^2\to\mathbb{R}.
\end{align*}
Let $\mathbf{v}=(v_1,v_2,v_3)^{\mathrm{T}}$ and $\mathbf{w}$ be vector fields on $S^2$.
We set
\begin{align*}
  (\mathbf{w}\cdot\nabla_{S^2})\mathbf{v} := (\mathbf{w}\cdot\nabla_{S^2}v_1,\mathbf{w}\cdot\nabla_{S^2}v_2,\mathbf{w}\cdot\nabla_{S^2}v_3)^{\mathrm{T}} \quad\text{on}\quad S^2.
\end{align*}
When $\mathbf{v}$ and $\mathbf{w}$ are tangential (i.e. $\mathbf{v}\cdot\mathbf{n}=\mathbf{w}\cdot\mathbf{n}=0$) on $S^2$, we define
\begin{align*}
  \nabla_{\mathbf{w}}\mathbf{v} := \mathbf{P}[(\mathbf{w}\cdot\nabla_{S^2})\mathbf{v}] \quad\text{on}\quad S^2
\end{align*}
and call it the covariant derivative of $\mathbf{v}$ along $\mathbf{w}$.
We also set
\begin{align*}
  \mathrm{div}_{S^2}\mathbf{A} := ([\mathrm{div}_{S^2}\mathbf{A}]_1,[\mathrm{div}_{S^2}\mathbf{A}]_2,[\mathrm{div}_{S^2}\mathbf{A}]_3)^{\mathrm{T}}, \quad [\mathrm{div}_{S^2}\mathbf{A}]_j := \textstyle\sum_{i=1}^3\underline{D}_iA_{ij}
\end{align*}
on $S^2$ for an $\mathbb{R}^{3\times3}$-valued function $\mathbf{A}=(A_{ij})_{i,j}$ on $S^2$.
Let us give basic formulas.

\begin{lemma} \label{L:Vec_Tan}
  For $\mathbf{v}\colon S^2\to\mathbb{R}^3$, let $\mathbf{v}_\tau:=\mathbf{P}\mathbf{v}$ be its tangential component.
  Then,
  \begin{align}
    \nabla_{S^2}\mathbf{v}_\tau &= (\nabla_{S^2}\mathbf{v})\mathbf{P}-\mathbf{v}_\tau\otimes\mathbf{n}-(\mathbf{v}\cdot\mathbf{n})\mathbf{P} \quad\text{on}\quad S^2, \label{E:Vec_Tan} \\
    \mathbf{D}_{S^2}(\mathbf{v}_\tau) &= \mathbf{D}_{S^2}(\mathbf{v})-(\mathbf{v}\cdot\mathbf{n})\mathbf{P} \quad\text{on}\quad S^2. \label{E:VT_str}
  \end{align}
\end{lemma}

\begin{proof}
  Since $\mathbf{v}_\tau=\mathbf{P}\mathbf{v}=\mathbf{v}-(\mathbf{v}\cdot\mathbf{n})\mathbf{n}$ on $S^2$,
  \begin{align*}
    \nabla_{S^2}\mathbf{v}_\tau = \nabla_{S^2}\mathbf{v}-[\nabla_{S^2}(\mathbf{v}\cdot\mathbf{n})]\otimes\mathbf{n}-(\mathbf{v}\cdot\mathbf{n})\nabla_{S^2}\mathbf{n}
  \end{align*}
  on $S^2$.
  Moreover, we see by \eqref{E:TGr_Nor} that
  \begin{align*}
    \nabla_{S^2}(\mathbf{v}\cdot\mathbf{n}) = (\nabla_{S^2}\mathbf{v})\mathbf{n}+(\nabla_{S^2}\mathbf{n})\mathbf{v} = (\nabla_{S^2}\mathbf{v})\mathbf{n}+\mathbf{P}\mathbf{v} = (\nabla_{S^2}\mathbf{v})\mathbf{n}+\mathbf{v}_\tau
  \end{align*}
  on $S^2$.
  Combining the above relations and using \eqref{E:TGr_Nor} and
  \begin{align*}
    \nabla_{S^2}\mathbf{v}-[(\nabla_{S^2}\mathbf{v})\mathbf{n}]\otimes\mathbf{n} = \nabla_{S^2}\mathbf{v}-(\nabla_{S^2}\mathbf{v})(\mathbf{n}\otimes\mathbf{n}) = (\nabla_{S^2}\mathbf{v})\mathbf{P},
  \end{align*}
  we obtain \eqref{E:Vec_Tan}.
  Also, applying $\mathbf{P}$ to \eqref{E:Vec_Tan} from both sides, and using
  \begin{align*}
    \mathbf{P}^2 = \mathbf{P}^{\mathrm{T}} = \mathbf{P}, \quad (\mathbf{v}_\tau\otimes\mathbf{n})\mathbf{P} = \mathbf{v}_\tau\otimes[\mathbf{P^{\mathrm{T}}\mathbf{n}}] = \mathbf{v}_\tau\otimes[\mathbf{P}\mathbf{n}] = \mathbf{v}_\tau\otimes\mathbf{0} = \mathbf{O}_3,
  \end{align*}
  we get $\mathbf{P}(\nabla_{S^2}\mathbf{v}_\tau)\mathbf{P}=\mathbf{P}(\nabla_{S^2}\mathbf{v})\mathbf{P}-(\mathbf{v}\cdot\mathbf{n})\mathbf{P}$, whose symmetric part gives \eqref{E:VT_str}.
\end{proof}

\begin{lemma} \label{L:Gauss}
  For tangential vector fields $\mathbf{v}$ and $\mathbf{w}$ on $S^2$, we have
  \begin{align}
    (\mathbf{w}\cdot\nabla_{S^2})\mathbf{v} &= \nabla_{\mathbf{w}}\mathbf{v}-(\mathbf{v}\cdot\mathbf{w})\mathbf{n} \quad\text{on}\quad S^2, \label{E:Gauss} \\
    \nabla_{S^2}\mathbf{v} &= \mathbf{P}(\nabla_{S^2}\mathbf{v})\mathbf{P}-\mathbf{v}\otimes\mathbf{n} \quad\text{on}\quad S^2. \label{E:TGr_Dec}
  \end{align}
  Note that the first formula is the Gauss formula.
\end{lemma}

\begin{proof}
  Applying $\nabla_{S^2}$ to $\mathbf{v}\cdot\mathbf{n}=0$ on $S^2$ and using \eqref{E:TGr_Nor}, we have
  \begin{align*}
    (\nabla_{S^2}\mathbf{v})\mathbf{n}+(\nabla_{S^2}\mathbf{n})\mathbf{v} = \mathbf{0}, \quad (\nabla_{S^2}\mathbf{v})\mathbf{n} = -(\nabla_{S^2}\mathbf{n})\mathbf{v} = -\mathbf{P}\mathbf{v} = -\mathbf{v}.
  \end{align*}
  Hence, $[(\mathbf{w}\cdot\nabla_{S^2})\mathbf{v}]\cdot\mathbf{n}=\mathbf{w}\cdot[(\nabla_{S^2}\mathbf{v})\mathbf{n}]=-\mathbf{w}\cdot\mathbf{v}$ and \eqref{E:Gauss} follows.
  Also, by
  \begin{align*}
    \nabla_{S^2}\mathbf{v} = (\nabla_{S^2}\mathbf{v})\mathbf{P}+(\nabla_{S^2}\mathbf{v})(\mathbf{n}\otimes\mathbf{n}) = (\nabla_{S^2}\mathbf{v})\mathbf{P}+[(\nabla_{S^2}\mathbf{v})\mathbf{n}]\otimes\mathbf{n},
  \end{align*}
  and by $\nabla_{S^2}\mathbf{v}=\mathbf{P}\nabla_{S^2}\mathbf{v}$ and $(\nabla_{S^2}\mathbf{v})\mathbf{n}=-\mathbf{v}$, we have \eqref{E:TGr_Dec}.
\end{proof}

\begin{lemma} \label{L:Const}
  For a function $\eta$ on $S^2$ and $x\in\mathbb{R}^3\setminus\{0\}$, let $\bar{\eta}(x):=\eta(x/|x|)$ be the constant extension of $\eta$ in the radial direction.
  Then,
  \begin{align} \label{E:Const}
    \nabla\bar{\eta}(x) = \frac{1}{|x|}\,\overline{\nabla_{S^2}\eta}(x), \quad x\in\mathbb{R}^3\setminus\{0\}.
  \end{align}
\end{lemma}

\begin{proof}
  Since $\mathbf{P}(y)=\mathbf{I}_3-y\otimes y$ for $y\in S^2$, we have
  \begin{align} \label{Pf_Co:Gr_x}
    \nabla\left(\frac{x}{|x|}\right) = \frac{1}{|x|}\left(\mathbf{I}_3-\frac{x}{|x|}\otimes\frac{x}{|x|}\right) = \frac{1}{|x|}\mathbf{P}\left(\frac{x}{|x|}\right) = \frac{1}{|x|}\overline{\mathbf{P}}(x)
  \end{align}
  for $x\in\mathbb{R}^3\setminus\{0\}$.
  We differentiate $\bar{\eta}(x)=\bar{\eta}(x/|x|)$ in $x$.
  Then, we see that
  \begin{align*}
    \nabla\bar{\eta}(x) = \frac{1}{|x|}\,\mathbf{P}\left(\frac{x}{|x|}\right)[\nabla\bar{\eta}]\left(\frac{x}{|x|}\right) = \frac{1}{|x|}\nabla_{S^2}\eta\left(\frac{x}{|x|}\right) = \frac{1}{|x|}\overline{\nabla_{S^2}\eta}(x)
  \end{align*}
  by \eqref{Pf_Co:Gr_x} and $\mathbf{P}(y)\nabla\bar{\eta}(y)=\nabla_{S^2}\eta(y)$ for $y=x/|x|\in S^2$.
  Hence, \eqref{E:Const} holds.
\end{proof}

Let $\mathcal{H}^2$ be the two-dimensional Hausdorff measure.
For functions $\eta$ and $\zeta$ on $S^2$, it is known that the integration by parts formula
\begin{align} \label{E:IbP_S2}
  \int_{S^2}(\underline{D}_i\eta)\zeta\,d\mathcal{H}^2 = -\int_{S^2}\eta(\underline{D}_i\zeta-2\zeta n_i)\,d\mathcal{H}^2, \quad i=1,2,3
\end{align}
holds (see \cite[Lemma 16.1]{GilTru01} and \cite[Theorem 2.10]{DziEll13_AN}), where $-2=H$ is the mean curvature of $S^2$ and $\mathbf{n}=(n_1,n_2,n_3)^{\mathrm{T}}$.
For $\eta\in L^2(S^2)$, we define $\underline{D}_i\eta\in L^2(S^2)$ as a function satisfying \eqref{E:IbP_S2} for all $\zeta\in C^1(S^2)$ (if exists), and set the Sobolev space
\begin{align*}
  H^1(S^2) := \{\eta\in L^2(S^2) \mid \nabla_{S^2}\eta = (\underline{D}_1\eta,\underline{D}_2\eta,\underline{D}_3\eta)^{\mathrm{T}}\in L^2(S^2)^3\}
\end{align*}
with inner product $(\eta_1,\eta_2)_{H^1(S^2)}:=(\eta_1,\eta_2)_{L^2(S^2)}+(\nabla_{S^2}\eta_1,\nabla_{S^2}\eta_2)_{L^2(S^2)}$.

We define the spaces of tangential vector fields on $S^2$ of class $\mathcal{X}$ by
\begin{align*}
  \mathcal{X}_\tau(S^2) := \{\mathbf{v}\in\mathcal{X}(S^2)^3 \mid \text{$\mathbf{v}\cdot\mathbf{n}=0$ on $S^2$}\}.
\end{align*}
Note that $\mathcal{X}_\tau(S^2)$ is a closed subspace of $\mathcal{X}(S^2)^3$.
Next, we define the spaces of tangential and solenoidal vector fields on $S^2$ by
\begin{align*}
  \mathcal{H}_0 := \{\mathbf{v}\in L_\tau^2(S^2) \mid \text{$\mathrm{div}_{S^2}\mathbf{v}=0$ on $S^2$}\}, \quad \mathcal{V}_0 := \mathcal{H}_0\cap H_\tau^1(S^2).
\end{align*}
In the above, we consider $\mathrm{div}_{S^2}\mathbf{v}$ for $\mathbf{v}\in L_\tau^2(S^2)$ as an element of $[H^1(S^2)]^\ast$ by
\begin{align*}
  \langle\mathrm{div}_{S^2}\mathbf{v},\eta\rangle_{H^1(S^2)} := -(\mathbf{v},\nabla_{S^2}\eta)_{L^2(S^2)}, \quad \eta\in H^1(S^2).
\end{align*}
We easily see that $\mathcal{H}_0$ and $\mathcal{V}_0$ are closed subspaces of $L_\tau^2(S^2)$ and $H_\tau^1(S^2)$, respectively.
Let $\mathbb{L}_0$ be the Helmholtz--Leray projection on $L_\tau^2(S^2)$, i.e. the orthogonal projection from $L_\tau^2(S^2)$ onto $\mathcal{H}_0$.
As in the planar case, $\mathbb{L}_0\mathbf{v}=\mathbf{v}-\nabla_{S^2}\eta$ for $\mathbf{v}\in L_\tau^2(S^2)$, where $\eta$ is a weak solution to Poisson's equation
\begin{align} \label{E:Pois_S2}
  \Delta_{S^2}\eta = \mathrm{div}_{S^2}\mathbf{v} \quad\text{on}\quad S^2, \quad \eta\in H^1(S^2), \quad \int_{S^2}\eta\,d\mathcal{H}^2 = 0,
\end{align}
which is uniquely solvable by Poincare's inequality and the Lax--Milgram theorem.

Let us state basic results related to functions on $S^2$.

\begin{lemma} \label{L:Den_S2}
  Let $\mathcal{X}=L^2,H^1$.
  The space $C_\tau^\infty(S^2)$ is dense in $\mathcal{X}_\tau(S^2)$.
\end{lemma}

\begin{proof}
  We first note that $C^\infty(S^2)$ is dense in $\mathcal{X}(S^2)$.
  This is shown by standard localization and mollification arguments, so we omit details here.

  Let $\mathbf{v}\in\mathcal{X}_\tau(S^2)\subset \mathcal{X}(S^2)^3$.
  We can take $\mathbf{w}_k\in C^\infty(S^2)^3$ such that $\mathbf{w}_k\to\mathbf{v}$ strongly in $\mathcal{X}(S^2)^3$ as $k\to\infty$.
  Then, $\mathbf{v}_k:=\mathbf{P}\mathbf{w}_k \in C_\tau^\infty(S^2)$ by the smoothness of $\mathbf{P}$.
  Moreover,
  \begin{align*}
    \|\mathbf{v}-\mathbf{v}_k\|_{\mathcal{X}(S^2)} = \|\mathbf{P}(\mathbf{v}-\mathbf{w}_k)\|_{\mathcal{X}(S^2)} \leq c\|\mathbf{v}-\mathbf{w}_k\|_{\mathcal{X}(S^2)} \to 0
  \end{align*}
  as $k\to\infty$, since $\mathbf{P}\mathbf{v}=\mathbf{v}$ by $\mathbf{v}\cdot\mathbf{n}=0$ on $S^2$.
  Hence, $C_\tau^\infty(S^2)$ is dense in $\mathcal{X}_\tau(S^2)$.
\end{proof}

\begin{lemma} \label{L:Lad_S2}
  For $\eta\in H^1(S^2)$, we have Ladyzhenskaya's inequality
  \begin{align} \label{E:Lad_S2}
    \|\eta\|_{L^4(S^2)} \leq c\|\eta\|_{L^2(S^2)}^{1/2}\|\eta\|_{H^1(S^2)}^{1/2}.
  \end{align}
\end{lemma}

\begin{proof}
  We refer to \cite[Lemma 4.1]{Miu21_02} for the proof.
\end{proof}

\begin{lemma} \label{L:Korn_S2}
  For $\mathbf{v}\in H_\tau^1(S^2)$, we have Korn's inequality
  \begin{align} \label{E:Korn_S2}
    \|\nabla_{S^2}\mathbf{v}\|_{L^2(S^2)} \leq c\Bigl(\|\mathbf{D}_{S^2}(\mathbf{v})\|_{L^2(S^2)}+\|\mathbf{v}\|_{L^2(S^2)}\Bigr).
  \end{align}
\end{lemma}

\begin{proof}
  We refer to \cite[Lemma 4.1]{JaOlRe18} and \cite[Lemma 4.2]{Miu20_03} for the proof.
\end{proof}

\begin{lemma} \label{L:HLH1_S2}
  If $\mathbf{v}\in H_\tau^1(S^2)$, then $\mathbb{L}_0\mathbf{v}\in\mathcal{V}_0$ and
  \begin{align} \label{E:HLH1_S2}
    \|\mathbf{v}-\mathbb{L}_0\mathbf{v}\|_{H^1(S^2)} \leq c\|\mathrm{div}_{S^2}\mathbf{v}\|_{L^2(S^2)}, \quad \|\mathbb{L}_0\mathbf{v}\|_{H^1(S^2)} \leq c\|\mathbf{v}\|_{H^1(S^2)}.
  \end{align}
\end{lemma}

\begin{proof}
  Since $\mathrm{div}_{S^2}\mathbf{v}\in L^2(S^2)$, the weak solution $\eta$ to \eqref{E:Pois_S2} satisfies
  \begin{align*}
    \eta \in H^2(S^2), \quad \|\eta\|_{H^2(S^2)} \leq c\|\mathrm{div}_{S^2}\mathbf{v}\|_{L^2(S^2)}
  \end{align*}
  by the elliptic regularity theorem (see e.g. \cite[Theorem 3.3]{DziEll13_AN}).
  Since $\mathbb{L}_0\mathbf{v}=\mathbf{v}-\nabla_{S^2}\eta$, the above results show that $\mathbb{L}_0\mathbf{v}\in\mathcal{V}_0$ and \eqref{E:HLH1_S2} holds.
\end{proof}

\begin{lemma} \label{L:Sol_S2}
  The space $\mathcal{V}_0$ is dense in $\mathcal{H}_0$.
\end{lemma}

\begin{proof}
  Let $\mathbf{v}\in\mathcal{H}_0\subset L_\tau^2(S^2)$.
  By Lemma \ref{L:Den_S2}, we can take $\mathbf{w}_k\in C_\tau^\infty(S^2)$ such that $\mathbf{w}_k\to\mathbf{v}$ strongly in $L^2(S^2)^3$ as $k\to\infty$.
  Then, $\mathbf{v}_k:=\mathbb{L}_0\mathbf{w}_k\in\mathcal{V}_0$ by Lemma \ref{L:HLH1_S2} and
  \begin{align*}
    \|\mathbf{v}-\mathbf{v}_k\|_{L^2(S^2)} = \|\mathbb{L}_0(\mathbf{v}-\mathbf{w}_k)\|_{L^2(S^2)} \leq \|\mathbf{v}-\mathbf{w}_k\|_{L^2(S^2)} \to 0
  \end{align*}
  as $k\to\infty$, since $\mathbf{v}\in\mathcal{H}_0$ and $\mathbb{L}_0$ is the orthogonal projection from $L_\tau^2(S^2)$ onto $\mathcal{H}_0$.
  Hence, $\mathcal{V}_0$ is dense in $\mathcal{H}_0$.
\end{proof}

By Lemma \ref{L:Sol_S2}, we can apply Lemma \ref{L:BS_Ipl} with $\mathcal{H}=\mathcal{H}_0$ and $\mathcal{V}=\mathcal{V}_0$.

\subsection{Thin spherical shell} \label{SS:Pre_TSS}
For a sufficiently small $\varepsilon\in(0,1)$, let
\begin{align*}
  \Omega_\varepsilon = \{x\in\mathbb{R}^3 \mid 1<|x|<1+\varepsilon\}
\end{align*}
be the thin spherical shell in $\mathbb{R}^3$ and $\mathbf{n}_\varepsilon$ be the unit outward normal vector field of $\partial\Omega_\varepsilon$.
The following results are shown in \cite{Miu21_02,Miu22_01} for curved thin domains around general closed surfaces including the thin spherical shell.
Recall that $c$ denotes a general positive constant independent of $\varepsilon$ and $\bar{\eta}(x)=\eta(x/|x|)$ is the constant extension of a function $\eta$ on $S^2$.

\begin{lemma} \label{L:Korn_Thin}
  For $\mathbf{u}\in H^1(\Omega_\varepsilon)^3$ satisfying $\mathbf{u}\cdot\mathbf{n}_\varepsilon=0$ on $\partial\Omega_\varepsilon$, we have
  \begin{align} \label{E:Korn_Thin}
    \|\nabla\mathbf{u}\|_{L^2(\Omega_\varepsilon)}^2 \leq 4\|\mathbf{D}(\mathbf{u})\|_{L^2(\Omega_\varepsilon)}^2+c\|\mathbf{u}\|_{L^2(\Omega_\varepsilon)}^2.
  \end{align}
\end{lemma}

\begin{proof}
  We refer to \cite[Lemma 5.1]{Miu22_01} for the proof.
\end{proof}

\begin{lemma} \label{L:Nor_Thin}
  For $\mathbf{u}\in H^1(\Omega_\varepsilon)^3$ satisfying $\mathbf{u}\cdot\mathbf{n}_\varepsilon=0$ on $\partial\Omega_\varepsilon$, we have
  \begin{align} \label{E:Nor_Thin}
    \|\mathbf{u}\cdot\bar{\mathbf{n}}\|_{L^2(\Omega_\varepsilon)} \leq c\varepsilon\|\mathbf{u}\|_{H^1(\Omega_\varepsilon)}.
  \end{align}
\end{lemma}

\begin{proof}
  We refer to \cite[Lemma 4.5]{Miu21_02} for the proof.
\end{proof}

\begin{lemma} \label{L:Quad_Thin}
  For $\eta\in H^1(S^2)$ and $\varphi\in H^1(\Omega_\varepsilon)$, we have
  \begin{align} \label{E:Quad_Thin}
    \|\bar{\eta}\varphi\|_{L^2(\Omega_\varepsilon)} \leq c\|\eta\|_{L^2(S^2)}^{1/2}\|\eta\|_{H^1(S^2)}^{1/2}\|\varphi\|_{L^2(\Omega_\varepsilon)}^{1/2}\|\varphi\|_{H^1(\Omega_\varepsilon)}^{1/2}.
  \end{align}
\end{lemma}

\begin{proof}
  We refer to \cite[Lemma 6.19]{Miu21_02} for the proof.
\end{proof}

We define the space of $L^2$ solenoidal vector fields on $\Omega_\varepsilon$ by
\begin{align*}
  \mathcal{H}_\varepsilon := \{\mathbf{u}\in L^2(\Omega_\varepsilon)^3 \mid \text{$\mathrm{div}\,\mathbf{u}=0$ in $\Omega_\varepsilon$, $\mathbf{u}\cdot\mathbf{n}_\varepsilon=0$ on $\partial\Omega_\varepsilon$}\}.
\end{align*}
It is known that $\mathcal{H}_\varepsilon$ is the closure in $L^2(\Omega_\varepsilon)^3$ of
\begin{align*}
  C_{c,\sigma}^\infty(\Omega_\varepsilon) := \{\mathbf{u}\in C_c^\infty(\Omega_\varepsilon)^3 \mid \text{$\mathrm{div}\,\mathbf{u}=0$ in $\Omega_\varepsilon$}\},
\end{align*}
see \cite[Chapter 1, Theorem 1.4]{Tem01} and \cite[Theorem IV.3.5]{BoyFab13}.
In particular, $\mathcal{H}_\varepsilon$ is a closed subspace of $L^2(\Omega_\varepsilon)^3$.
We also define the space of $H^1$ solenoidal vector fields
\begin{align*}
  \mathcal{V}_\varepsilon := \mathcal{H}_\varepsilon\cap H^1(\Omega_\varepsilon) = \{\mathbf{u}\in H^1(\Omega_\varepsilon)^3 \mid \text{$\mathrm{div}\,\mathbf{u}=0$ in $\Omega_\varepsilon$, $\mathbf{u}\cdot\mathbf{n}_\varepsilon=0$ on $\partial\Omega_\varepsilon$}\}.
\end{align*}
It is a closed subspace of $H^1(\Omega_\varepsilon)^3$.
Also, $\mathcal{V}_\varepsilon$ is dense in $\mathcal{H}_\varepsilon$, since $C_{c,\sigma}^\infty(\Omega_\varepsilon)\subset\mathcal{V}_\varepsilon$.
Thus, we can use Lemma \ref{L:BS_Ipl} with $\mathcal{H}=\mathcal{H}_\varepsilon$ and $\mathcal{V}=\mathcal{V}_\varepsilon$.
Note that
\begin{align*}
  \mathcal{V}_\varepsilon \neq \{\mathbf{u} \in H^1(\Omega_\varepsilon)^3 \mid \text{$\mathrm{div}\,\mathbf{u}=0$ in $\Omega_\varepsilon$, $\mathbf{u}=\mathbf{0}$ on $\partial\Omega_\varepsilon$}\},
\end{align*}
where the latter space is the closure of $C_{c,\sigma}^\infty(\Omega_\varepsilon)$ in $H^1(\Omega_\varepsilon)^3$.

\section{Definition and properties of weak solutions} \label{S:Weak}
We recall the definition and properties of weak solutions to the Navier--Stokes equations in the $L^2$-setting.
For details of the theory of weak solutions, see \cite{Soh01,Tem01,BoyFab13}.

In what follows, we sometimes suppress the time variable of functions and write
\begin{align*}
  (\varphi_1,\varphi_2)_{L^2(\Omega_\varepsilon)} = \int_{\Omega_\varepsilon}\varphi_1(x)\varphi_2(x)\,dx, \quad (\eta_1,\eta_2)_{L^2(S^2)} = \int_{S^2}\eta_1(y)\eta_2(y)\,d\mathcal{H}^2(y)
\end{align*}
whenever the right-hand integrals make sense.

\subsection{Problem on the thin spherical shell} \label{SS:We_TSS}
First, we consider the problem \eqref{E:NS_TSS} on $\Omega_\varepsilon$.
Let $\mathbf{u}_1$ and $\mathbf{u}_2$ be smooth vector fields on $\overline{\Omega_\varepsilon}$.
If $\mathrm{div}\,\mathbf{u}_1=0$ in $\Omega_\varepsilon$, then
\begin{align*}
  \Delta\mathbf{u}_1 = \Delta\mathbf{u}_1+\nabla(\mathrm{div}\,\mathbf{u}_1) = 2\,\mathrm{div}\,\mathbf{D}(\mathbf{u}_1) \quad\text{in}\quad \Omega_\varepsilon.
\end{align*}
Thus, if in addition $\mathbf{P}_\varepsilon\mathbf{D}(\mathbf{u}_1)\mathbf{n}_\varepsilon=\mathbf{0}$ and $\mathbf{u}_2\cdot\mathbf{n}_\varepsilon=0$ on $\partial\Omega_\varepsilon$, then
\begin{align*}
  \int_{\Omega_\varepsilon}(\Delta\mathbf{u}_1)\cdot\mathbf{u}_2\,dx &= 2\int_{\Omega_\varepsilon}\{\mathrm{div}\,\mathbf{D}(\mathbf{u}_1)\}\cdot\mathbf{u}_2\,dx = -2\int_{\Omega_\varepsilon}\mathbf{D}(\mathbf{u}_1):\nabla\mathbf{u}_2\,dx \\
  &= -2\int_{\Omega_\varepsilon}\mathbf{D}(\mathbf{u}_1):\mathbf{D}(\mathbf{u}_2)\,dx
\end{align*}
by integration by parts and the symmetry of $\mathbf{D}(\mathbf{u}_1)=(\nabla\mathbf{u}_1)_{\mathrm{S}}$.
Based on this observation, we define a weak solution to \eqref{E:NS_TSS} as follows.

\begin{definition} \label{D:WS_TSS}
  For given data
  \begin{align*}
    \mathbf{u}_0^\varepsilon \in \mathcal{H}_\varepsilon, \quad \mathbf{f}^\varepsilon \in L_{\mathrm{loc}}^2([0,\infty);\mathcal{V}_\varepsilon^\ast),
  \end{align*}
  we say that $\mathbf{u}^\varepsilon$ is a weak solution to \eqref{E:NS_TSS} if
  \begin{align*}
    \mathbf{u}^\varepsilon \in L_{\mathrm{loc}}^\infty([0,\infty);\mathcal{H}_\varepsilon)\cap L_{\mathrm{loc}}^2([0,\infty);\mathcal{V}_\varepsilon)
  \end{align*}
  and the following equality holds for all $\bm{\psi}\in C_c^1([0,\infty);\mathcal{V}_\varepsilon)$:
  \begin{multline} \label{E:WF_TSS}
    -\int_0^\infty(\mathbf{u}^\varepsilon,\partial_t\bm{\psi})_{L^2(\Omega_\varepsilon)}\,dt+2\nu\int_0^\infty\bigl(\mathbf{D}(\mathbf{u}^\varepsilon),\mathbf{D}(\bm{\psi})\bigr)_{L^2(\Omega_\varepsilon)}\,dt \\
    +\int_0^\infty\bigl((\mathbf{u}^\varepsilon\cdot\nabla)\mathbf{u}^\varepsilon,\bm{\psi}\bigr)_{L^2(\Omega_\varepsilon)}\,dt = \bigl(\mathbf{u}_0^\varepsilon,\bm{\psi}(0)\bigr)_{L^2(\Omega_\varepsilon)}+\int_0^\infty\langle\mathbf{f}^\varepsilon,\bm{\psi}\rangle_{\mathcal{V}_\varepsilon}\,dt.
  \end{multline}
\end{definition}

Let us state a few properties of a weak solution to \eqref{E:NS_TSS}.
In what follows, we write $c_\varepsilon$ for a general positive constant which may depend on $\varepsilon$.
We do not take the limit $\varepsilon\to0$ in this subsection, so the dependence of each $c_\varepsilon$ on $\varepsilon$ does not matter.

\begin{lemma} \label{L:Tri_TSS}
  For $\mathbf{u},\mathbf{z},\bm{\psi}\in H^1(\Omega_\varepsilon)^3$, we have
  \begin{align} \label{E:TrTS_Bdd}
    \Bigl|\bigl((\mathbf{u}\cdot\nabla)\mathbf{z},\bm{\psi}\bigr)_{L^2(\Omega_\varepsilon)}\Bigr| \leq c_\varepsilon\|\mathbf{u}\|_{L^2(\Omega_\varepsilon)}^{1/4}\|\mathbf{u}\|_{H^1(\Omega_\varepsilon)}^{3/4}\|\mathbf{z}\|_{H^1(\Omega_\varepsilon)}\|\bm{\psi}\|_{L^2(\Omega_\varepsilon)}^{1/4}\|\bm{\psi}\|_{H^1(\Omega_\varepsilon)}^{3/4}.
  \end{align}
  Moreover, if $\mathbf{u}\in\mathcal{V}_\varepsilon$, then
  \begin{align} \label{E:TrTS_AnS}
    \bigl((\mathbf{u}\cdot\nabla)\mathbf{z},\bm{\psi}\bigr)_{L^2(\Omega_\varepsilon)} = -\bigl(\mathbf{z},(\mathbf{u}\cdot\nabla)\bm{\psi}\bigr)_{L^2(\Omega_\varepsilon)}, \quad \bigl((\mathbf{u}\cdot\nabla)\mathbf{z},\mathbf{z}\bigr)_{L^2(\Omega_\varepsilon)} = 0.
  \end{align}
\end{lemma}

\begin{proof}
  We see by H\"{o}lder's inequality that
  \begin{align*}
    \Bigl|\bigl((\mathbf{u}\cdot\nabla)\mathbf{z},\bm{\psi}\bigr)_{L^2(\Omega_\varepsilon)}\Bigr| \leq \|\mathbf{u}\|_{L^4(\Omega_\varepsilon)}\|\nabla\mathbf{z}\|_{L^2(\Omega_\varepsilon)}\|\bm{\psi}\|_{L^4(\Omega_\varepsilon)}.
  \end{align*}
  Moreover, we have the interpolation inequality (see \cite[Theorem 5.8]{AdaFou03})
  \begin{align*}
    \|\varphi\|_{L^4(\Omega_\varepsilon)} \leq c_\varepsilon\|\varphi\|_{L^2(\Omega_\varepsilon)}^{1/4}\|\varphi\|_{H^1(\Omega_\varepsilon)}^{3/4}, \quad \varphi\in H^1(\Omega_\varepsilon).
  \end{align*}
  Combining these inequalities, we obtain \eqref{E:TrTS_Bdd}.

  Next, we show \eqref{E:TrTS_AnS}.
  Since $\Omega_\varepsilon$ is a bounded smooth domain, $C^\infty(\overline{\Omega_\varepsilon})$ is dense in $H^1(\Omega_\varepsilon)$ (see \cite[Theorem 3.22]{AdaFou03}).
  By this fact, \eqref{E:TrTS_Bdd}, and a density argument, we may assume that $\mathbf{z}$ and $\bm{\psi}$ are in $C^\infty(\overline{\Omega_\varepsilon})^3$.
  Then, we can carry out integration by parts to get
  \begin{align*}
    \int_{\Omega_\varepsilon}\{(\mathbf{u}\cdot\nabla)\mathbf{z}\}\cdot\bm{\psi}\,dx = \int_{\partial\Omega_\varepsilon}(\mathbf{u}\cdot\mathbf{n}_\varepsilon)(\mathbf{z}\cdot\bm{\psi})\,d\mathcal{H}^2-\int_{\Omega_\varepsilon}\mathbf{z}\cdot\{(\mathrm{div}\,\mathbf{u})\bm{\psi}+(\mathbf{u}\cdot\nabla)\bm{\psi}\}\,dx.
  \end{align*}
  Since $\mathrm{div}\,\mathbf{u}=0$ in $\Omega_\varepsilon$ and $\mathbf{u}\cdot\mathbf{n}_\varepsilon=0$ on $\partial\Omega_\varepsilon$ by $\mathbf{u}\in\mathcal{V}_\varepsilon$, the above equality gives the first relation of \eqref{E:TrTS_AnS}.
  The second relation of \eqref{E:TrTS_AnS} follows from the first one with $\bm{\psi}=\mathbf{z}$.
\end{proof}

\begin{proposition} \label{P:TSS_WeCo}
  For given $\mathbf{u}_0^\varepsilon\in\mathcal{H}_\varepsilon$ and $\mathbf{f}^\varepsilon\in L_{\mathrm{loc}}^2([0,\infty);\mathcal{V}_\varepsilon^\ast)$, let $\mathbf{u}^\varepsilon$ be a weak solution to \eqref{E:NS_TSS}.
  Then, $\mathbf{u}^\varepsilon\in C_{\mathrm{weak}}([0,T];\mathcal{H}_\varepsilon)$ for all $T>0$.
\end{proposition}

\begin{proof}
  Let $\bm{\psi}\in C_c^1(0,T;\mathcal{V}_\varepsilon)$ in \eqref{E:WF_TSS}.
  Since
  \begin{align*}
    \Bigl|\bigl((\mathbf{u}^\varepsilon\cdot\nabla)\mathbf{u}^\varepsilon,\bm{\psi}\bigr)_{L^2(\Omega_\varepsilon)}\Bigr| = \Bigl|\bigl(\mathbf{u}^\varepsilon,(\mathbf{u}^\varepsilon\cdot\nabla)\bm{\psi}\bigr)_{L^2(\Omega_\varepsilon)}\Bigr| \leq c_\varepsilon\|\mathbf{u}^\varepsilon\|_{L^2(\Omega_\varepsilon)}^{1/2}\|\mathbf{u}^\varepsilon\|_{H^1(\Omega_\varepsilon)}^{3/2}\|\bm{\psi}\|_{H^1(\Omega_\varepsilon)}
  \end{align*}
  on $(0,T)$ by $\mathbf{u}^\varepsilon\in L^2(0,T;\mathcal{V}_\varepsilon)$ and Lemma \ref{L:Tri_TSS}, we see that
  \begin{align*}
    \left|\int_0^T\bigl((\mathbf{u}^\varepsilon\cdot\nabla)\mathbf{u}^\varepsilon,\bm{\psi}\bigr)_{L^2(\Omega_\varepsilon)}\,dt\right| \leq c_\varepsilon K_1^\varepsilon\|\bm{\psi}\|_{L^4(0,T;H^1(\Omega_\varepsilon))}
  \end{align*}
  by H\"{o}lder's inequality, where
  \begin{align*}
    K_1^\varepsilon := \|\mathbf{u}^\varepsilon\|_{L^\infty(0,T;L^2(\Omega_\varepsilon))}^{1/2}\|\mathbf{u}^\varepsilon\|_{L^2(0,T;H^1(\Omega_\varepsilon))}^{3/2}.
  \end{align*}
  We apply this inequality and H\"{o}lder's inequality to \eqref{E:WF_TSS} and use
  \begin{align*}
    \|\bm{\psi}\|_{L^2(0,T;H^1(\Omega_\varepsilon))} \leq T^{1/4}\|\bm{\psi}\|_{L^4(0,T;H^1(\Omega_\varepsilon))}
  \end{align*}
  by H\"{o}lder's inequality.
  Then, noting that $\bm{\psi}(0)=\mathbf{0}$, we find that
  \begin{align*}
    \left|\int_0^\infty(\mathbf{u}^\varepsilon,\partial_t\bm{\psi})_{L^2(\Omega_\varepsilon)}\,dt\right| \leq c_{\varepsilon,\nu}(K_1^\varepsilon+T^{1/4}K_2^\varepsilon)\|\bm{\psi}\|_{L^4(0,T;H^1(\Omega_\varepsilon))}
  \end{align*}
  for $\bm{\psi}\in C_c^1(0,T;\mathcal{V}_\varepsilon)$, where $c_{\varepsilon,\nu}>0$ is a constant depending on $\varepsilon$ and $\nu$, and
  \begin{align*}
    K_2^\varepsilon := \|\mathbf{D}(\mathbf{u}^\varepsilon)\|_{L^2(0,T;L^2(\Omega_\varepsilon))}+\|\mathbf{f}^\varepsilon\|_{L^2(0,T;\mathcal{V}_\varepsilon^\ast)}.
  \end{align*}
  Since $C_c^1(0,T;\mathcal{V}_\varepsilon)$ is dense in $L^4(0,T;\mathcal{V}_\varepsilon)$, it follows that
  \begin{align*}
    \partial_t\mathbf{u}^\varepsilon \in [L^4(0,T;\mathcal{V}_\varepsilon)]^\ast = L^{4/3}(0,T;\mathcal{V}_\varepsilon^\ast).
  \end{align*}
  By this result and $\mathbf{u}^\varepsilon\in L^2(0,T;\mathcal{V}_\varepsilon)\subset L^2(0,T;\mathcal{V}_\varepsilon^\ast)$, we get
  \begin{align*}
    \mathbf{u}^\varepsilon \in C([0,T];\mathcal{V}_\varepsilon^\ast) \subset C_{\mathrm{weak}}([0,T];\mathcal{V}_\varepsilon^\ast).
  \end{align*}
  Since $\mathbf{u}^\varepsilon$ is also in $L^\infty(0,T;\mathcal{H}_\varepsilon)$, Lemma \ref{L:BS_WeCo} yields $\mathbf{u}^\varepsilon\in C_{\mathrm{weak}}([0,T];\mathcal{H}_\varepsilon)$.
\end{proof}

The space $\mathcal{V}_\varepsilon$ can be seen as a subspace of the separable space $L^2(\Omega_\varepsilon)^{12}$ by
\begin{align*}
  \{(\mathbf{u},\nabla\mathbf{u}) \mid \mathbf{u}\in\mathcal{V}_\varepsilon\} \subset L^2(\Omega_\varepsilon)^{12}.
\end{align*}
Thus, $\mathcal{V}_\varepsilon$ is separable and we can take a countable basis $\{\mathbf{u}_k\}_{k=1}^\infty$ of $\mathcal{V}_\varepsilon$.
We may assume that $\{\mathbf{u}_k\}_{k=1}^\infty$ is an orthonormal basis of $\mathcal{H}_\varepsilon$ by applying the Gram--Schmidt orthonormalization to $\{\mathbf{u}_k\}_{k=1}^\infty$ and denoting the resulting elements by $\{\mathbf{u}_k\}_{k=1}^\infty$ again.
Using this basis and the inequality \eqref{E:Korn_Thin}, we can carry out the Galerkin method to get the next existence result as in the case of the no-slip boundary condition (but the uniqueness is open).

\begin{proposition} \label{P:TSS_Exist}
  For all $\mathbf{u}_0^\varepsilon\in\mathcal{H}_\varepsilon$ and $\mathbf{f}^\varepsilon\in L_{\mathrm{loc}}^2([0,\infty);\mathcal{V}_\varepsilon^\ast)$, there exists at least one weak solution $\mathbf{u}^\varepsilon$ to \eqref{E:NS_TSS} that satisfies the energy inequality \eqref{E:TSS_Ener}.
\end{proposition}

\subsection{Problem on the unit sphere} \label{SS:We_Sph}
Next, we consider the problem \eqref{E:NS_S2} on $S^2$.
Let $\mathbf{v}_1$ and $\mathbf{v}_2$ be smooth tangential vector fields on $S^2$.
Then, carrying out integration by parts \eqref{E:IbP_S2} componentwisely, we see that
\begin{align*}
  \int_{S^2}\Bigl(\mathbf{P}\,\mathrm{div}_{S^2}[\mathbf{D}_{S^2}(\mathbf{v}_1)]\Bigr)\cdot\mathbf{v}_2\,d\mathcal{H}^2 &= \int_{S^2}\Bigl(\mathrm{div}_{S^2}[\mathbf{D}_{S^2}(\mathbf{v}_1)]\Bigr)\cdot\mathbf{v}_2\,d\mathcal{H}^2 \\
  &= -\int_{S^2}\mathbf{D}_{S^2}(\mathbf{v}_1):[\nabla_{S^2}\mathbf{v}_2-2\mathbf{n}\otimes\mathbf{v}_2]\,d\mathcal{H}^2,
\end{align*}
where $\mathbf{D}_{S^2}(\mathbf{v}_1)=\mathbf{P}(\nabla_{S^2}\mathbf{v}_1)_{\mathrm{S}}\mathbf{P}$.
Moreover, since $\mathbf{P}^{\mathrm{T}}=\mathbf{P}^2=\mathbf{P}$ and
\begin{align*}
  \mathbf{D}_{S^2}(\mathbf{v}_1) = \mathbf{P}\mathbf{D}_{S^2}(\mathbf{v}_1)\mathbf{P}, \quad \mathbf{P}(\mathbf{n}\otimes\mathbf{v}_2) = (\mathbf{P}\mathbf{n})\otimes\mathbf{v}_2 = \mathbf{0}\otimes\mathbf{v}_2 = \mathbf{O}_3,
\end{align*}
and since $\mathbf{D}_{S^2}(\mathbf{v}_1)$ is symmetric, we have
\begin{align*}
  \mathbf{D}_{S^2}(\mathbf{v}_1):[\nabla_{S^2}\mathbf{v}_2-2\mathbf{n}\otimes\mathbf{v}_2] &= \mathbf{D}_{S^2}(\mathbf{v}_1):\mathbf{P}[\nabla_{S^2}\mathbf{v}_2-2\mathbf{n}\otimes\mathbf{v}_2]\mathbf{P} \\
  &= \mathbf{D}_{S^2}(\mathbf{v}_1):\mathbf{P}(\nabla_{S^2}\mathbf{v}_2)\mathbf{P} \\
  &= \mathbf{D}_{S^2}(\mathbf{v}_1):\mathbf{D}_{S^2}(\mathbf{v}_2)
\end{align*}
on $S^2$.
Hence, we obtain
\begin{align*}
  \int_{S^2}\Bigl(\mathbf{P}\,\mathrm{div}_{S^2}[\mathbf{D}_{S^2}(\mathbf{v}_1)]\Bigr)\cdot\mathbf{v}_2\,d\mathcal{H}^2 = -\int_{S^2}\mathbf{D}_{S^2}(\mathbf{v}_1):\mathbf{D}_{S^2}(\mathbf{v}_2)\,d\mathcal{H}^2.
\end{align*}
Using this formula, we give the definition of a weak solution to \eqref{E:NS_S2} as follows.

\begin{definition} \label{D:WS_S2}
  For given data
  \begin{align*}
    \mathbf{v}_0 \in \mathcal{H}_0, \quad \mathbf{f} \in L_{\mathrm{loc}}^2([0,\infty);\mathcal{V}_0^\ast),
  \end{align*}
  we say that $\mathbf{v}$ is a weak solution to \eqref{E:NS_S2} if
  \begin{align*}
    \mathbf{v} \in L_{\mathrm{loc}}^\infty([0,\infty);\mathcal{H}_0)\cap L_{\mathrm{loc}}^2([0,\infty);\mathcal{V}_0)
  \end{align*}
  and the following equality holds for all $\bm{\zeta}\in C_c^1([0,\infty);\mathcal{V}_0)$:
  \begin{multline} \label{E:WF_S2}
    -\int_0^\infty(\mathbf{v},\partial_t\bm{\zeta})_{L^2(S^2)}\,dt+2\nu\int_0^\infty\bigl(\mathbf{D}_{S^2}(\mathbf{v}),\mathbf{D}_{S^2}(\bm{\zeta})\bigr)_{L^2(S^2)}\,dt \\
    +\int_0^\infty(\nabla_{\mathbf{v}}\mathbf{v},\bm{\zeta})_{L^2(S^2)}\,dt = \bigl(\mathbf{v}_0,\bm{\zeta}(0)\bigr)_{L^2(S^2)}+\int_0^\infty\langle\mathbf{f},\bm{\zeta}\rangle_{\mathcal{V}_0}\,dt.
  \end{multline}
\end{definition}

Since $S^2$ is of dimension two, a weak solution to \eqref{E:NS_S2} has better properties than those of a weak solution to \eqref{E:NS_TSS} stated in Proposition \ref{P:TSS_WeCo}.

\begin{lemma} \label{L:Tri_S2}
  For $\mathbf{v},\mathbf{w},\bm{\zeta}\in H_\tau^1(S^2)$, we have
  \begin{align} \label{E:TrS2_Bdd}
    |(\nabla_{\mathbf{v}}\mathbf{w},\bm{\zeta})_{L^2(S^2)}| \leq c\|\mathbf{v}\|_{L^2(S^2)}^{1/2}\|\mathbf{v}\|_{H^1(S^2)}^{1/2}\|\mathbf{w}\|_{H^1(S^2)}\|\bm{\zeta}\|_{L^2(S^2)}^{1/2}\|\bm{\zeta}\|_{H^1(S^2)}^{1/2}.
  \end{align}
  Moreover, if $\mathbf{v}\in\mathcal{V}_0$, then
  \begin{align} \label{E:TrS2_AnS}
    (\nabla_{\mathbf{v}}\mathbf{w},\bm{\zeta})_{L^2(S^2)} = -(\mathbf{w},\nabla_{\mathbf{v}}\bm{\zeta})_{L^2(S^2)}, \quad (\nabla_{\mathbf{v}}\mathbf{w},\mathbf{w})_{L^2(S^2)} = 0.
  \end{align}
\end{lemma}

\begin{proof}
  Since $\nabla_{\mathbf{v}}\mathbf{w}=\mathbf{P}[(\mathbf{v}\cdot\nabla_{S^2})\mathbf{w}]$ on $S^2$, we see by H\"{o}lder's inequality that
  \begin{align*}
    |(\nabla_{\mathbf{v}}\mathbf{w},\bm{\zeta})_{L^2(S^2)}| \leq c\int_{S^2}|\mathbf{v}|\,|\nabla_{S^2}\mathbf{w}|\,|\bm{\zeta}|\,d\mathcal{H}^2 \leq c\|\mathbf{v}\|_{L^4(S^2)}\|\nabla_{S^2}\mathbf{w}\|_{L^2(S^2)}\|\bm{\zeta}\|_{L^4(S^2)}.
  \end{align*}
  Applying \eqref{E:Lad_S2} to the right-hand side, we obtain \eqref{E:TrS2_Bdd}.

  Next, we prove \eqref{E:TrS2_AnS}.
  By Lemma \ref{L:Den_S2}, the estimate \eqref{E:TrS2_Bdd}, and a density argument, it is sufficient to consider the case $\mathbf{w},\bm{\zeta}\in C_\tau^\infty(S^2)$.
  Let $\eta:=\mathbf{w}\cdot\bm{\zeta}$ on $S^2$.
  Then,
  \begin{align} \label{Pf_TS2:Dr}
    \mathbf{v}\cdot\nabla_{S^2}\eta = [(\mathbf{v}\cdot\nabla_{S^2})\mathbf{w}]\cdot\bm{\zeta}+\mathbf{w}\cdot[(\mathbf{v}\cdot\nabla_{S^2})\bm{\zeta}] = \nabla_{\mathbf{v}}\mathbf{w}\cdot\bm{\zeta}+\mathbf{w}\cdot\nabla_{\mathbf{v}}\bm{\zeta}
  \end{align}
  on $S^2$, since $\mathbf{w}$ are $\bm{\zeta}$ are tangential on $S^2$.
  Moreover, since $\eta\in C^\infty(S^2)$ and $\mathbf{v}\in\mathcal{V}_0$, we can carry out integration by parts \eqref{E:IbP_S2} to get
  \begin{align*}
     \int_{S^2}\mathbf{v}\cdot\nabla_{S^2}\eta\,d\mathcal{H}^2 = -\int_{S^2}(\mathrm{div}_{S^2}\mathbf{v}-2\mathbf{v}\cdot\mathbf{n})\eta\,d\mathcal{H}^2 = 0.
  \end{align*}
  Thus, integrating \eqref{Pf_TS2:Dr} over $S^2$, we obtain the first equality of \eqref{E:TrS2_AnS}.
  Also, setting $\bm{\zeta}=\mathbf{w}$ in the first equality, we get the second equality of \eqref{E:TrS2_AnS}.
\end{proof}

\begin{proposition} \label{P:S2_Cont}
  For given $\mathbf{v}_0\in\mathcal{H}_0$ and $\mathbf{f}\in L_{\mathrm{loc}}^2([0,\infty);\mathcal{V}_0^\ast)$, let $\mathbf{v}$ be a weak solution to \eqref{E:NS_S2}.
  Then, for all $T>0$, we have
  \begin{align*}
    \mathbf{v}\in \mathbb{E}_T(\mathcal{V}_0) \subset C([0,T];\mathcal{H}_0),
  \end{align*}
  where $\mathbb{E}_T(\mathcal{V}_0)$ is the function space given by \eqref{E:Def_ETV}, and $\mathbf{v}$ satisfies
  \begin{multline} \label{E:WFS2_Ano}
    \int_0^T\langle\partial_t\mathbf{v},\bm{\zeta}\rangle_{\mathcal{V}_0}\,dt+2\nu\int_0^T\bigl(\mathbf{D}_{S^2}(\mathbf{v}),\mathbf{D}_{S^2}(\bm{\zeta})\bigr)_{L^2(S^2)}\,dt \\
    +\int_0^T(\nabla_{\mathbf{v}}\mathbf{v},\bm{\zeta})_{L^2(S^2)}\,dt = \int_0^T\langle\mathbf{f},\bm{\zeta}\rangle_{\mathcal{V}_0}\,dt
  \end{multline}
  for all $\bm{\zeta}\in L^2(0,T;\mathcal{V}_0)$.
  Also, for all $t\in[0,\infty)$, we have the energy equality
  \begin{align} \label{E:S2_Ener}
    \frac{1}{2}\|\mathbf{v}(t)\|_{L^2(S^2)}^2+2\nu\int_0^t\|\mathbf{D}_{S^2}(\mathbf{v})\|_{L^2(S^2)}^2\,dt = \frac{1}{2}\|\mathbf{v}_0\|_{L^2(S^2)}^2+\int_0^t\langle\mathbf{f},\mathbf{v}\rangle_{\mathcal{V}_0}\,ds.
  \end{align}
\end{proposition}

\begin{proof}
  Let $\bm{\zeta}\in C_c^1(0,T;\mathcal{V}_0)$ in \eqref{E:WF_S2}.
  We see that
  \begin{align*}
    |(\nabla_{\mathbf{v}}\mathbf{v},\bm{\zeta})_{L^2(S^2)}| = |(\mathbf{v},\nabla_{\mathbf{v}}\bm{\zeta})_{L^2(S^2)}| \leq c\|\mathbf{v}\|_{L^2(S^2)}\|\mathbf{v}\|_{H^1(S^2)}\|\bm{\zeta}\|_{H^1(S^2)}
  \end{align*}
  on $(0,T)$ by $\mathbf{v}\in L^2(0,T;\mathcal{V}_0)$ and Lemma \ref{L:Tri_S2}.
  Thus, by H\"{o}lder's inequality,
  \begin{align} \label{Pf_S2C:Tri}
    \left|\int_0^T(\nabla_{\mathbf{v}}\mathbf{v},\bm{\zeta})_{L^2(S^2)}\,dt\right| \leq c\|\mathbf{v}\|_{L^\infty(0,T;L^2(S^2))}\|\mathbf{v}\|_{L^2(0,T;H^1(S^2))}\|\bm{\zeta}\|_{L^2(0,T;H^1(S^2))}.
  \end{align}
  We apply $\bm{\zeta}(0)=\mathbf{0}$, H\"{o}lder's inequality, and \eqref{Pf_S2C:Tri} to \eqref{E:WF_S2}.
  Then,
  \begin{align*}
    \left|\int_0^T(\mathbf{v},\partial_t\bm{\zeta})_{L^2(S^2)}\,dt\right| \leq c_\nu K_0\|\bm{\zeta}\|_{L^2(0,T;H^1(S^2))},
  \end{align*}
  where $c_\nu>0$ is a constant depending on $\nu$ and
  \begin{align*}
    K_0 := \Bigl(1+\|\mathbf{v}\|_{L^\infty(0,T;L^2(S^2))}\Bigr)\|\mathbf{v}\|_{L^2(0,T;H^1(S^2))}+\|\mathbf{f}\|_{L^2(0,T;\mathcal{V}_0^\ast)}.
  \end{align*}
  Since $C_c^1(0,T;\mathcal{V}_0)$ is dense in $L^2(0;T;\mathcal{V}_0)$, we get
  \begin{align*}
    \partial_t\mathbf{v} \in [L^2(0,T;\mathcal{V}_0)]^\ast = L^2(0,T;\mathcal{V}_0^\ast).
  \end{align*}
  Also, $\mathbf{v}\in L^2(0,T;\mathcal{V}_0)$.
  Thus, $\mathbf{v}\in\mathbb{E}_T(\mathcal{V}_0)\subset C([0,T];\mathcal{H}_0)$ by Lemma \ref{L:BS_Ipl}.

  Next, for $\bm{\zeta}\in C_c^1(0,T;\mathcal{V}_0)$, we have \eqref{E:WFS2_Ano} by applying \eqref{E:Dt_IbP} to \eqref{E:WF_S2}.
  Since each term of \eqref{E:WFS2_Ano} is linear and bounded in $\bm{\zeta}\in L^2(0,T;\mathcal{V}_0)$ by \eqref{Pf_S2C:Tri} and
  \begin{align*}
    \mathbf{v} \in L^2(0,T;\mathcal{V}_0), \quad \partial_t\mathbf{v},\mathbf{f}\in L^2(0,T;\mathcal{V}_0^\ast),
  \end{align*}
  and since $C_c^1(0,T;\mathcal{V}_0)$ is dense in $L^2(0,T;\mathcal{V}_0)$, we see by a density argument that \eqref{E:WFS2_Ano} also holds for $\bm{\zeta}\in L^2(0,T;\mathcal{V}_0)$.
  Thus, we can set $\bm{\zeta}=\mathbf{v}$ in \eqref{E:WFS2_Ano} and use \eqref{E:Dt_IbP} and \eqref{E:TrS2_AnS} to get \eqref{E:S2_Ener} with $t$ replaced by $T$.
\end{proof}

As in \eqref{E:NS_TSS}, we can get the existence of a weak solution to \eqref{E:NS_S2} by the Galerkin method with the aid of \eqref{E:Korn_S2}.
Moreover, by Proposition \ref{P:S2_Cont}, we can carry out the energy method to show that a weak solution is unique.
In summary, the following result holds.

\begin{proposition} \label{P:S2_Exist}
  For all $\mathbf{v}_0\in\mathcal{H}_0$ and $\mathbf{f}\in L_{\mathrm{loc}}^2([0,\infty);\mathcal{V}_0^\ast)$, there exists a unique weak solution $\mathbf{v}$ to \eqref{E:NS_S2}.
  Moreover, $\mathbf{v}$ satisfies the energy equality \eqref{E:S2_Ener}.
\end{proposition}

\section{Average and extension} \label{S:AvEx}
In this section, we study the average of a function on $\Omega_\varepsilon$ and the extension of a function on $S^2$ in the radial direction.
Recall that the change of variables formula
\begin{align} \label{E:CoV_Thin}
  \int_{\Omega_\varepsilon}\varphi(x)\,dx = \int_{S^2}\int_1^{1+\varepsilon}\varphi(ry)r^2\,dr\,d\mathcal{H}^2(y)
\end{align}
holds for a function $\varphi$ on $\Omega_\varepsilon$.
Also, we write $\bar{\eta}(x)=\eta(x/|x|)$ for the constant extension of a function $\eta$ on $S^2$ in the radial direction.
Then, by \eqref{E:CoV_Thin} and $r\leq2$, we have
\begin{align} \label{E:CoEx_L2}
  \|\bar{\eta}\|_{L^2(\Omega_\varepsilon)} = \left(\int_{S^2}\int_1^{1+\varepsilon}|\eta(y)|^2r^2\,dr\,d\mathcal{H}^2(y)\right)^{1/2} \leq 2\varepsilon^{1/2}\|\eta\|_{L^2(S^2)}.
\end{align}

\subsection{Average} \label{SS:AE_Ave}
For $k\in\mathbb{Z}_{\geq0}$ and a function $\varphi$ on $\Omega_\varepsilon$, we define the weighted average
\begin{align} \label{E:Def_Ave}
  \mathcal{M}_\varepsilon^k\varphi(y) := \frac{1}{\varepsilon}\int_1^{1+\varepsilon}\varphi(ry)r^k\,dr, \quad y\in S^2.
\end{align}
We use the same notation for a vector field $\mathbf{u}$ on $\Omega_\varepsilon$, and write
\begin{align} \label{E:Def_Atan}
  \mathcal{M}_{\varepsilon,\tau}^k\mathbf{u}(y) := \mathbf{P}(y)\mathcal{M}_\varepsilon^k\mathbf{u}(y) = \mathcal{M}_\varepsilon^k\mathbf{u}(y)-\{\mathcal{M}_\varepsilon^k\mathbf{u}(y)\cdot\mathbf{n}(y)\}\mathbf{n}(y), \quad y\in S^2
\end{align}
for the tangential component of the weighted average of $\mathbf{u}$.
The reason why we introduce the weighted averages of different order is that we will encounter the expression
\begin{align*}
  \int_{\Omega_\varepsilon}|x|^k\bar{\eta}(x)\varphi(x)\,dx &= \int_{S^2}\int_1^{1+\varepsilon}\{r^k\eta(y)\varphi(ry)\}r^2\,dr\,d\mathcal{H}^2(y) \\
  &= \varepsilon\int_{S^2}\eta(y)\mathcal{M}_\varepsilon^{k+2}\varphi(y)\,d\mathcal{H}^2(y)
\end{align*}
for functions $\varphi$ on $\Omega_\varepsilon$ and $\eta$ on $S^2$ and for several $k\geq0$ later.

Let us show some properties of the weighted averages.
We write $c$ for a general positive constant which may depend on $k,\ell\in\mathbb{Z}_{\geq0}$ but is independent of $\varepsilon\in(0,1)$.

\begin{lemma} \label{L:Ave_L2}
  For $k,\ell\in\mathbb{Z}_{\geq0}$ and $\varphi\in L^2(\Omega_\varepsilon)$, we have
  \begin{align}
    \|\mathcal{M}_\varepsilon^k\varphi\|_{L^2(S^2)} &\leq c\varepsilon^{-1/2}\|\varphi\|_{L^2(\Omega_\varepsilon)}, \label{E:Ave_L2} \\
    \|\mathcal{M}_\varepsilon^k\varphi-\mathcal{M}_\varepsilon^\ell\varphi\|_{L^2(S^2)} &\leq c\varepsilon^{1/2}\|\varphi\|_{L^2(\Omega_\varepsilon)}. \label{E:AvDf_L2}
  \end{align}
\end{lemma}

\begin{proof}
  For $y\in S^2$, we see by H\"{o}lder's inequality (and $0<\varepsilon<1$) that
  \begin{align*}
    |\mathcal{M}_\varepsilon^k\varphi(y)| &\leq \frac{1}{\varepsilon}\left(\int_1^{1+\varepsilon}|\varphi(ry)|^2r^2\,dr\right)^{1/2}\left(\int_1^{1+\varepsilon}r^{2k-2}\,dr\right)^{1/2} \\
    &\leq \frac{\max\{1,2^{k-1}\}}{\varepsilon^{1/2}}\left(\int_1^{1+\varepsilon}|\varphi(ry)|^2r^2\,dr\right)^{1/2}.
  \end{align*}
  We integrate the square of both sides over $S^2$ and use \eqref{E:CoV_Thin} to get \eqref{E:Ave_L2}.

  Next, we show \eqref{E:AvDf_L2}.
  We may assume $k>\ell$ without loss of generality.
  Then, since
  \begin{align*}
    \mathcal{M}_\varepsilon^k\varphi(y)-\mathcal{M}_\varepsilon^\ell\varphi(y) = \frac{1}{\varepsilon}\int_1^{1+\varepsilon}\varphi(ry)r^\ell(r^{k-\ell}-1)\,dr, \quad y\in S^2
  \end{align*}
  and $|r^{k-\ell}-1|=|r-1|\sum_{j=0}^{k-\ell-1}r^j\leq c\varepsilon$ for $r\in[1,1+\varepsilon]$, we have
  \begin{align*}
    |\mathcal{M}_\varepsilon^k\varphi(y)-\mathcal{M}_\varepsilon^\ell\varphi(y)| \leq c\int_1^{1+\varepsilon}|\varphi(ry)|r^\ell\,dr = c\varepsilon[\mathcal{M}_\varepsilon^\ell(|\varphi|)](y), \quad y\in S^2.
  \end{align*}
  Using this estimate and applying \eqref{E:Ave_L2} to $\mathcal{M}_\varepsilon^\ell(|\varphi|)$, we get \eqref{E:AvDf_L2}.
\end{proof}

\begin{lemma} \label{L:AvCo_Df}
  For $k\in\mathbb{Z}_{\geq0}$ and $\varphi\in H^1(S^2)$, we have
  \begin{align} \label{E:AvCo_Df}
    \Bigl\|\varphi-\overline{\mathcal{M}_\varepsilon^k\varphi}\Bigr\|_{L^2(S^2)} \leq c\varepsilon\|\varphi\|_{H^1(\Omega_\varepsilon)}.
  \end{align}
\end{lemma}

\begin{proof}
  For $y\in S^2$ and $r\in[1,1+\varepsilon]$, we have
  \begin{align*}
    &\varphi(ry)-\overline{\mathcal{M}_\varepsilon^k\varphi}(ry) = \varphi(ry)-\frac{1}{\varepsilon}\int_1^{1+\varepsilon}\varphi(r_1y)r_1^k\,dr_1 = K_1+K_2, \\
    &K_1 := \frac{1}{\varepsilon}\int_1^{1+\varepsilon}\{\varphi(ry)-\varphi(r_1y)\}\,dr_1, \quad K_2 := \frac{1}{\varepsilon}\int_1^{1+\varepsilon}\varphi(r_1y)(1-r_1^k)\,dr_1.
  \end{align*}
  Since $\partial_r[\varphi(ry)]=y\cdot\nabla\varphi(ry)$ and $|y|=1$ for $y\in S^2$,
  \begin{align*}
    |\varphi(ry)-\varphi(r_1y)| = \left|\int_{r_1}^ry\cdot\nabla\varphi(r_2y)\,dr_2\right| \leq \int_1^{1+\varepsilon}|\nabla\varphi(r_2y)|\,dr_2,
  \end{align*}
  where right-hand side is independent of $r_1$.
  Thus,
  \begin{align*}
    |K_1| \leq \int_1^{1+\varepsilon}|\nabla\varphi(r_2y)|\,dr_2 = \varepsilon[\mathcal{M}_\varepsilon^0(|\nabla\varphi|)](y).
  \end{align*}
  Also, $|K_2|\leq c\int_1^{1+\varepsilon}|\varphi(r_1y)|\,dr_1=c\varepsilon[\mathcal{M}_\varepsilon^0(|\varphi|)](y)$ by $|1-r_1^k|\leq c\varepsilon$.
  Hence,
  \begin{align*}
    \Bigl|\varphi(ry)-\overline{\mathcal{M}_\varepsilon^k\varphi}(ry)\Bigr| \leq |K_1|+|K_2| \leq c\varepsilon[\mathcal{M}_\varepsilon^0(|\varphi|+|\nabla\varphi|)](y).
  \end{align*}
  Since the right-hand side is independent of $r$, we see by \eqref{E:CoEx_L2} that
  \begin{align*}
    \Bigl\|\varphi-\overline{\mathcal{M}_\varepsilon^k\varphi}\Bigr\|_{L^2(\Omega_\varepsilon)} \leq c\varepsilon^{3/2}\|\mathcal{M}_\varepsilon^0(|\varphi|+|\nabla\varphi|)\|_{L^2(S^2)}.
  \end{align*}
  Applying \eqref{E:Ave_L2} to the right-hand side, we get \eqref{E:AvCo_Df}.
\end{proof}

\begin{lemma} \label{L:Ave_NC}
  Let $\mathbf{u}\in H^1(\Omega_\varepsilon)^3$ satisfy $\mathbf{u}\cdot\mathbf{n}_\varepsilon=0$ on $\partial\Omega_\varepsilon$.
  Then, for $k\in\mathbb{Z}_{\geq0}$,
  \begin{align} \label{E:Ave_NC}
    \|\mathcal{M}_\varepsilon^k\mathbf{u}\cdot\mathbf{n}\|_{L^2(S^2)} = \|\mathcal{M}_\varepsilon^k(\mathbf{u}\cdot\bar{\mathbf{n}})\|_{L^2(S^2)} \leq c\varepsilon^{1/2}\|\mathbf{u}\|_{H^1(\Omega_\varepsilon)}.
  \end{align}
  Moreover, for the tangential component $\mathcal{M}_{\varepsilon,\tau}^k\mathbf{u}$, we have
  \begin{align} \label{E:Atan_Con}
    \Bigl\|\mathbf{u}-\overline{\mathcal{M}_{\varepsilon,\tau}^k\mathbf{u}}\Bigr\|_{L^2(\Omega_\varepsilon)} \leq c\varepsilon\|\mathbf{u}\|_{H^1(\Omega_\varepsilon)}.
  \end{align}
\end{lemma}

\begin{proof}
  By the definition \eqref{E:Def_Ave} of $\mathcal{M}_\varepsilon^k$, we have $\mathcal{M}_\varepsilon^k\mathbf{u}\cdot\mathbf{n}=\mathcal{M}_\varepsilon^k(\mathbf{u}\cdot\bar{\mathbf{n}})$ on $S^2$.
  Also,
  \begin{align*}
    \|\mathcal{M}_\varepsilon^k(\mathbf{u}\cdot\bar{\mathbf{n}})\|_{L^2(S^2)} \leq c\varepsilon^{-1/2}\|\mathbf{u}\cdot\bar{\mathbf{n}}\|_{L^2(\Omega_\varepsilon)} \leq c\varepsilon^{1/2}\|\mathbf{u}\|_{H^1(\Omega_\varepsilon)}
  \end{align*}
  by \eqref{E:Ave_L2} and then by \eqref{E:Nor_Thin}.
  Thus, \eqref{E:Ave_NC} is valid.
  Also, since
  \begin{align*}
    \Bigl\|\mathbf{u}-\overline{\mathcal{M}_{\varepsilon,\tau}^k\mathbf{u}}\Bigr\|_{L^2(\Omega_\varepsilon)} \leq \Bigl\|\mathbf{u}-\overline{\mathcal{M}_\varepsilon^k\mathbf{u}}\Bigr\|_{L^2(\Omega_\varepsilon)}+\Bigl\|\overline{\mathcal{M}_\varepsilon^k\mathbf{u}\cdot\mathbf{n}}\Bigr\|_{L^2(\Omega_\varepsilon)}
  \end{align*}
  by \eqref{E:Def_Atan} and $|\mathbf{n}|=1$ on $S^2$, we apply \eqref{E:CoEx_L2}, \eqref{E:AvCo_Df}, and \eqref{E:Ave_NC} to get \eqref{E:Atan_Con}.
\end{proof}

\begin{lemma} \label{L:Ave_TGr}
  For $k\in\mathbb{Z}_{\geq0}$ and $\varphi\in H^1(\Omega_\varepsilon)$, we have
  \begin{align} \label{E:Ave_TGr}
    \nabla_{S^2}\mathcal{M}_\varepsilon^k\varphi = \mathcal{M}_\varepsilon^{k+1}\Bigl(\overline{\mathbf{P}}\nabla\varphi\Bigr) = \mathbf{P}\mathcal{M}_\varepsilon^{k+1}(\nabla\varphi) \quad\text{on}\quad S^2.
  \end{align}
\end{lemma}

\begin{proof}
  The constant extension of $\mathcal{M}_\varepsilon^k\varphi$ in the radial direction is given by
  \begin{align*}
    \overline{\mathcal{M}_\varepsilon^k\varphi}(x) = \frac{1}{\varepsilon}\int_1^{1+\varepsilon}\varphi\left(r\,\frac{x}{|x|}\right)r^k\,dr, \quad x\in\mathbb{R}^3\setminus\{0\}.
  \end{align*}
  We differentiate both sides in $x$ and use \eqref{Pf_Co:Gr_x} to get
  \begin{align*}
    \nabla\overline{\mathcal{M}_\varepsilon^k\varphi}(x) = \frac{1}{\varepsilon}\int_1^{1+\varepsilon}\frac{r}{|x|}\overline{\mathbf{P}}(x)\nabla\varphi\left(r\,\frac{x}{|x|}\right)r^k\,dr, \quad x\in\mathbb{R}^3\setminus\{0\}.
  \end{align*}
  Combining this equality and \eqref{E:Const}, and setting $x=y\in S^2$, we get \eqref{E:Ave_TGr}.
\end{proof}

\begin{lemma} \label{L:Ave_H1}
  For $k,\ell\in\mathbb{Z}_{\geq0}$ and $\varphi\in H^1(\Omega_\varepsilon)$, we have
  \begin{align}
    \|\nabla_{S^2}\mathcal{M}_\varepsilon^k\varphi\|_{L^2(S^2)} &\leq c\varepsilon^{-1/2}\|\nabla\varphi\|_{L^2(\Omega_\varepsilon)}, \label{E:Ave_H1} \\
    \|\nabla_{S^2}\mathcal{M}_\varepsilon^k\varphi-\nabla_{S^2}\mathcal{M}_\varepsilon^\ell\varphi\|_{L^2(S^2)} &\leq c\varepsilon^{1/2}\|\nabla\varphi\|_{L^2(\Omega_\varepsilon)}. \label{E:AvDf_H1}
  \end{align}
\end{lemma}

\begin{proof}
  Lemmas \ref{L:Ave_L2} and \ref{L:Ave_TGr} imply \eqref{E:Ave_H1} and \eqref{E:AvDf_H1}, since $|\mathbf{P}\mathbf{a}|\leq|\mathbf{a}|$ on $S^2$ for $\mathbf{a}\in\mathbb{R}^3$.
\end{proof}

\begin{lemma} \label{L:Atan_H1}
  For $k,\ell\in\mathbb{Z}_{\geq0}$ and $\mathbf{u}\in H^1(\Omega_\varepsilon)^3$, we have $\mathcal{M}_{\varepsilon,\tau}^k\mathbf{u}\in H_\tau^1(S^2)$ and
  \begin{align}
    \|\mathcal{M}_{\varepsilon,\tau}^k\mathbf{u}\|_{H^1(S^2)} &\leq c\varepsilon^{-1/2}\|\mathbf{u}\|_{H^1(\Omega_\varepsilon)}, \label{E:Atan_H1} \\
    \|\mathcal{M}_{\varepsilon,\tau}^k\mathbf{u}-\mathcal{M}_{\varepsilon,\tau}^\ell\mathbf{u}\|_{H^1(S^2)} &\leq c\varepsilon^{1/2}\|\mathbf{u}\|_{H^1(\Omega_\varepsilon)}. \label{E:Atan_Df}
  \end{align}
\end{lemma}

\begin{proof}
  Since $\mathcal{M}_{\varepsilon,\tau}^k\mathbf{u}=\mathbf{P}\mathcal{M}_\varepsilon^k\mathbf{u}$ is the tangential component of $\mathcal{M}_\varepsilon^k\mathbf{u}$, we have
  \begin{align*}
    \|\mathcal{M}_{\varepsilon,\tau}^k\mathbf{u}\|_{L^2(S^2)} \leq \|\mathcal{M}_\varepsilon^k\mathbf{u}\|_{L^2(S^2)} \leq c\varepsilon^{-1/2}\|\mathbf{u}\|_{L^2(\Omega_\varepsilon)}
  \end{align*}
  by \eqref{E:Ave_L2}.
  Also, we see by \eqref{E:Vec_Tan}, $|\mathbf{n}|=1$ and $|\mathbf{P}|=\sqrt{2}$ on $S^2$, \eqref{E:Ave_L2}, and \eqref{E:Ave_H1} that
  \begin{align*}
    \|\nabla_{S^2}\mathcal{M}_{\varepsilon,\tau}^k\mathbf{u}\|_{L^2(S^2)} \leq c\Bigl(\|\nabla_{S^2}\mathcal{M}_\varepsilon^k\mathbf{u}\|_{L^2(S^2)}+\|\mathcal{M}_\varepsilon^k\mathbf{u}\|_{L^2(S^2)}\Bigr) \leq c\varepsilon^{-1/2}\|\mathbf{u}\|_{H^1(\Omega_\varepsilon)}.
  \end{align*}
  Hence, \eqref{E:Atan_H1} follows.
  Similarly, we can get \eqref{E:Atan_Df} by \eqref{E:AvDf_L2} and \eqref{E:AvDf_H1}.
\end{proof}

\begin{lemma} \label{L:Ave_div}
  Let $\mathbf{u}\in H^1(\Omega_\varepsilon)^3$ satisfy $\mathbf{u}\cdot\mathbf{n}_\varepsilon=0$ on $\partial\Omega_\varepsilon$.
  Then, for $k\in\mathbb{Z}_{\geq0}$,
  \begin{align} \label{E:Ave_div}
    \mathrm{div}_{S^2}\mathcal{M}_\varepsilon^k\mathbf{u} = \mathcal{M}_\varepsilon^{k+1}(\mathrm{div}\,\mathbf{u})+(k+1)\mathcal{M}_\varepsilon^k(\mathbf{u}\cdot\bar{\mathbf{n}}) \quad\text{on}\quad S^2.
  \end{align}
  Moreover, for the tangential component $\mathcal{M}_{\varepsilon,\tau}^k\mathbf{u}$, we have
  \begin{align} \label{E:Atan_div}
    \mathrm{div}_{S^2}\mathcal{M}_{\varepsilon,\tau}^k\mathbf{u} = \mathcal{M}_\varepsilon^{k+1}(\mathrm{div}\,\mathbf{u})+(k-1)\mathcal{M}_\varepsilon^k(\mathbf{u}\cdot\bar{\mathbf{n}}) \quad\text{on}\quad S^2.
  \end{align}
\end{lemma}

\begin{proof}
  Since $\mathbf{P}=\mathbf{I}_3-\mathbf{n}\otimes\mathbf{n}$ on $S^2$, we see by \eqref{E:Ave_TGr} that
  \begin{align*}
    \nabla_{S^2}\mathcal{M}_\varepsilon^k\mathbf{u} = \mathcal{M}_\varepsilon^{k+1}(\nabla\mathbf{u})-(\mathbf{n}\otimes\mathbf{n})\mathcal{M}_\varepsilon^{k+1}(\nabla\mathbf{u}) \quad\text{on}\quad S^2.
  \end{align*}
  Taking the trace of both sides, we get
  \begin{align} \label{Pf_Ad:div}
    \mathrm{div}_{S^2}\mathcal{M}_\varepsilon^k\mathbf{u} = \mathcal{M}_\varepsilon^{k+1}(\mathrm{div}\,\mathbf{u})-f_{\mathbf{u}}, \quad f_{\mathbf{u}} := \sum_{i,j=1}^3n_in_j\mathcal{M}_\varepsilon^{k+1}(\partial_ju_i)
  \end{align}
  on $S^2$.
  Moreover, letting $r_\varepsilon:=1+\varepsilon$, we see by $\mathbf{n}(y)=y$ for $y\in S^2$ that
  \begin{align*}
    \sum_{j=1}^3n_j(y)[\mathcal{M}_\varepsilon^{k+1}(\partial_ju_i)](y) &= \frac{1}{\varepsilon}\sum_{j=1}^3\int_1^{r_\varepsilon}y_j\partial_ju_i(ry)r^{k+1}\,dr \\
    &= \frac{1}{\varepsilon}\int_1^{r_\varepsilon}\left[\frac{\partial}{\partial r}\Bigl(u_i(ry)\Bigr)\right]r^{k+1}\,dr \\
    &= \frac{1}{\varepsilon}\{u_i(r_\varepsilon y)r_\varepsilon^{k+1}-u_i(y)\}-\frac{k+1}{\varepsilon}\int_1^{r_\varepsilon}u_i(ry)r^k\,dr.
  \end{align*}
  We multiply both sides by $n_i(y)=y_i$ and sum over $i=1,2,3$ to get
  \begin{align*}
    f_{\mathbf{u}}(y) = \frac{1}{\varepsilon}\{[\mathbf{u}(r_\varepsilon y)\cdot y]r_\varepsilon^{k+1}-\mathbf{u}(y)\cdot y\}-\frac{k+1}{\varepsilon}\int_1^{r_\varepsilon}[\mathbf{u}(ry)\cdot y]r^k\,dr.
  \end{align*}
  Now, we recall that $\mathbf{u}$ satisfies $\mathbf{u}\cdot\mathbf{n}_\varepsilon=0$ on $\partial\Omega_\varepsilon$ and
  \begin{align*}
    \partial\Omega_\varepsilon = S^2\cup\{r_\varepsilon y \mid y\in S^2\}, \quad \mathbf{n}_\varepsilon(y) = -y, \quad \mathbf{n}_\varepsilon(r_\varepsilon y) = y, \quad y\in S^2.
  \end{align*}
  Thus, $\mathbf{u}(r_\varepsilon y)\cdot y=\mathbf{u}(y)\cdot y=0$.
  Also, for $y\in S^2$ and $r\in[1,r_\varepsilon]$,
  \begin{align*}
    \mathbf{u}(ry)\cdot y = \mathbf{u}(ry)\cdot\mathbf{n}(y) = \mathbf{u}(ry)\cdot\bar{\mathbf{n}}(ry) = [\mathbf{u}\cdot\bar{\mathbf{n}}](ry).
  \end{align*}
  Therefore, noting that $r_\varepsilon=1+\varepsilon$, we have
  \begin{align*}
    f_{\mathbf{u}}(y) = -\frac{k+1}{\varepsilon}\int_1^{r_\varepsilon}[\mathbf{u}\cdot\bar{\mathbf{n}}](ry)r^k\,dr = -(k+1)[\mathcal{M}_\varepsilon^k(\mathbf{u}\cdot\bar{\mathbf{n}})](y)
  \end{align*}
  for $y\in S^2$, and we apply this to \eqref{Pf_Ad:div} to find that \eqref{E:Ave_div} is valid.

  Next, let $\mathbf{v}:=\mathcal{M}_\varepsilon^k\mathbf{u}$ on $S^2$.
  Then, $\mathcal{M}_{\varepsilon,\tau}^k\mathbf{u}=\mathbf{P}\mathbf{v}=\mathbf{v}-(\mathbf{v}\cdot\mathbf{n})\mathbf{n}$ and
  \begin{align*}
    \mathrm{div}_{S^2}[\mathbf{P}\mathbf{v}] = \mathrm{div}_{S^2}\mathbf{v}-[\nabla_{S^2}(\mathbf{v}\cdot\mathbf{n})]\cdot\mathbf{n}-(\mathbf{v}\cdot\mathbf{n})\mathrm{div}_{S^2}\mathbf{n} = \mathrm{div}_{S^2}\mathbf{v}-2(\mathbf{v}\cdot\mathbf{n})
  \end{align*}
  on $S^2$ by \eqref{E:TGr_Nor} and $\nabla_{S^2}\eta\cdot\mathbf{n}=0$ for a function $\eta$ on $S^2$.
  Thus, we obtain \eqref{E:Atan_div} by applying \eqref{E:Ave_div} and $\mathbf{v}\cdot\mathbf{n}=\mathcal{M}_\varepsilon^k(\mathbf{u}\cdot\bar{\mathbf{n}})$ to the above relation.
\end{proof}

\begin{lemma} \label{L:Atdiv_L2}
  For $k\in\mathbb{Z}_{\geq0}$ and $\mathbf{u}\in\mathcal{V}_\varepsilon$, we have
  \begin{align} \label{E:Atdiv_L2}
    \|\mathrm{div}_{S^2}\mathcal{M}_{\varepsilon,\tau}^k\mathbf{u}\|_{L^2(S^2)} \leq c\varepsilon^{1/2}\|\mathbf{u}\|_{H^1(\Omega_\varepsilon)}.
  \end{align}
  Moreover, letting $\mathbb{L}_0$ be the Helmholtz--Leray projection on $L_\tau^2(S^2)$, we have
  \begin{align} \label{E:Ave_HL}
    \|\mathcal{M}_{\varepsilon,\tau}^k\mathbf{u}-\mathbb{L}_0\mathcal{M}_{\varepsilon,\tau}^k\mathbf{u}\|_{H^1(S^2)} \leq c\varepsilon^{1/2}\|\mathbf{u}\|_{H^1(\Omega_\varepsilon)}.
  \end{align}
\end{lemma}

\begin{proof}
  Since $\mathrm{div}\,\mathbf{u}=0$ in $\Omega_\varepsilon$ and $\mathbf{u}\cdot\mathbf{n}_\varepsilon=0$ on $\partial\Omega_\varepsilon$ by $\mathbf{u}\in\mathcal{V}_\varepsilon$, we have
  \begin{align*}
    \|\mathrm{div}_{S^2}\mathcal{M}_{\varepsilon,\tau}^k\mathbf{u}\|_{L^2(S^2)} = (k-1)\|\mathcal{M}_\varepsilon^k(\mathbf{u}\cdot\bar{\mathbf{n}})\|_{L^2(S^2)} \leq c\varepsilon^{1/2}\|\mathbf{u}\|_{H^1(\Omega_\varepsilon)}
  \end{align*}
  by \eqref{E:Ave_NC} and \eqref{E:Atan_div}.
  Thus, \eqref{E:Atdiv_L2} holds.
  Also, we have \eqref{E:Ave_HL} by \eqref{E:HLH1_S2} and \eqref{E:Atdiv_L2}.
\end{proof}

\subsection{Extension} \label{SS:AE_Ext}
For a vector field $\mathbf{v}$ on $S^2$, we define the weighted extension
\begin{align} \label{E:Def_Ext}
  \mathbf{v}_{\mathrm{E}}(x) := |x|\,\bar{\mathbf{v}}(x) = |x|\,\mathbf{v}\left(\frac{x}{|x|}\right), \quad x\in\mathbb{R}^3\setminus\{0\}.
\end{align}
Of course, this extension is close to the constant extension $\bar{\mathbf{v}}$ in $L^2(\Omega_\varepsilon)^3$.

\begin{lemma} \label{L:DfEx_L2}
  For $\mathbf{v}\in L^2(S^2)^3$, we have
  \begin{align} \label{E:DfEx_L2}
    \|\mathbf{v}_{\mathrm{E}}-\bar{\mathbf{v}}\|_{L^2(\Omega_\varepsilon)} \leq 2\varepsilon^{3/2}\|\mathbf{v}\|_{L^2(S^2)}.
  \end{align}
\end{lemma}

\begin{proof}
  We see by \eqref{E:CoV_Thin} that
  \begin{align*}
    \|\mathbf{v}_{\mathrm{E}}-\bar{\mathbf{v}}\|_{L^2(\Omega_\varepsilon)}^2 =\int_{S^2}\int_1^{1+\varepsilon}|r-1|^2|\mathbf{v}(y)|^2r^2\,dr\,d\mathcal{H}^2(y).
  \end{align*}
  To the right-hand side, we apply $|r-1|\leq\varepsilon$ and $r\leq2$.
  Then, we get \eqref{E:DfEx_L2}.
\end{proof}

On the other hand, the extension $\mathbf{v}_{\mathrm{E}}$ has the following good properties.

\begin{lemma} \label{L:Ext_Grad}
  Let $\mathbf{v}\in H^1(S^2)^3$.
  Then, for $x\in\mathbb{R}^3\setminus\{0\}$,
  \begin{align}
    \nabla\mathbf{v}_{\mathrm{E}}(x) &= \overline{\nabla_{S^2}\mathbf{v}}(x)+\bar{\mathbf{n}}(x)\otimes\bar{\mathbf{v}}(x), \label{E:Ext_Grad} \\
    \mathrm{div}\,\mathbf{v}_{\mathrm{E}}(x) &= \overline{\mathrm{div}_{S^2}\mathbf{v}}(x)+\bar{\mathbf{n}}(x)\cdot\bar{\mathbf{v}}(x). \label{E:Ext_div}
  \end{align}
  Suppose further that $\mathbf{v}$ is tangential on $S^2$.
  Then, for $x\in\mathbb{R}^3\setminus\{0\}$,
  \begin{align}
    [\mathbf{D}(\mathbf{v}_{\mathrm{E}})](x) &= \Bigl[\overline{\mathbf{D}_{S^2}(\mathbf{v})}\Bigr](x) \label{E:Ext_str} \\
    [(\mathbf{v}_{\mathrm{E}}\cdot\nabla)\mathbf{v}_{\mathrm{E}}](x) &= |x|\Bigl\{\overline{\nabla_{\mathbf{v}}\mathbf{v}}(x)-|\bar{\mathbf{v}}(x)|^2\bar{\mathbf{n}}(x)\Bigr\}. \label{E:Ext_CoDe}
  \end{align}
\end{lemma}

\begin{proof}
  We see by \eqref{E:Const} and $\nabla|x|=x/|x|=\bar{\mathbf{n}}(x)$ that
  \begin{align*}
    \nabla\mathbf{v}_{\mathrm{E}}(x) = (\nabla|x|)\otimes\bar{\mathbf{v}}(x)+|x|\nabla\bar{\mathbf{v}}(x) = \bar{\mathbf{n}}(x)\otimes\bar{\mathbf{v}}(x)+\overline{\nabla_{S^2}\mathbf{v}}(x).
  \end{align*}
  Thus, \eqref{E:Ext_Grad} holds, and its trace gives \eqref{E:Ext_div}.

  Next, suppose that $\mathbf{v}$ is tangential on $S^2$.
  Then, by \eqref{E:TGr_Dec} and \eqref{E:Ext_Grad}, we have
  \begin{align*}
    \nabla\mathbf{v}_{\mathrm{E}}(x) = \overline{\mathbf{P}(\nabla_{S^2}\mathbf{v})\mathbf{P}}(x)-\bar{\mathbf{v}}(x)\otimes\bar{\mathbf{n}}(x)+\bar{\mathbf{n}}(x)\otimes\bar{\mathbf{v}}(x).
  \end{align*}
  Taking the symmetric part of both sides, we obtain \eqref{E:Ext_str}.
  Also, since
  \begin{align*}
    [(\mathbf{v}_{\mathrm{E}}\cdot\nabla)\mathbf{v}_{\mathrm{E}}](x) = |x|\Bigl\{\overline{(\mathbf{v}\cdot\nabla_{S^2})\mathbf{v}}(x)+[\bar{\mathbf{v}}(x)\cdot\bar{\mathbf{n}}(x)]\bar{\mathbf{v}}(x)\Bigr\}
  \end{align*}
  by \eqref{E:Ext_Grad}, we apply $\mathbf{v}\cdot\mathbf{n}=0$ on $S^2$ and \eqref{E:Gauss} to the right-hand side to get \eqref{E:Ext_CoDe}.
\end{proof}

\begin{lemma} \label{L:Ext_Bdd}
  Let $\mathcal{X}=L^2,H^1$.
  For $\mathbf{v}\in\mathcal{X}(S^2)^3$, we have
  \begin{align} \label{E:Ext_Bdd}
    \|\mathbf{v}_{\mathrm{E}}\|_{\mathcal{X}(\Omega_\varepsilon)} \leq c\varepsilon^{1/2}\|\mathbf{v}\|_{\mathcal{X}(S^2)}.
  \end{align}
\end{lemma}

\begin{proof}
  Let $\eta\in L^2(S^2)$ and $k=0,1$.
  Noting that $r\leq1+\varepsilon\leq2$, we see that
  \begin{align*}
    \int_{\Omega_\varepsilon}|x|^{2k}|\bar{\eta}(x)|^2\,dx = \int_{S^2}\int_1^{1+\varepsilon}|\eta(y)|^2r^{2k+2}\,dr\,d\mathcal{H}^2(y) \leq 2^{2k+2}\varepsilon\int_{S^2}|\eta(y)|^2\,d\mathcal{H}^2(y)
  \end{align*}
  by \eqref{E:CoV_Thin}.
  We get \eqref{E:Ext_Bdd} by this inequality, \eqref{E:Ext_Grad}, and $|\mathbf{n}|=1$ on $S^2$.
\end{proof}

\begin{lemma} \label{L:Ext_Sol}
  If $\mathbf{v}\in\mathcal{V}_0$, then $\mathbf{v}_{\mathrm{E}}\in\mathcal{V}_\varepsilon$.
  Also, if $\mathbf{v}\in\mathcal{H}_0$, then $\mathbf{v}_{\mathrm{E}}\in\mathcal{H}_\varepsilon$.
\end{lemma}

\begin{proof}
  Let $\mathbf{v}\in\mathcal{V}_0$.
  Then,
  \begin{align*}
    \mathbf{v}(y)\cdot\mathbf{n}(y) = \mathbf{v}(y)\cdot y = 0, \quad \mathrm{div}_{S^2}\mathbf{v}(y) = 0, \quad y\in S^2.
  \end{align*}
  Thus, $\mathrm{div}\,\mathbf{v}_{\mathrm{E}}=0$ in $\mathbb{R}^3\setminus\{0\}$ by \eqref{E:Ext_div}.
  We also have $\mathbf{v}_{\mathrm{E}}\cdot\mathbf{n}_\varepsilon=0$ on $\partial\Omega_\varepsilon$ by the above first relation, since $\mathbf{n}_\varepsilon(y)=-y$ and $\mathbf{n}_\varepsilon((1+\varepsilon)y)=y$ for $y\in S^2$.
  Thus, $\mathbf{v}_{\mathrm{E}}\in\mathcal{V}_\varepsilon$.

  Next, let $\mathbf{v}\in\mathcal{H}_0$.
  By Lemma \ref{L:Sol_S2}, we can take $\mathbf{v}_k\in\mathcal{V}_0$ such that $\mathbf{v}_k\to\mathbf{v}$ strongly in $L_\tau^2(S^2)$ as $k\to\infty$.
  Then, $[\mathbf{v}_k]_{\mathrm{E}}\in\mathcal{V}_\varepsilon\subset\mathcal{H}_\varepsilon$ and $[\mathbf{v}_k]_{\mathrm{E}}\to\mathbf{v}_{\mathrm{E}}$ strongly in $L^2(\Omega_\varepsilon)^3$ as $k\to\infty$ by \eqref{E:Ext_Bdd}.
  Since $\mathcal{H}_\varepsilon$ is closed in $L^2(\Omega_\varepsilon)^3$, it follows that $\mathbf{v}_{\mathrm{E}}\in\mathcal{H}_\varepsilon$.
\end{proof}

When $\mathbf{v}$ is a weak solution to \eqref{E:NS_S2} on $S^2$, we expect that $\mathbf{v}_{\mathrm{E}}$ is an approximate weak solution to \eqref{E:NS_TSS} on $\Omega_\varepsilon$.
To verify it, we show that each term of the weak form \eqref{E:WF_TSS} with $\mathbf{u}^\varepsilon$ replaced by $\mathbf{v}_{\mathrm{E}}$ is approximated by that of the weak form \eqref{E:WF_S2}.

Let $\mathbb{L}_0$ be the Helmholtz--Leray projection on $L_\tau^2(S^2)$.
Also, let $\mathcal{M}_\varepsilon^k$ and $\mathcal{M}_{\varepsilon,\tau}^k$ be the average operators given by \eqref{E:Def_Ave} and \eqref{E:Def_Atan}.
We first give a basic relation.

\begin{lemma} \label{L:ExAv_L2}
  Let $\mathbf{v}\in\mathcal{H}_0$ and $\bm{\psi}\in L^2(\Omega_\varepsilon)^3$.
  Then,
  \begin{align} \label{E:ExAv_L2}
    (\mathbf{v}_{\mathrm{E}},\bm{\psi})_{L^2(\Omega_\varepsilon)} = \varepsilon(\mathbf{v},\mathbb{L}_0\mathcal{M}_{\varepsilon,\tau}^3\bm{\psi})_{L^2(S^2)}.
  \end{align}
\end{lemma}

\begin{proof}
  By the change of variables \eqref{E:CoV_Thin}, we see that
  \begin{align*}
    (\mathbf{v}_{\mathrm{E}},\bm{\psi})_{L^2(\Omega_\varepsilon)} = \int_{S^2}\int_1^{1+\varepsilon}\{r\mathbf{v}(y)\cdot\bm{\psi}(ry)\}r^2\,dr\,d\mathcal{H}^2(y) = \varepsilon(\mathbf{v},\mathcal{M}_\varepsilon^3\bm{\psi})_{L^2(S^2)}.
  \end{align*}
  Moreover, since $\mathbf{v}\cdot\mathbf{n}=0$ on $S^2$, and since $\mathbf{v}\in\mathcal{H}_0$ and $\mathbb{L}_0$ is the orthogonal projection from $L_\tau^2(S^2)$ onto $\mathcal{H}_0$, it follows that
  \begin{align*}
    (\mathbf{v},\mathcal{M}_\varepsilon^3\bm{\psi})_{L^2(S^2)} = (\mathbf{v},\mathcal{M}_{\varepsilon,\tau}^3\bm{\psi})_{L^2(S^2)} = (\mathbf{v},\mathbb{L}_0\mathcal{M}_{\varepsilon,\tau}^3\bm{\psi})_{L^2(S^2)}.
  \end{align*}
  By these equalities, we obtain \eqref{E:ExAv_L2}.
\end{proof}

Based on \eqref{E:ExAv_L2}, we compare bilinear and trilinear terms on $\Omega_\varepsilon$ involving $\mathbf{v}_{\mathrm{E}}$ and $\bm{\psi}$ with those on $S^2$ involving $\mathbf{v}$ and $\mathbb{L}_0\mathcal{M}_{\varepsilon,\tau}^3\bm{\psi}$.
For the sake of simplicity, we write
\begin{align} \label{E:Def_Le}
  \mathcal{L}_\varepsilon\bm{\psi} := \mathbb{L}_0\mathcal{M}_{\varepsilon,\tau}^3\bm{\psi} \in\mathcal{H}_0, \quad \bm{\psi}\in L^2(\Omega_\varepsilon)^3.
\end{align}
When $\bm{\psi}\in H^1(\Omega_\varepsilon)^3$, we have $\mathcal{L}_\varepsilon\bm{\psi}\in\mathcal{V}_0$ by Lemmas \ref{L:HLH1_S2} and \ref{L:Atan_H1}.

\begin{lemma} \label{L:ExAv_Bi}
  Let $\mathbf{v}\in H_\tau^1(S^2)$ and $\bm{\psi}\in\mathcal{V}_\varepsilon$.
  Then,
  \begin{align} \label{E:ExAv_Bi}
    \begin{aligned}
      &\Bigl|\bigl(\mathbf{D}(\mathbf{v}_{\mathrm{E}}),\mathbf{D}(\bm{\psi})\bigr)_{L^2(\Omega_\varepsilon)}-\varepsilon\bigl(\mathbf{D}_{S^2}(\mathbf{v}),\mathbf{D}_{S^2}(\mathcal{L}_\varepsilon\bm{\psi})\bigr)_{L^2(S^2)}\Bigr| \\
      &\qquad \leq c\varepsilon^{3/2}\|\mathbf{D}_{S^2}(\mathbf{v})\|_{L^2(S^2)}\|\bm{\psi}\|_{H^1(\Omega_\varepsilon)}.
    \end{aligned}
  \end{align}
\end{lemma}

\begin{proof}
  We see by \eqref{E:Mat_Inn}, the symmetry of $\mathbf{D}(\mathbf{v}_{\mathrm{E}})$, and \eqref{E:Ext_str} that
  \begin{align*}
    \mathbf{D}(\mathbf{v}_{\mathrm{E}}):\mathbf{D}(\bm{\psi}) = \mathbf{D}(\mathbf{v}_{\mathrm{E}}):\nabla\bm{\psi} = \overline{\mathbf{D}_{S^2}(\mathbf{v})}:\nabla\bm{\psi} \quad\text{in}\quad \Omega_\varepsilon.
  \end{align*}
  Moreover, since $\mathbf{P}^2=\mathbf{P}^\mathrm{T}=\mathbf{P}$ on $S^2$, we have $\mathbf{D}_{S^2}(\mathbf{v})=\mathbf{P}\mathbf{D}_{S^2}(\mathbf{v})$ on $S^2$ and
  \begin{align*}
    \mathbf{D}(\mathbf{v}_{\mathrm{E}}):\mathbf{D}(\bm{\psi}) = \Bigl[\overline{\mathbf{P}\mathbf{D}_{S^2}(\mathbf{v})}\Bigr]:\nabla\bm{\psi} = \overline{\mathbf{D}_{S^2}(\mathbf{v})}:\Bigl[\overline{\mathbf{P}}\nabla\bm{\psi}\Bigr]
  \end{align*}
  in $\Omega_\varepsilon$.
  By this equality, \eqref{E:CoV_Thin}, and \eqref{E:Ave_TGr}, we find that
  \begin{align*}
    \bigl(\mathbf{D}(\mathbf{v}_{\mathrm{E}}),\mathbf{D}(\bm{\psi})\bigr)_{L^2(\Omega_\varepsilon)} &= \int_{S^2}\int_1^{1+\varepsilon}\Bigl\{[\mathbf{D}_{S^2}(\mathbf{v})](y):[\mathbf{P}(y)\nabla\bm{\psi}(ry)]\Bigr\}r^2\,dr\,d\mathcal{H}^2(y) \\
    &= \varepsilon\bigl(\mathbf{D}_{S^2}(\mathbf{v}),\mathbf{P}\mathcal{M}_\varepsilon^2(\nabla\bm{\psi})\bigr)_{L^2(S^2)} \\
    &= \varepsilon\bigl(\mathbf{D}_{S^2}(\mathbf{v}),\nabla_{S^2}\mathcal{M}_\varepsilon^1\bm{\psi}\bigr)_{L^2(S^2)}.
  \end{align*}
  Again, since $\mathbf{D}_{S^2}(\mathbf{v})=\mathbf{P}(\nabla_{S^2}\mathbf{v})_{\mathrm{S}}\mathbf{P}$ is symmetric and $\mathbf{P}^2=\mathbf{P}^{\mathrm{T}}=\mathbf{P}$ on $S^2$,
  \begin{align*}
    \mathbf{D}_{S^2}(\mathbf{v}):\nabla_{S^2}\mathcal{M}_\varepsilon^1\bm{\psi} &= \mathbf{D}_{S^2}(\mathbf{v}):(\nabla_{S^2}\mathcal{M}_\varepsilon^1\bm{\psi})_{\mathrm{S}} = [\mathbf{P}\mathbf{D}_{S^2}(\mathbf{v})\mathbf{P}]:(\nabla_{S^2}\mathcal{M}_\varepsilon^1\bm{\psi})_{\mathrm{S}} \\
    &= \mathbf{D}_{S^2}(\mathbf{v}):[\mathbf{P}(\nabla_{S^2}\mathcal{M}_\varepsilon^1\bm{\psi})_{\mathrm{S}}\mathbf{P}] = \mathbf{D}_{S^2}(\mathbf{v}):\mathbf{D}_{S^2}(\mathcal{M}_\varepsilon^1\bm{\psi})
  \end{align*}
  on $S^2$.
  Hence, we can write
  \begin{align} \label{Pf_EAB:Decom}
    \begin{aligned}
      &\bigl(\mathbf{D}(\mathbf{v}_{\mathrm{E}}),\mathbf{D}(\bm{\psi})\bigr)_{L^2(\Omega_\varepsilon)}-\varepsilon\bigl(\mathbf{D}_{S^2}(\mathbf{v}),\mathbf{D}_{S^2}(\mathcal{L}_\varepsilon\bm{\psi})\bigr)_{L^2(S^2)} \\
      &\qquad = \varepsilon\bigl(\mathbf{D}_{S^2}(\mathbf{v}),\mathbf{D}_{S^2}(\mathcal{M}_\varepsilon^1\bm{\psi})\bigr)_{L^2(S^2)}-\varepsilon\bigl(\mathbf{D}_{S^2}(\mathbf{v}),\mathbf{D}_{S^2}(\mathcal{L}_\varepsilon\bm{\psi})\bigr)_{L^2(S^2)} \\
      &\qquad = \varepsilon\bigl(\mathbf{D}_{S^2}(\mathbf{v}),\mathbf{A}_1+\mathbf{A}_2+\mathbf{A}_3\bigr)_{L^2(S^2)},
    \end{aligned}
  \end{align}
  where $\mathbf{A}_1,\mathbf{A}_2,\mathbf{A}_3\colon S^2\to\mathbb{R}^{3\times3}$ are given by
  \begin{align*}
    \mathbf{A}_1 &:= \mathbf{D}_{S^2}(\mathcal{M}_\varepsilon^1\bm{\psi})-\mathbf{D}_{S^2}(\mathcal{M}_\varepsilon^3\bm{\psi}), \\
    \mathbf{A}_2 &:= \mathbf{D}_{S^2}(\mathcal{M}_\varepsilon^3\bm{\psi})-\mathbf{D}_{S^2}(\mathcal{M}_{\varepsilon,\tau}^3\bm{\psi}), \\
    \mathbf{A}_3 &:= \mathbf{D}_{S^2}(\mathcal{M}_{\varepsilon,\tau}^3\bm{\psi})-\mathbf{D}_{S^2}(\mathcal{L}_\varepsilon\bm{\psi}).
  \end{align*}
  Since $\mathbf{D}_{S^2}(\mathbf{w})=\mathbf{P}(\nabla_{S^2}\mathbf{w})_{\mathrm{S}}\mathbf{P}$ for $\mathbf{w}\colon S^2\to\mathbb{R}^3$ and $|\mathbf{P}|=\sqrt{2}$ on $S^2$,
  \begin{align*}
    \|\mathbf{A}_1\|_{L^2(S^2)} \leq c\|\nabla_{S^2}\mathcal{M}_\varepsilon^1\bm{\psi}-\nabla_{S^2}\mathcal{M}_\varepsilon^3\bm{\psi}\|_{L^2(S^2)} \leq c\varepsilon^{1/2}\|\nabla\bm{\psi}\|_{L^2(\Omega_\varepsilon)}
  \end{align*}
  by \eqref{E:AvDf_H1}.
  Next, since $\mathbf{A}_2=(\mathcal{M}_\varepsilon^3\bm{\psi}\cdot\mathbf{n})\mathbf{P}$ on $S^2$ by \eqref{E:VT_str}, and since $\bm{\psi}\in\mathcal{V}_\varepsilon$ satisfies $\bm{\psi}\cdot\mathbf{n}_\varepsilon=0$ on $\partial\Omega_\varepsilon$, we can use \eqref{E:Ave_NC} to find that
  \begin{align*}
    \|\mathbf{A}_2\|_{L^2(S^2)} \leq c\|\mathcal{M}_\varepsilon^3\bm{\psi}\cdot\mathbf{n}\|_{L^2(S^2)} \leq c\varepsilon^{1/2}\|\bm{\psi}\|_{H^1(\Omega_\varepsilon)}.
  \end{align*}
  For $\mathbf{A}_3$, we can apply \eqref{E:Ave_HL} to $\bm{\psi}\in\mathcal{V}_\varepsilon$ to get (recall that $\mathcal{L}_\varepsilon=\mathbb{L}_0\mathcal{M}_{\varepsilon,\tau}^3$)
  \begin{align*}
    \|\mathbf{A}_3\|_{L^2(S^2)} \leq c\|\nabla_{S^2}\mathcal{M}_{\varepsilon,\tau}^3\bm{\psi}-\nabla_{S^2}\mathcal{L}_\varepsilon\bm{\psi}\|_{L^2(S^2)} \leq c\varepsilon^{1/2}\|\bm{\psi}\|_{H^1(\Omega_\varepsilon)}.
  \end{align*}
  Applying H\"{o}lder's inequality and the above estimates to \eqref{Pf_EAB:Decom}, we obtain \eqref{E:ExAv_Bi}.
\end{proof}

\begin{lemma} \label{L:ExAv_Tri}
  Let $\mathbf{v}\in\mathcal{V}_0$ and $\bm{\psi}\in\mathcal{V}_\varepsilon$.
  Then,
  \begin{align} \label{E:ExAv_Tri}
    \begin{aligned}
      &\Bigl|\bigl((\mathbf{v}_{\mathrm{E}}\cdot\nabla)\mathbf{v}_{\mathrm{E}},\bm{\psi}\bigr)_{L^2(\Omega_\varepsilon)}-\varepsilon(\nabla_{\mathbf{v}}\mathbf{v},\mathcal{L}_\varepsilon\bm{\psi})_{L^2(S^2)}\Bigr| \\
      &\qquad \leq c\varepsilon^{3/2}\|\mathbf{v}\|_{L^2(S^2)}\|\mathbf{v}\|_{H^1(S^2)}\|\bm{\psi}\|_{H^1(\Omega_\varepsilon)}.
    \end{aligned}
  \end{align}
\end{lemma}

\begin{proof}
  Since $\mathbf{v}\in\mathcal{V}_0$ is tangential on $S^2$, we use \eqref{E:CoV_Thin} and \eqref{E:Ext_CoDe} to get
  \begin{align*}
    \bigl((\mathbf{v}_{\mathrm{E}}\cdot\nabla)\mathbf{v}_{\mathrm{E}},\bm{\psi}\bigr)_{L^2(\Omega_\varepsilon)} &= \int_{S^2}\int_1^{1+\varepsilon}\{r[\nabla_{\mathbf{v}}\mathbf{v}-|\mathbf{v}|^2\mathbf{n}](y)\cdot\bm{\psi}(ry)\}r^2\,dr\,d\mathcal{H}^2(y) \\
    &= \varepsilon(\nabla_{\mathbf{v}}\mathbf{v},\mathcal{M}_\varepsilon^3\bm{\psi})_{L^2(S^2)}-\varepsilon(|\mathbf{v}|^2,\mathcal{M}_\varepsilon^3\bm{\psi}\cdot\bm{n})_{L^2(S^2)}.
  \end{align*}
  Moreover, noting that $\nabla_{\mathbf{v}}\mathbf{v}$ is tangential on $S^2$, we have
  \begin{align*}
    (\nabla_{\mathbf{v}}\mathbf{v},\mathcal{M}_\varepsilon^3\bm{\psi})_{L^2(S^2)} = (\nabla_{\mathbf{v}}\mathbf{v},\mathbf{P}\mathcal{M}_\varepsilon^3\bm{\psi})_{L^2(S^2)} = (\nabla_{\mathbf{v}}\mathbf{v},\mathcal{M}_{\varepsilon,\tau}^3\bm{\psi})_{L^2(S^2)}.
  \end{align*}
  Also, since $\mathbf{v}\in\mathcal{V}_0$ and $\mathcal{M}_{\varepsilon,\tau}^3\bm{\psi},\mathcal{L}_\varepsilon\bm{\psi}\in H_\tau^1(S^2)$, we see by \eqref{E:TrS2_AnS} that
  \begin{align*}
    (\nabla_{\mathbf{v}}\mathbf{v},\mathbf{w})_{L^2(S^2)} = -(\mathbf{v},\nabla_{\mathbf{v}}\mathbf{w})_{L^2(S^2)}, \quad \mathbf{w}=\mathcal{M}_{\varepsilon,\tau}^3\bm{\psi},\mathcal{L}_\varepsilon\bm{\psi}.
  \end{align*}
  Therefore, we can write
  \begin{align} \label{Pf_EAT:Sum}
    \begin{aligned}
      &\bigl((\mathbf{v}_{\mathrm{E}}\cdot\nabla)\mathbf{v}_{\mathrm{E}},\bm{\psi}\bigr)_{L^2(\Omega_\varepsilon)}-\varepsilon(\nabla_{\mathbf{v}}\mathbf{v},\mathcal{L}_\varepsilon\bm{\psi})_{L^2(S^2)} = \varepsilon(K_1+K_2), \\
      &K_1 := -\bigl(\mathbf{v},\nabla_{\mathbf{v}}\mathcal{M}_{\varepsilon,\tau}^3\bm{\psi}\bigr)_{L^2(S^2)}+\bigl(\mathbf{v},\nabla_{\mathbf{v}}\mathcal{L}_\varepsilon\bm{\psi}\bigr)_{L^2(S^2)}, \\
      &K_2 := -(|\mathbf{v}|^2,\mathcal{M}_\varepsilon^3\bm{\psi}\cdot\bm{n})_{L^2(S^2)}.
    \end{aligned}
  \end{align}
  Now, we see by $\nabla_\mathbf{v}\mathbf{w}=\mathbf{P}[(\mathbf{v}\cdot\nabla_{S^2})\mathbf{w}]$ and H\"{o}lder's inequality that
  \begin{align*}
    |K_1| &\leq c\int_{S^2}|\mathbf{v}|^2|\nabla_{S^2}\mathcal{M}_{\varepsilon,\tau}^3\bm{\psi}-\nabla_{S^2}\mathcal{L}_\varepsilon\bm{\psi}|\,d\mathcal{H}^2 \\
    &\leq c\|\mathbf{v}\|_{L^4(S^2)}^2\|\nabla_{S^2}\mathcal{M}_{\varepsilon,\tau}^3\bm{\psi}-\nabla_{S^2}\mathcal{L}_\varepsilon\bm{\psi}\|_{L^2(S^2)}.
  \end{align*}
  Since $\mathcal{L}_\varepsilon=\mathbb{L}_0\mathcal{M}_{\varepsilon,\tau}^3$ and $\bm{\psi}\in\mathcal{V}_\varepsilon$, we further use \eqref{E:Lad_S2} and \eqref{E:Ave_HL} to get
  \begin{align*}
    |K_1| \leq c\varepsilon^{1/2}\|\mathbf{v}\|_{L^2(S^2)}\|\mathbf{v}\|_{H^1(S^2)}\|\bm{\psi}\|_{H^1(S^2)}.
  \end{align*}
  Also, since $\bm{\psi}\cdot\mathbf{n}_\varepsilon=0$ on $\partial\Omega_\varepsilon$ by $\bm{\psi}\in\mathcal{V}_\varepsilon$, we observe that
  \begin{align*}
    |K_2| \leq \|\mathbf{v}\|_{L^4(S^2)}\|\mathcal{M}_\varepsilon^3\bm{\psi}\cdot\bm{n}\|_{L^2(S^2)} \leq c\varepsilon^{1/2}\|\mathbf{v}\|_{L^2(S^2)}\|\mathbf{v}\|_{H^1(S^2)}\|\bm{\psi}\|_{H^1(S^2)}
  \end{align*}
  by H\"{o}lder's inequality, \eqref{E:Lad_S2}, and \eqref{E:Ave_NC}.
  Applying these to \eqref{Pf_EAT:Sum}, we get \eqref{E:ExAv_Tri}.
\end{proof}

Let us also consider the weak time derivative of $\mathbf{v}_{\mathrm{E}}$.
For $\mathcal{V}=\mathcal{V}_0,\mathcal{V}_\varepsilon$, let $\mathbb{E}_T(\mathcal{V})$ be the function space given by \eqref{E:Def_ETV}.

\begin{lemma} \label{L:Ext_Dt}
  For $T>0$, let $\mathbf{v}\in\mathbb{E}_T(\mathcal{V}_0)$.
  Then, $\mathbf{v}_{\mathrm{E}}\in\mathbb{E}_T(\mathcal{V}_\varepsilon)$ and
  \begin{align} \label{E:Ext_Dt}
    \int_0^T\langle\partial_t\mathbf{v}_{\mathrm{E}}(t),\bm{\psi}(t)\rangle_{\mathcal{V}_\varepsilon}\,dt = \varepsilon\int_0^T\langle\partial_t\mathbf{v}(t),\mathcal{L}_\varepsilon\bm{\psi}(t)\rangle_{\mathcal{V}_0}\,dt
  \end{align}
  for all $\bm{\psi}\in L^2(0,T;\mathcal{V}_\varepsilon)$.
\end{lemma}

\begin{proof}
  Since $\mathbf{v}\in L^2(0,T;\mathcal{V}_0)$, we have $\mathbf{v}_{\mathrm{E}}\in L^2(0,T;\mathcal{V}_\varepsilon)$ by Lemmas \ref{L:Ext_Bdd} and \ref{L:Ext_Sol}.

  Next, let $\bm{\psi}\in C_c^1(0,T;\mathcal{V}_\varepsilon)$.
  Then, by $\mathcal{L}_\varepsilon=\mathbb{L}_0\mathcal{M}_{\varepsilon,\tau}^3$, \eqref{E:HLH1_S2}, and \eqref{E:Atan_H1},
  \begin{align*}
    \mathcal{L}_\varepsilon\bm{\psi} \in C_c^1(0,T;\mathcal{V}_0), \quad \|\mathcal{L}_\varepsilon\bm{\psi}\|_{L^2(0,T;H^1(S^2))} \leq c\varepsilon^{-1/2}\|\bm{\psi}\|_{L^2(0,T;H^1(\Omega_\varepsilon))}.
  \end{align*}
  Since $\mathbf{v}\in L^2(0,T;\mathcal{V}_0)$, we see by \eqref{E:ExAv_L2} that
  \begin{align*}
    -\int_0^T\bigl(\mathbf{v}_{\mathrm{E}}(t),\partial_t\bm{\psi}(t)\bigr)_{L^2(\Omega_\varepsilon)}\,dt = -\varepsilon\int_0^T\bigl(\mathbf{v}(t),[\mathcal{L}_\varepsilon(\partial_t\bm{\psi})](t)\bigr)_{L^2(S^2)}\,dt.
  \end{align*}
  Moreover, since $\mathcal{L}_\varepsilon$ is independent of time and $\partial_t\mathbf{v}\in L^2(0,T;\mathcal{V}_0^\ast)$, we can write
  \begin{align} \label{Pf_EAt:Pair}
    \begin{aligned}
      -\int_0^T\bigl(\mathbf{v}_{\mathrm{E}}(t),\partial_t\bm{\psi}(t)\bigr)_{L^2(\Omega_\varepsilon)}\,dt &= -\varepsilon\int_0^T\bigl(\mathbf{v}(t),[\partial_t(\mathcal{L}_\varepsilon\bm{\psi})](t)\bigr)_{L^2(S^2)}\,dt \\
      &= \varepsilon\int_0^T\langle\partial_t\mathbf{v}(t),\mathcal{L}_\varepsilon\bm{\psi}(t)\rangle_{\mathcal{V}_0}\,dt.
    \end{aligned}
  \end{align}
  Hence, H\"{o}lder's inequality and the above estimate for $\mathcal{L}_\varepsilon\bm{\psi}$ imply that
  \begin{align*}
    \left|\int_0^T\bigl(\mathbf{v}_{\mathrm{E}}(t),\partial_t\bm{\psi}(t)\bigr)_{L^2(\Omega_\varepsilon)}\,dt\right| &\leq \varepsilon\|\partial_t\mathbf{v}\|_{L^2(0,T;\mathcal{V}_0^\ast)}\|\mathcal{L}_\varepsilon\bm{\psi}\|_{L^2(0,T;H^1(S^2))} \\
    &\leq c\varepsilon^{1/2}\|\partial_t\mathbf{v}\|_{L^2(0,T;\mathcal{V}_0^\ast)}\|\bm{\psi}\|_{L^2(0,T;H^1(\Omega_\varepsilon))}
  \end{align*}
  for all $\bm{\psi}\in C_c^1(0,T;\mathcal{V}_\varepsilon)$.
  Since this space is dense in $L^2(0,T;\mathcal{V}_\varepsilon)$, we get
  \begin{align*}
    \partial_t\mathbf{v}_{\mathrm{E}} \in [L^2(0,T;\mathcal{V}_\varepsilon)]^\ast = L^2(0,T;\mathcal{V}_\varepsilon^\ast), \quad \mathbf{v}_{\mathrm{E}} \in \mathbb{E}_T(\mathcal{V}_\varepsilon).
  \end{align*}
  Also, we have \eqref{E:Ext_Dt} by \eqref{Pf_EAt:Pair} and a density argument.
\end{proof}

Lastly, we consider the extension of $\mathbf{f}\in\mathcal{V}_0^\ast$.
For $\mathbf{v}\in\mathcal{H}_0$ and $\bm{\psi}\in L^2(\Omega_\varepsilon)^3$, we have
\begin{align*}
  (\bar{\mathbf{v}},\bm{\psi})_{L^2(\Omega_\varepsilon)} = \varepsilon(\mathbf{v},\mathbb{L}_0\mathcal{M}_{\varepsilon,\tau}^2\bm{\psi}), \quad (\mathbf{v}_{\mathrm{E}},\bm{\psi})_{L^2(\Omega_\varepsilon)} = \varepsilon(\mathbf{v},\mathbb{L}_0\mathcal{M}_{\varepsilon,\tau}^3\bm{\psi})_{L^2(S^2)}
\end{align*}
by Lemma \ref{L:ExAv_L2} and its proof.
Based on these relations, we define the constant extension $\bar{\mathbf{f}}$ and the weighted extension $\mathbf{f}_{\mathrm{E}}$ by
\begin{align} \label{E:Def_fext}
  \begin{aligned}
    \langle\bar{\mathbf{f}},\bm{\psi}\rangle_{\mathcal{V}_\varepsilon} &:= \varepsilon\langle\mathbf{f},\mathbb{L}_0\mathcal{M}_{\varepsilon,\tau}^2\bm{\psi}\rangle_{\mathcal{V}_0}, \\
    \langle\mathbf{f}_{\mathrm{E}},\bm{\psi}\rangle_{\mathcal{V}_\varepsilon} &:= \varepsilon\langle\mathbf{f},\mathbb{L}_0\mathcal{M}_{\varepsilon,\tau}^3\bm{\psi}\rangle_{\mathcal{V}_0} = \varepsilon\langle\mathbf{f},\mathcal{L}_\varepsilon\bm{\psi}\rangle_{\mathcal{V}_0}
  \end{aligned}
\end{align}
for $\bm{\psi}\in\mathcal{V}_\varepsilon$.
Let us show that $\bar{\mathbf{f}}$ and $\mathbf{f}_{\mathrm{E}}$ are well-defined and close to each other.

\begin{lemma} \label{L:Func_Ext}
  For $\mathbf{f}\in\mathcal{V}_0^\ast$, we have $\bar{\mathbf{f}},\mathbf{f}_{\mathrm{E}}\in\mathcal{V}_\varepsilon^\ast$ and
  \begin{align} \label{E:Func_Ext}
    \|\bar{\mathbf{f}}\|_{\mathcal{V}_\varepsilon^\ast} \leq c\varepsilon^{1/2}\|\mathbf{f}\|_{\mathcal{V}_0^\ast}, \quad \|\mathbf{f}_{\mathrm{E}}\|_{\mathcal{V}_\varepsilon^\ast} \leq c\varepsilon^{1/2}\|\mathbf{f}\|_{\mathcal{V}_0^\ast}, \quad \|\mathbf{f}_{\mathrm{E}}-\bar{\mathbf{f}}\|_{\mathcal{V}_\varepsilon^\ast} \leq c\varepsilon^{3/2}\|\mathbf{f}\|_{\mathcal{V}_0^\ast}.
  \end{align}
\end{lemma}

\begin{proof}
  For all $\bm{\psi}\in\mathcal{V}_\varepsilon$, we see by \eqref{E:HLH1_S2} and \eqref{E:Atan_H1} that
  \begin{align*}
    |\langle\bar{\mathbf{f}},\bm{\psi}\rangle_{\mathcal{V}_\varepsilon}| = \varepsilon|\langle\mathbf{f},\mathbb{L}_0\mathcal{M}_{\varepsilon,\tau}^2\bm{\psi}\rangle_{\mathcal{V}_0}| \leq \varepsilon\|\mathbf{f}\|_{\mathcal{V}_0^\ast}\|\mathbb{L}_0\mathcal{M}_{\varepsilon,\tau}^2\bm{\psi}\|_{H^1(S^2)} \leq c\varepsilon^{1/2}\|\mathbf{f}\|_{\mathcal{V}_0^\ast}\|\bm{\psi}\|_{H^1(\Omega_\varepsilon)}.
  \end{align*}
  Thus, the first inequality of \eqref{E:Func_Ext} is valid.
  Similarly, we can get the other inequalities by using \eqref{E:HLH1_S2}, \eqref{E:Atan_H1}, and \eqref{E:Atan_Df}.
\end{proof}

\section{Difference estimate for weak solutions} \label{S:Diff}
The goal of this section is to establish Theorems \ref{T:Diff_Est} and \ref{T:Ave_Conv}.

Recall that we write $c$ for a general positive constant independent of $\varepsilon$, $t$, and $\nu$.
Also, we set $\bar{\eta}(x)=\eta(x/|x|)$ for a function $\eta$ on $S^2$.
Let $\mathbb{E}_T(\mathcal{V}_\varepsilon)$ be the function space given by \eqref{E:Def_ETV}.
We define the extensions $\mathbf{v}_{\mathrm{E}}$ and $\mathbf{f}_{\mathrm{E}}$ of a vector field $\mathbf{v}$ on $S^2$ and $\mathbf{f}\in\mathcal{V}_0^\ast$ by \eqref{E:Def_Ext} and \eqref{E:Def_fext}, respectively.
Let $\mathcal{L}_\varepsilon$ be the operator given by \eqref{E:Def_Le}.

Let $\mathbf{u}^\varepsilon$ and $\mathbf{v}$ be weak solutions to \eqref{E:NS_TSS} and \eqref{E:NS_S2} with given data
\begin{align*}
  (\mathbf{u}_0^\varepsilon,\mathbf{f}^\varepsilon) \in \mathcal{H}_\varepsilon \times L_{\mathrm{loc}}^2([0,\infty);\mathcal{V}_\varepsilon^\ast), \quad (\mathbf{v}_0,\mathbf{f}) \in \mathcal{H}_0 \times L_{\mathrm{loc}}^2([0,\infty);\mathcal{V}_0^\ast),
\end{align*}
respectively.
We assume that $\mathbf{u}^\varepsilon$ satisfies the energy inequality \eqref{E:TSS_Ener}.
Note that
\begin{align*}
  \mathbf{v}_{\mathrm{E}} \in \mathbb{E}_T(\mathcal{V}_\varepsilon) \subset C([0,T];\mathcal{H}_\varepsilon), \quad [\mathbf{v}_0]_{\mathrm{E}} = \mathbf{v}_{\mathrm{E}}(0) \in \mathcal{H}_\varepsilon, \quad \mathbf{f}_{\mathrm{E}} \in L_{\mathrm{loc}}^2([0,\infty);\mathcal{V}_\varepsilon^\ast)
\end{align*}
for all $T>0$ by Proposition \ref{P:S2_Cont} and Lemmas \ref{L:BS_Ipl}, \ref{L:Ext_Sol}, \ref{L:Ext_Dt}, and \ref{L:Func_Ext}.

\subsection{Integrability of trilinear terms} \label{SS:Df_Tri}
Let us show the integrability of trilinear terms involving $\mathbf{v}_{\mathrm{E}}$ by using the fact that $\mathbf{v}_{\mathrm{E}}$ is essentially two-dimensional.

\begin{proposition} \label{P:InTr_v2}
  Let $T>0$ and $\bm{\psi}\in L^2(0,T;H^1(\Omega_\varepsilon)^3)$.
  Then,
  \begin{align} \label{E:InTr_v2}
    \begin{aligned}
      &\int_0^T\Bigl|\bigl((\mathbf{v}_{\mathrm{E}}\cdot\nabla)\mathbf{v}_{\mathrm{E}},\bm{\psi}\bigr)_{L^2(\Omega_\varepsilon)}\Bigr|\,dt \\
      &\qquad \leq c\varepsilon^{1/4}\|\mathbf{v}\|_{L^\infty(0,T;L^2(S^2))}\|\mathbf{v}\|_{L^2(0,T;H^1(S^2))}\|\bm{\psi}\|_{L^2(0,T;H^1(\Omega_\varepsilon))}.
    \end{aligned}
  \end{align}
\end{proposition}

\begin{proof}
  Since $\mathbf{v}_{\mathrm{E}}\in L^2(0,T;\mathcal{V}_\varepsilon)$, we have
  \begin{align*}
    \Bigl|\bigl((\mathbf{v}_{\mathrm{E}}\cdot\nabla)\mathbf{v}_{\mathrm{E}},\bm{\psi}\bigr)_{L^2(\Omega_\varepsilon)}\Bigr| = \Bigl|\bigl(\mathbf{v}_{\mathrm{E}},(\mathbf{v}_{\mathrm{E}}\cdot\nabla)\bm{\psi}\bigr)_{L^2(\Omega_\varepsilon)}\Bigr| \leq \|\mathbf{v}_{\mathrm{E}}\|_{L^4(\Omega_\varepsilon)}^2\|\nabla\bm{\psi}\|_{L^2(\Omega_\varepsilon)}
  \end{align*}
  on $(0,T)$ by \eqref{E:TrTS_AnS} and H\"{o}lder's inequality.
  Moreover,
  \begin{align*}
    \|\mathbf{v}_{\mathrm{E}}\|_{L^4(\Omega_\varepsilon)} = \left(\int_{S^2}\int_1^{1+\varepsilon}|\mathbf{v}(y)|^4r^6\,dr\,d\mathcal{H}^2(y)\right)^{1/4} \leq 2^{3/2}\varepsilon^{1/4}\|\mathbf{v}\|_{L^4(S^2)}
  \end{align*}
  by \eqref{E:CoV_Thin}, \eqref{E:Def_Ext}, and $r\leq1+\varepsilon\leq2$.
  These estimates and \eqref{E:Lad_S2} show that
  \begin{align*}
    \Bigl|\bigl((\mathbf{v}_{\mathrm{E}}\cdot\nabla)\mathbf{v}_{\mathrm{E}},\bm{\psi}\bigr)_{L^2(\Omega_\varepsilon)}\Bigr| \leq c\varepsilon^{1/4}\|\mathbf{v}\|_{L^2(S^2)}\|\mathbf{v}\|_{H^1(S^2)}\|\bm{\psi}\|_{H^1(\Omega_\varepsilon)}
  \end{align*}
  on $(0,T)$.
  By this estimate and H\"{o}lder's inequality, we obtain \eqref{E:InTr_v2}.
\end{proof}

\begin{proposition} \label{P:InTr_Ps2}
  Let $T>0$.
  Then,
  \begin{align} \label{E:InTr_Ps2}
      \int_0^T\Bigl|\bigl((\bm{\psi}\cdot\nabla)\bm{\psi},\mathbf{v}_{\mathrm{E}}\bigr)_{L^2(\Omega_\varepsilon)}\Bigr|\,dt \leq c\sigma(\mathbf{v};T)R_\varepsilon(\bm{\psi};T)
  \end{align}
  for all $\bm{\psi}\in L^\infty(0,T;L^2(\Omega_\varepsilon)^3)\cap L^2(0,T;H^1(\Omega_\varepsilon)^3)$, where
  \begin{align} \label{E:Def_sigma}
    \begin{aligned}
      \sigma(\mathbf{v};T) &:= \|\mathbf{v}\|_{L^\infty(0,T;L^2(S^2))}^{1/2}\|\mathbf{v}\|_{L^2(0,T;H^1(S^2))}^{1/2}, \\
      R_\varepsilon(\bm{\psi};T) &:= \|\bm{\psi}\|_{L^\infty(0,T;L^2(\Omega_\varepsilon))}^{1/2}\|\bm{\psi}\|_{L^2(0,T;H^1(\Omega_\varepsilon))}^{3/2}.
    \end{aligned}
  \end{align}
\end{proposition}

\begin{proof}
  We see by \eqref{E:Def_Ext} and $|x|\leq1+\varepsilon\leq 2$ for $x\in\Omega_\varepsilon$ that
  \begin{align*}
    |\mathbf{v}_{\mathrm{E}}(x,t)| = |x|\,|\bar{\mathbf{v}}(x,t)| \leq 2|\bar{\mathbf{v}}(x,t)|, \quad (x,t) \in \Omega_\varepsilon\times(0,\infty).
  \end{align*}
  By this inequality and H\"{o}lder's inequality, we have
  \begin{align*}
    \Bigl|\bigl((\bm{\psi}\cdot\nabla)\bm{\psi},\mathbf{v}_{\mathrm{E}}\bigr)_{L^2(\Omega_\varepsilon)}\Bigr| \leq 2\bigl\|\,|\bar{\mathbf{v}}|\,|\bm{\psi}|\,\bigr\|_{L^2(\Omega_\varepsilon)}\|\nabla\bm{\psi}\|_{L^2(\Omega_\varepsilon)}
  \end{align*}
  on $(0,T)$.
  Moreover, we can apply \eqref{E:Quad_Thin} to get
  \begin{align*}
    \bigl\|\,|\bar{\mathbf{v}}|\,|\bm{\psi}|\,\bigr\|_{L^2(\Omega_\varepsilon)} \leq c\|\mathbf{v}\|_{L^2(S^2)}^{1/2}\|\mathbf{v}\|_{H^1(S^2)}^{1/2}\|\bm{\psi}\|_{L^2(\Omega_\varepsilon)}^{1/2}\|\bm{\psi}\|_{H^1(\Omega_\varepsilon)}^{1/2}.
  \end{align*}
  Combining these estimates, we find that
  \begin{align} \label{Pf_ITr:Est}
    \Bigl|\bigl((\bm{\psi}\cdot\nabla)\bm{\psi},\mathbf{v}_{\mathrm{E}}\bigr)_{L^2(\Omega_\varepsilon)}\Bigr| \leq c\|\mathbf{v}\|_{L^2(S^2)}^{1/2}\|\mathbf{v}\|_{H^1(S^2)}^{1/2}\|\bm{\psi}\|_{L^2(\Omega_\varepsilon)}^{1/2}\|\bm{\psi}\|_{H^1(\Omega_\varepsilon)}^{3/2}
  \end{align}
  on $(0,T)$.
  We get \eqref{E:InTr_Ps2} by this inequality and H\"{o}lder's inequality.
\end{proof}

\subsection{Extension of the weak form on the unit sphere} \label{SS:Df_Ext}
Next, we observe that $\mathbf{v}_{\mathrm{E}}$ is an approximate weak solution to \eqref{E:NS_TSS}.

\begin{proposition} \label{P:WF_Ext}
  For all $T>0$ and $\bm{\psi}\in L^2(0,T;\mathcal{V}_\varepsilon)$, we have
  \begin{multline} \label{E:WF_Ext}
    \int_0^T\langle\partial_t\mathbf{v}_{\mathrm{E}},\bm{\psi}\rangle_{\mathcal{V}_\varepsilon}\,dt+2\nu\int_0^T\bigl(\mathbf{D}(\mathbf{v}_{\mathrm{E}}),\mathbf{D}(\bm{\psi})\bigr)_{L^2(\Omega_\varepsilon)}\,dt \\
    +\int_0^T\bigl((\mathbf{v}_{\mathrm{E}}\cdot\nabla)\mathbf{v}_{\mathrm{E}},\bm{\psi}\bigr)_{L^2(\Omega_\varepsilon)}\,dt = \int_0^T\langle\mathbf{f}_{\mathrm{E}},\bm{\psi}\rangle_{\mathcal{V}_\varepsilon}\,dt+\int_0^T\mathcal{R}_{\mathbf{v}}(\bm{\psi})\,dt.
  \end{multline}
  Here, $\mathcal{R}_{\mathbf{v}}(\bm{\psi})$ is a residual term that is linear in $\bm{\psi}$ and satisfies
  \begin{align} \label{E:WFEx_Res}
      \Bigl|\mathcal{R}_{\mathbf{v}}\bigl(\bm{\psi}(t)\bigr)\Bigr| \leq c\varepsilon^{3/2}\eta_{\mathbf{v}}(t)\|\bm{\psi}(t)\|_{H^1(\Omega_\varepsilon)}
  \end{align}
  for a.a. $t\in(0,T)$, where
  \begin{align} \label{E:Def_eta}
    \eta_{\mathbf{v}}(t) := 2\nu\|[\mathbf{D}_{S^2}(\mathbf{v})](t)\|_{L^2(S^2)}+\|\mathbf{v}(t)\|_{L^2(S^2)}\|\mathbf{v}(t)\|_{H^1(S^2)}, \quad t\in(0,\infty).
  \end{align}
\end{proposition}

\begin{proof}
  Let $\bm{\psi}\in L^2(0,T;\mathcal{V}_\varepsilon)$.
  Since $\mathcal{L}_\varepsilon\bm{\psi}\in L^2(0,T;\mathcal{V}_0)$ by Lemmas \ref{L:HLH1_S2} and \ref{L:Atan_H1}, we can set $\bm{\zeta}=\mathcal{L}_\varepsilon\bm{\psi}$ in \eqref{E:WFS2_Ano}.
  Then, we multiply \eqref{E:WFS2_Ano} by $\varepsilon$, use \eqref{E:Ext_Dt} and \eqref{E:Def_fext}, and define
  \begin{multline*}
    \mathcal{R}_{\mathbf{v}}(\bm{\psi}) := 2\nu\Bigl\{\bigl(\mathbf{D}(\mathbf{v}_{\mathrm{E}}),\mathbf{D}(\bm{\psi})\bigr)_{L^2(\Omega_\varepsilon)}-\varepsilon\bigl(\mathbf{D}_{S^2}(\mathbf{v}),\mathbf{D}_{S^2}(\mathcal{L}_\varepsilon\bm{\psi})\bigr)_{L^2(S^2)}\Bigr\} \\
    +\Bigl\{\bigl((\mathbf{v}_{\mathrm{E}}\cdot\nabla)\mathbf{v}_{\mathrm{E}},\bm{\psi}\bigr)_{L^2(\Omega_\varepsilon)}-\varepsilon(\nabla_{\mathbf{v}}\mathbf{v},\mathcal{L}_\varepsilon\bm{\psi})_{L^2(S^2)}\Bigr\}
  \end{multline*}
  on $(0,T)$ to get \eqref{E:WF_Ext}.
  Clearly, $\mathcal{R}_{\mathbf{v}}(\bm{\psi})$ is linear in $\bm{\psi}$.
  Also, since $\mathbf{v}(t)\in\mathcal{V}_0$ and $\bm{\psi}(t)\in\mathcal{V}_\varepsilon$ for a.a. $t\in(0,T)$, we can apply \eqref{E:ExAv_Bi} and \eqref{E:ExAv_Tri} to find that \eqref{E:WFEx_Res} holds.
\end{proof}

\begin{proposition} \label{P:Ext_Ener}
  For all $t\in[0,\infty)$, we have the energy equality with error term
  \begin{multline} \label{E:Ext_Ener}
    \frac{1}{2}\|\mathbf{v}_{\mathrm{E}}(t)\|_{L^2(\Omega_\varepsilon)}^2+2\nu\int_0^t\|\mathbf{D}(\mathbf{v}_{\mathrm{E}})\|_{L^2(\Omega_\varepsilon)}^2\,ds \\
    = \frac{1}{2}\|[\mathbf{v}_0]_{\mathrm{E}}\|_{L^2(\Omega_\varepsilon)}^2+\int_0^t\langle\mathbf{f}_{\mathrm{E}},\mathbf{v}_{\mathrm{E}}\rangle_{\mathcal{V}_\varepsilon}\,ds+\int_0^t\mathcal{R}_{\mathbf{v}}(\mathbf{v}_{\mathrm{E}})\,ds.
  \end{multline}
\end{proposition}

\begin{proof}
  Let $T>0$.
  Since $\mathbf{v}_{\mathrm{E}}\in\mathbb{E}_T(\mathcal{V}_\varepsilon)$, we can take $\bm{\psi}=\mathbf{v}_{\mathrm{E}}$ in \eqref{E:WF_Ext}.
  Applying \eqref{E:Dt_IbP} and \eqref{E:TrTS_AnS} to the resulting equality, we find that \eqref{E:Ext_Ener} holds with $t$ replaced by $T$.
\end{proof}

\subsection{Mollification in time} \label{SS:Df_Moll}
To proceed the weak-strong uniqueness argument, we would like to set $\bm{\psi}=\mathbf{v}_{\mathrm{E}}$ in \eqref{E:WF_TSS} and $\bm{\psi}=\mathbf{u}^\varepsilon$ in \eqref{E:WF_Ext}.
The latter is allowed by Proposition \ref{P:WF_Ext}, but the former is not.
Thus, we mollify $\mathbf{u}^\varepsilon$ and $\mathbf{v}_{\mathrm{E}}$ in time.

We take a standard mollifier $\rho\in C^\infty(\mathbb{R})$, i.e.
\begin{align*}
  0\leq\rho(t)\leq1, \quad \rho(-t) = \rho(t) \quad (t\in\mathbb{R}), \quad \rho(t) = 0 \quad (|t|\geq1), \quad \int_{-\infty}^\infty\rho(t)\,dt = 1.
\end{align*}
Let $\mathcal{X}$ be a Banach space.
For $u\in L_{\mathrm{loc}}^1(\mathbb{R};\mathcal{X})$ and $k\in\mathbb{N}$, we set
\begin{align*}
  \rho_k(t) := k\rho(kt), \quad u_k(t) := (\rho_k\ast u)(t) = \int_{-\infty}^\infty\rho_k(t-z)u(z)\,dz, \quad t\in\mathbb{R}.
\end{align*}
We extend $\mathbf{u}^\varepsilon$, $\mathbf{v}$, and $\mathbf{v}_{\mathrm{E}}$ to $\mathbb{R}$ by setting zero on $(-\infty,0)$ and apply the above mollification in time.
Then, $\mathbf{u}_k^\varepsilon\in C^\infty(\mathbb{R};\mathcal{V}_\varepsilon)$ and thus $\mathbf{u}_k^\varepsilon\in\mathbb{E}_T(\mathcal{V}_\varepsilon)$ for all $T>0$.

\begin{proposition} \label{P:WF_Moll}
  Fix any $0<s<T$, and let
  \begin{align*}
    \chi_{(s,T)}(t) :=
    \begin{cases}
      1, &t\in(s,T), \\
      0, &t\in\mathbb{R}\setminus(s,T),
    \end{cases}
    \quad \bm{\psi}_{s,T}^k(t) := [\rho_k\ast(\chi_{(s,T)}\mathbf{v}_{\mathrm{E}})](t), \quad t\in\mathbb{R}
  \end{align*}
  for $k\in\mathbb{N}$ such that $s-1/k>0$ and $T+1/k<2T$.
  Then, we have
  \begin{multline} \label{E:WF_Moll}
    \int_s^T\langle\partial_t\mathbf{u}_k^\varepsilon,\mathbf{v}_{\mathrm{E}}\rangle_{\mathcal{V}_\varepsilon}\,dt+2\nu\int_0^{2T}\bigl(\mathbf{D}(\mathbf{u}^\varepsilon),\mathbf{D}(\bm{\psi}_{s,T}^k)\bigr)_{L^2(\Omega_\varepsilon)}\,dt \\
    +\int_0^{2T}\bigl((\mathbf{u}^\varepsilon\cdot\nabla)\mathbf{u}^\varepsilon,\bm{\psi}_{s,T}^k\bigr)_{L^2(\Omega_\varepsilon)}\,dt = \int_0^{2T}\langle\mathbf{f}^\varepsilon,\bm{\psi}_{s,T}^k\rangle_{\mathcal{V}_\varepsilon}\,dt.
  \end{multline}
\end{proposition}

\begin{proof}
  By $s-1/k>0$ and $T+1/k<2T$, we have $\bm{\psi}_{s,T}^k\in C_c^\infty(0,2T;\mathcal{V}_\varepsilon)$.
  Also,
  \begin{align*}
    &\int_0^\infty\bigl(\mathbf{u}^\varepsilon(t),\partial_t\bm{\psi}_{s,T}^k(t)\bigr)_{L^2(\Omega_\varepsilon)}\,dt = \int_{-\infty}^\infty\bigl(\mathbf{u}^\varepsilon(t),\partial_t\bm{\psi}_{s,T}^k(t)\bigr)_{L^2(\Omega_\varepsilon)}\,dt \\
    &\qquad = \int_{-\infty}^\infty\left(\int_{-\infty}^\infty\partial_t[\rho_k(t-z)]\bigl(\mathbf{u}^\varepsilon(t),[\chi_{(s,T)}\mathbf{v}_{\mathrm{E}}](z)\bigr)_{L^2(\Omega_\varepsilon)}\,dz\right)\,dt \\
    &\qquad = -\int_{-\infty}^\infty\bigl(\partial_z\mathbf{u}_k^\varepsilon(z),[\chi_{(s,T)}\mathbf{v}_{\mathrm{E}}](z)\bigr)_{L^2(\Omega_\varepsilon)}\,dz
  \end{align*}
  by Fubini's theorem and $\partial_t[\rho_k(t-z)]=-\partial_z[\rho_k(z-t)]$, which follows from
  \begin{align*}
    \rho_k(t-z) = \rho_k(z-t), \quad \partial_t[\rho_k(t-z)] = -\rho_k'(z-t) = -\partial_z[\rho_k(z-t)].
  \end{align*}
  Since $\chi_{(s,T)}=0$ on $\mathbb{R}\setminus(s,T)$ and $\mathcal{H}_\varepsilon\subset\mathcal{V}_\varepsilon^\ast$, we can further write
  \begin{align*}
    \int_0^\infty(\mathbf{u}^\varepsilon,\partial_t\bm{\psi}_{s,T}^k)_{L^2(\Omega_\varepsilon)}\,dt = -\int_s^T(\partial_t\mathbf{u}_k^\varepsilon,\mathbf{v}_{\mathrm{E}})_{L^2(\Omega_\varepsilon)}\,dt = -\int_s^T\langle\partial_t\mathbf{u}_k^\varepsilon,\mathbf{v}_{\mathrm{E}}\rangle_{\mathcal{V}_\varepsilon}\,dt.
  \end{align*}
  We set $\bm{\psi}=\bm{\psi}_{s,T}^k$ in \eqref{E:WF_TSS} and use this equality and $\bm{\psi}_{s,T}^k(0)=\mathbf{0}$ to get \eqref{E:WF_Moll}.
\end{proof}

Before combining \eqref{E:WF_Ext} and \eqref{E:WF_Moll}, we give the convergence results as $k\to\infty$.

\begin{proposition} \label{P:MC_Psi}
  For all $0<s<T$, we have
  \begin{align}
    \lim_{k\to\infty}\int_0^{2T}\bigl(\mathbf{D}(\mathbf{u}^\varepsilon),\mathbf{D}(\bm{\psi}_{s,T}^k)\bigr)_{L^2(\Omega_\varepsilon)}\,dt &= \int_s^T\bigl(\mathbf{D}(\mathbf{u}^\varepsilon),\mathbf{D}(\mathbf{v}_{\mathrm{E}})\bigr)_{L^2(\Omega_\varepsilon)}\,dt, \label{E:MC_PsBi} \\
    \lim_{k\to\infty}\int_0^{2T}\bigl((\mathbf{u}^\varepsilon\cdot\nabla)\mathbf{u}^\varepsilon,\bm{\psi}_{s,T}^k\bigr)_{L^2(\Omega_\varepsilon)}\,dt &= \int_s^T\bigl((\mathbf{u}^\varepsilon\cdot\nabla)\mathbf{u}^\varepsilon,\mathbf{v}_{\mathrm{E}}\bigr)_{L^2(\Omega_\varepsilon)}\,dt, \label{E:MC_PsTri} \\
    \lim_{k\to\infty}\int_0^{2T}\langle\mathbf{f}^\varepsilon,\bm{\psi}_{s,T}^k\rangle_{\mathcal{V}_\varepsilon}\,dt &= \int_s^T\langle\mathbf{f}^\varepsilon,\mathbf{v}_{\mathrm{E}}\rangle_{\mathcal{V}_\varepsilon}\,dt. \label{E:MC_PsFor}
  \end{align}
\end{proposition}

\begin{proof}
  Since $\chi_{(s,T)}\mathbf{v}_{\mathrm{E}}\in L^2(\mathbb{R};\mathcal{V}_\varepsilon)$, its mollification $\bm{\psi}_{s,T}^k$ satisfies
  \begin{align*}
    \lim_{k\to\infty}\|\bm{\psi}_{s,T}^k-\chi_{(s,T)}\mathbf{v}_{\mathrm{E}}\|_{L^2(\mathbb{R};H^1(\Omega_\varepsilon))} = 0.
  \end{align*}
  We have \eqref{E:MC_PsBi} and \eqref{E:MC_PsFor} by this result and $\chi_{(s,T)}=0$ on $\mathbb{R}\setminus(s,T)$.

  Next, we prove \eqref{E:MC_PsTri}.
  To this end, we observe that the cut-off and mollification in time commute with the extension \eqref{E:Def_Ext} in the spatial variable, i.e.
  \begin{align*}
    \bm{\psi}_{s,T}^k = \rho_k\ast(\chi_{(s,T)}\mathbf{v}_{\mathrm{E}}) = [\rho_k\ast(\chi_{(s,T)}\mathbf{v})]_{\mathrm{E}}, \quad \chi_{(s,T)}\mathbf{v}_{\mathrm{E}} = [\chi_{(s,T)}\mathbf{v}]_{\mathrm{E}}.
  \end{align*}
  Thus, $\bm{\psi}_{s,T}^k-\chi_{(s,T)}\mathbf{v}_{\mathrm{E}}=[\bm{\zeta}_k]_{\mathrm{E}}$ is the extension of
  \begin{align*}
    \bm{\zeta}_k := \rho_k\ast(\chi_{(s,T)}\mathbf{v})-\chi_{(s,T)}\mathbf{v} \in L^\infty(\mathbb{R};\mathcal{H}_0)\cap L^2(\mathbb{R};\mathcal{V}_0),
  \end{align*}
  and we can show as in the proof of \eqref{E:InTr_Ps2} that
  \begin{align} \label{Pf_MCP:Dif}
    \begin{aligned}
      &\left|\int_0^{2T}\bigl((\mathbf{u}^\varepsilon\cdot\nabla)\mathbf{u}^\varepsilon,\bm{\psi}_{s,T}^k\bigr)_{L^2(\Omega_\varepsilon)}\,dt-\int_0^{2T}\bigl((\mathbf{u}^\varepsilon\cdot\nabla)\mathbf{u}^\varepsilon,\chi_{(s,T)}\mathbf{v}_{\mathrm{E}}\bigr)_{L^2(\Omega_\varepsilon)}\,dt\right| \\
      &\qquad = \left|\int_0^{2T}\bigl((\mathbf{u}^\varepsilon\cdot\nabla)\mathbf{u}^\varepsilon,[\bm{\zeta}_k]_{\mathrm{E}}\bigr)_{L^2(\Omega_\varepsilon)}\,dt\right| \leq c\sigma(\bm{\zeta}_k;2T)R_\varepsilon(\mathbf{u}^\varepsilon;2T),
    \end{aligned}
  \end{align}
  where $\sigma(\bm{\zeta}_k;2T)$ and $R_\varepsilon(\mathbf{u}^\varepsilon;2T)$ are given by \eqref{E:Def_sigma}.
  Moreover,
  \begin{gather*}
    \|\bm{\zeta}_k\|_{L^\infty(\mathbb{R};L^2(S^2))} \leq 2\|\chi_{(s,T)}\mathbf{v}\|_{L^\infty(\mathbb{R};L^2(S^2))}, \quad \lim_{k\to\infty}\|\bm{\zeta}_k\|_{L^2(\mathbb{R};H^1(S^2))} = 0
  \end{gather*}
  by $\chi_{(s,T)}\mathbf{v}\in L^\infty(\mathbb{R};\mathcal{H}_0)\cap L^2(\mathbb{R};\mathcal{V}_0)$, and thus
  \begin{align*}
    0 \leq \sigma(\bm{\zeta}_k;2T) \leq \|\bm{\zeta}_k\|_{L^\infty(\mathbb{R};L^2(S^2))}^{1/2}\|\bm{\zeta}_k\|_{L^2(\mathbb{R};H^1(S^2))}^{1/2} \to 0 \quad\text{as}\quad k\to\infty.
  \end{align*}
  Applying this result and $\chi_{(s,T)}=0$ on $\mathbb{R}\setminus(s,T)$ to \eqref{Pf_MCP:Dif}, we get \eqref{E:MC_PsTri}.
\end{proof}

\begin{proposition} \label{P:MC_ue}
  For all $0<s<T$, we have
  \begin{align}
    \lim_{k\to\infty}\int_s^T\bigl(\mathbf{D}(\mathbf{v}_{\mathrm{E}}),\mathbf{D}(\mathbf{u}_k^\varepsilon)\bigr)_{L^2(\Omega_\varepsilon)}\,dt &= \int_s^T\bigl(\mathbf{D}(\mathbf{v}_{\mathrm{E}}),\mathbf{D}(\mathbf{u}^\varepsilon)\bigr)_{L^2(\Omega_\varepsilon)}\,dt, \label{E:MC_ueBi} \\
    \lim_{k\to\infty}\int_s^T\bigl((\mathbf{v}_{\mathrm{E}}\cdot\nabla)\mathbf{v}_{\mathrm{E}},\mathbf{u}_k^\varepsilon\bigr)_{L^2(\Omega_\varepsilon)}\,dt &= \int_s^T\bigl((\mathbf{v}_{\mathrm{E}}\cdot\nabla)\mathbf{v}_{\mathrm{E}},\mathbf{u}^\varepsilon\bigr)_{L^2(\Omega_\varepsilon)}\,dt, \label{E:MC_ueTri} \\
    \lim_{k\to\infty}\int_s^T\langle\mathbf{f}_{\mathrm{E}},\mathbf{u}_k^\varepsilon\rangle_{\mathcal{V}_\varepsilon}\,dt &= \int_s^T\langle\mathbf{f}_{\mathrm{E}},\mathbf{u}^\varepsilon\rangle_{\mathcal{V}_\varepsilon}\,dt, \label{E:MC_ueFor} \\
    \lim_{k\to\infty}\int_s^T\mathcal{R}_{\mathbf{v}}(\mathbf{u}_k^\varepsilon)\,dt &= \int_s^T\mathcal{R}_{\mathbf{v}}(\mathbf{u}^\varepsilon)\,dt. \label{E:MC_ueRes}
  \end{align}
  Here, $\mathcal{R}_{\mathbf{v}}(\bm{\psi})$ is given in Proposition \ref{P:WF_Ext}.
\end{proposition}

\begin{proof}
  Since $\mathbf{u}^\varepsilon\in L_{\mathrm{loc}}^2(\mathbb{R};\mathcal{V}_\varepsilon)$ (by zero extension) and $\mathbf{u}_k^\varepsilon=\rho_k\ast\mathbf{u}^\varepsilon$,
  \begin{align*}
    \lim_{k\to\infty}\|\mathbf{u}_k^\varepsilon-\mathbf{u}^\varepsilon\|_{L^2(s,T;H^1(\Omega_\varepsilon))} = 0.
  \end{align*}
  This result implies \eqref{E:MC_ueBi} and \eqref{E:MC_ueFor}.
  Also, we can show that
  \begin{align*}
    &\left|\int_s^T\bigl((\mathbf{v}_{\mathrm{E}}\cdot\nabla)\mathbf{v}_{\mathrm{E}},\mathbf{u}_k^\varepsilon\bigr)_{L^2(\Omega_\varepsilon)}\,dt-\int_s^T\bigl((\mathbf{v}_{\mathrm{E}}\cdot\nabla)\mathbf{v}_{\mathrm{E}},\mathbf{u}^\varepsilon\bigr)_{L^2(\Omega_\varepsilon)}\,dt\right| \\
    &\qquad \leq c\varepsilon^{1/4}\|\mathbf{v}\|_{L^\infty(s,T;L^2(S^2))}\|\mathbf{v}\|_{L^2(s,T;H^1(S^2))}\|\mathbf{u}_k^\varepsilon-\mathbf{u}^\varepsilon\|_{L^2(s,T;H^1(\Omega_\varepsilon))}
  \end{align*}
  as in the proof of \eqref{E:InTr_v2}, and we see by \eqref{E:WFEx_Res} and H\"{o}lder's inequality that
  \begin{align*}
    &\left|\int_s^T\mathcal{R}_{\mathbf{v}}(\mathbf{u}_k^\varepsilon)\,dt-\int_s^T\mathcal{R}_{\mathbf{v}}(\mathbf{u}^\varepsilon)\,dt\right| = \left|\int_s^T\mathcal{R}_{\mathbf{v}}(\mathbf{u}_k^\varepsilon-\mathbf{u}^\varepsilon)\,dt\right| \\
    &\qquad \leq c\varepsilon^{3/2}\Bigl(\nu+\|\mathbf{v}\|_{L^\infty(s,T;L^2(S^2))}\Bigr)\|\mathbf{v}\|_{L^2(s,T;H^1(S^2))}\|\mathbf{u}_k^\varepsilon-\mathbf{u}^\varepsilon\|_{L^2(s,T;H^1(\Omega_\varepsilon))}.
  \end{align*}
  Recall that $\mathcal{R}_{\mathbf{v}}(\bm{\psi})$ is linear in $\bm{\psi}$.
  By these inequalities and the strong convergence of $\mathbf{u}_k^\varepsilon$ towards $\mathbf{u}^\varepsilon$ as $k\to\infty$, we obtain \eqref{E:MC_ueTri} and \eqref{E:MC_ueRes}.
\end{proof}

\begin{proposition} \label{P:MC_Trace}
  For all $t\in(0,\infty)$, we have
  \begin{align} \label{E:MC_Trace}
    \lim_{k\to\infty}\bigl(\mathbf{u}_k^\varepsilon(t),\mathbf{v}_{\mathrm{E}}(t)\bigr)_{L^2(\Omega_\varepsilon)} = \bigl(\mathbf{u}^\varepsilon(t),\mathbf{v}_{\mathrm{E}}(t)\bigr)_{L^2(\Omega_\varepsilon)}.
  \end{align}
\end{proposition}

\begin{proof}
  Fix any $\bm{\psi}\in\mathcal{H}_\varepsilon$.
  For each $t\in\mathbb{R}$, we see that
  \begin{align*}
    \bigl(\mathbf{u}_k^\varepsilon(t),\bm{\psi}\bigr)_{L^2(\Omega_\varepsilon)} = (\rho_k\ast\varphi^\varepsilon)(t), \quad \varphi^\varepsilon(t) := \bigl(\mathbf{u}^\varepsilon(t),\bm{\psi}\bigr)_{L^2(\Omega_\varepsilon)}.
  \end{align*}
  Since $\mathbf{u}^\varepsilon\in C_{\mathrm{weak}}(0,\infty;\mathcal{H}_\varepsilon)$ by Proposition \ref{P:TSS_WeCo}, we have $\varphi^\varepsilon\in C(0,\infty)$.
  Thus,
  \begin{align*}
    \lim_{k\to\infty}(\rho_k\ast\varphi^\varepsilon)(t) = \varphi^\varepsilon(t), \quad\text{i.e.}\quad \lim_{k\to\infty}\bigl(\mathbf{u}_k^\varepsilon(t),\bm{\psi}\bigr)_{L^2(\Omega_\varepsilon)} = \bigl(\mathbf{u}^\varepsilon(t),\bm{\psi}\bigr)_{L^2(\Omega_\varepsilon)}
  \end{align*}
  for each $t\in(0,\infty)$.
  In particular, setting $\bm{\psi}=\mathbf{v}_{\mathrm{E}}(t)$, we obtain \eqref{E:MC_Trace}.
\end{proof}

Now, let us combine \eqref{E:WF_Ext} and \eqref{E:WF_Moll} to derive the following crucial formula.

\begin{proposition} \label{P:uv_pair}
  For all $T>0$, we have
  \begin{align} \label{E:uv_pair}
    \begin{aligned}
      &\bigl(\mathbf{u}^\varepsilon(T),\mathbf{v}_{\mathrm{E}}(T)\bigr)_{L^2(\Omega_\varepsilon)}+4\nu\int_0^T\bigl(\mathbf{D}(\mathbf{u}^\varepsilon),\mathbf{D}(\mathbf{v}_{\mathrm{E}})\bigr)_{L^2(\Omega_\varepsilon)}\,dt \\
      &\qquad +\int_0^T\Bigl\{\bigl((\mathbf{u}^\varepsilon\cdot\nabla)\mathbf{u}^\varepsilon,\mathbf{v}_{\mathrm{E}}\bigr)_{L^2(\Omega_\varepsilon)}+\bigl((\mathbf{v}_{\mathrm{E}}\cdot\nabla)\mathbf{v}_{\mathrm{E}},\mathbf{u}^\varepsilon\bigr)_{L^2(\Omega_\varepsilon)}\Bigr\}\,dt \\
      &\qquad\qquad = (\mathbf{u}_0^\varepsilon,[\mathbf{v}_0]_{\mathrm{E}})_{L^2(\Omega_\varepsilon)}+\int_0^T\{\langle\mathbf{f}^\varepsilon,\mathbf{v}_{\mathrm{E}}\rangle_{\mathcal{V}_\varepsilon}+\langle\mathbf{f}_{\mathrm{E}},\mathbf{u}^\varepsilon\rangle_{\mathcal{V}_\varepsilon}\}\,dt+\int_0^T\mathcal{R}_{\mathbf{v}}(\mathbf{u}^\varepsilon)\,dt.
    \end{aligned}
  \end{align}
\end{proposition}

\begin{proof}
  Let $0<s<T$ and $\mathbf{u}_k^\varepsilon=\rho_k\ast\mathbf{u}^\varepsilon$ such that $s-1/k>0$ and $T+1/k<2T$.
  We can take the test function $\bm{\psi}=\chi_{(s,T)}\mathbf{u}_k^\varepsilon$ in \eqref{E:WF_Ext}, since $\mathbf{u}_k^\varepsilon \in C^\infty(\mathbb{R};\mathcal{V}_\varepsilon)$.
  Then,
  \begin{multline} \label{Pf_Pa:Ext}
    \int_s^T\langle\partial_t\mathbf{v}_{\mathrm{E}},\mathbf{u}_k^\varepsilon\rangle_{\mathcal{V}_\varepsilon}\,dt+2\nu\int_s^T\bigl(\mathbf{D}(\mathbf{v}_{\mathrm{E}}),\mathbf{D}(\mathbf{u}_k^\varepsilon)\bigr)_{L^2(\Omega_\varepsilon)}\,dt \\
    +\int_s^T\bigl((\mathbf{v}_{\mathrm{E}}\cdot\nabla)\mathbf{v}_{\mathrm{E}},\mathbf{u}_k^\varepsilon\bigr)_{L^2(\Omega_\varepsilon)}\,dt = \int_s^T\langle\mathbf{f}_{\mathrm{E}},\mathbf{u}_k^\varepsilon\rangle_{\mathcal{V}_\varepsilon}\,dt+\int_s^T\mathcal{R}_{\mathbf{v}}(\mathbf{u}_k^\varepsilon)\,dt
  \end{multline}
  by $\chi_{(s,T)}=0$ on $\mathbb{R}\setminus(s,T)$.
  We take the sum of \eqref{E:WF_Moll} and \eqref{Pf_Pa:Ext} and use
  \begin{align*}
    \int_s^T\Bigl\{\langle\partial_t\mathbf{u}_k^\varepsilon,\mathbf{v}_{\mathrm{E}}\rangle_{\mathcal{V}_\varepsilon}+\langle\partial_t\mathbf{v}_{\mathrm{E}},\mathbf{u}_k^\varepsilon\rangle_{\mathcal{V}_\varepsilon}\Bigr\}\,dt = \bigl(\mathbf{u}_k^\varepsilon(T),\mathbf{v}_{\mathrm{E}}(T)\bigr)_{L^2(\Omega_\varepsilon)}-\bigl(\mathbf{u}_k^\varepsilon(s),\mathbf{v}_{\mathrm{E}}(s)\bigr)_{L^2(\Omega_\varepsilon)}
  \end{align*}
  by $\mathbf{u}_k^\varepsilon,\mathbf{v}_{\mathrm{E}}\in\mathbb{E}_T(\mathcal{V}_\varepsilon)$ and \eqref{E:Dt_IbP} (recall that $\mathbf{u}_k^\varepsilon$ is the mollification of $\mathbf{u}^\varepsilon$ in time).
  After that, we send $k\to\infty$ and use Propositions \ref{P:MC_Psi}--\ref{P:MC_Trace} to get
  \begin{align*}
    &\bigl(\mathbf{u}^\varepsilon(T),\mathbf{v}_{\mathrm{E}}(T)\bigr)_{L^2(\Omega_\varepsilon)}+4\nu\int_s^T\bigl(\mathbf{D}(\mathbf{u}^\varepsilon),\mathbf{D}(\mathbf{v}_{\mathrm{E}})\bigr)_{L^2(\Omega_\varepsilon)}\,dt \\
    &\qquad +\int_s^T\Bigl\{\bigl((\mathbf{u}^\varepsilon\cdot\nabla)\mathbf{u}^\varepsilon,\mathbf{v}_{\mathrm{E}}\bigr)_{L^2(\Omega_\varepsilon)}+\bigl((\mathbf{v}_{\mathrm{E}}\cdot\nabla)\mathbf{v}_{\mathrm{E}},\mathbf{u}^\varepsilon\bigr)_{L^2(\Omega_\varepsilon)}\Bigr\}\,dt \\
    &\qquad\qquad = \bigl(\mathbf{u}^\varepsilon(s),\mathbf{v}_{\mathrm{E}}(s)\bigr)_{L^2(\Omega_\varepsilon)}+\int_s^T\{\langle\mathbf{f}^\varepsilon,\mathbf{v}_{\mathrm{E}}\rangle_{\mathcal{V}_\varepsilon}+\langle\mathbf{f}_{\mathrm{E}},\mathbf{u}^\varepsilon\rangle_{\mathcal{V}_\varepsilon}\}\,dt+\int_s^T\mathcal{R}_{\mathbf{v}}(\mathbf{u}^\varepsilon)\,dt.
  \end{align*}
  Now, let $s\to0$ in this equality.
  We recall that
  \begin{align*}
    \mathbf{u}^\varepsilon\in C_{\mathrm{weak}}([0,T];\mathcal{H}_\varepsilon), \quad \mathbf{v}_{\mathrm{E}} \in \mathbb{E}_T(\mathcal{V}_\varepsilon) \subset C([0,T];\mathcal{H}_\varepsilon)
  \end{align*}
  by Propositions \ref{P:TSS_WeCo} and \ref{P:S2_Cont} and Lemmas \ref{L:BS_Ipl} and \ref{L:Ext_Dt}.
  Thus, by a standard argument of the weak and strong convergence, we can show that
  \begin{align*}
    \lim_{s\to0}\bigl(\mathbf{u}_k^\varepsilon(s),\mathbf{v}_{\mathrm{E}}(s)\bigr)_{L^2(\Omega_\varepsilon)} = (\mathbf{u}_0^\varepsilon,[\mathbf{v}_0]_{\mathrm{E}})_{L^2(\Omega_\varepsilon)}.
  \end{align*}
  The other integral terms converge to the integrals over $(0,T)$ by the dominated convergence theorem, since the integrands are in $L^1(0,T)$ (see \eqref{E:InTr_v2}, \eqref{E:InTr_Ps2}, and \eqref{E:WFEx_Res} for the trilinear and residual terms).
  Hence, we obtain \eqref{E:uv_pair}.
\end{proof}

\subsection{Estimates for the difference of solutions} \label{SS:Df_Prf}
We are ready to estimate the difference of $\mathbf{u}^\varepsilon$ and $\mathbf{v}_{\mathrm{E}}$, and to prove the difference estimate \eqref{E:Diff_Est} for $\mathbf{u}^\varepsilon$ and $\mathbf{v}$.

\begin{theorem} \label{T:Df_vE}
  Under the assumptions of Theorem \ref{T:Diff_Est}, let
  \begin{align} \label{E:Def_weps}
    \mathbf{w}^\varepsilon := \mathbf{u}^\varepsilon-\mathbf{v}_{\mathrm{E}}, \quad \mathbf{w}_0^\varepsilon := \mathbf{u}_0^\varepsilon-[\mathbf{v}_0]_{\mathrm{E}}, \quad \mathbf{h}^\varepsilon:=\mathbf{f}^\varepsilon-\mathbf{f}_{\mathrm{E}}.
  \end{align}
  Then, for all $t\geq0$, we have
  \begin{multline} \label{E:Df_vE}
    \|\mathbf{w}^\varepsilon(t)\|_{L^2(\Omega_\varepsilon)}^2+\frac{\nu}{2}\int_0^t\|\nabla\mathbf{w}^\varepsilon\|_{L^2(\Omega_\varepsilon)}^2\,ds \\
    \leq F_{\mathbf{v}}(t)\left\{\|\mathbf{w}_0^\varepsilon\|_{L^2(\Omega_\varepsilon)}^2+\frac{c}{\nu}\int_0^t\Bigl(\|\mathbf{h}^\varepsilon\|_{\mathcal{V}_\varepsilon^\ast}^2+\varepsilon^3\eta_{\mathbf{v}}^2\Bigr)\,ds\right\},
  \end{multline}
  where $F_{\mathbf{v}}$ and $\eta_{\mathbf{v}}$ are given by \eqref{E:Def_FG} and \eqref{E:Def_eta}, respectively.
\end{theorem}

\begin{proof}
  We take the sum of
  \begin{align*}
    \bigl((\mathbf{u}^\varepsilon\cdot\nabla)\mathbf{u}^\varepsilon,\mathbf{v}_{\mathrm{E}}\bigr)_{L^2(\Omega_\varepsilon)} &= \bigl((\mathbf{u}^\varepsilon\cdot\nabla)\mathbf{w}^\varepsilon,\mathbf{v}_{\mathrm{E}}\bigr)_{L^2(\Omega_\varepsilon)}+\bigl((\mathbf{u}^\varepsilon\cdot\nabla)\mathbf{v}_{\mathrm{E}},\mathbf{v}_{\mathrm{E}}\bigr)_{L^2(\Omega_\varepsilon)}, \\
    \bigl((\mathbf{v}_{\mathrm{E}}\cdot\nabla)\mathbf{v}_{\mathrm{E}},\mathbf{u}^\varepsilon\bigr)_{L^2(\Omega_\varepsilon)} &= \bigl((\mathbf{v}_{\mathrm{E}}\cdot\nabla)\mathbf{v}_{\mathrm{E}},\mathbf{w}^\varepsilon\bigr)_{L^2(\Omega_\varepsilon)}+\bigl((\mathbf{v}_{\mathrm{E}}\cdot\nabla)\mathbf{v}_{\mathrm{E}},\mathbf{v}_{\mathrm{E}}\bigr)_{L^2(\Omega_\varepsilon)}
  \end{align*}
  and apply $\bigl((\bm{\psi}\cdot\nabla)\mathbf{v}_{\mathrm{E}},\mathbf{v}_{\mathrm{E}}\bigr)_{L^2(\Omega_\varepsilon)}=0$ for $\bm{\psi}=\mathbf{u}^\varepsilon,\mathbf{v}_{\mathrm{E}}$ and
  \begin{align*}
    \bigl((\mathbf{v}_{\mathrm{E}}\cdot\nabla)\mathbf{v}_{\mathrm{E}},\mathbf{w}^\varepsilon\bigr)_{L^2(\Omega_\varepsilon)} = -\bigl(\mathbf{v}_{\mathrm{E}},(\mathbf{v}_{\mathrm{E}}\cdot\nabla)\mathbf{w}^\varepsilon\bigr)_{L^2(\Omega_\varepsilon)} = -\bigl((\mathbf{v}_{\mathrm{E}}\cdot\nabla)\mathbf{w}^\varepsilon,\mathbf{v}_{\mathrm{E}}\bigr)_{L^2(\Omega_\varepsilon)}
  \end{align*}
  by $\mathbf{u}^\varepsilon,\mathbf{v}_{\mathrm{E}}\in L_{\mathrm{loc}}^2([0,\infty);\mathcal{V}_\varepsilon)$ and \eqref{E:TrTS_AnS} to get
  \begin{align} \label{Pf_DvE:TrSum}
    \bigl((\mathbf{u}^\varepsilon\cdot\nabla)\mathbf{u}^\varepsilon,\mathbf{v}_{\mathrm{E}}\bigr)_{L^2(\Omega_\varepsilon)}+\bigl((\mathbf{v}_{\mathrm{E}}\cdot\nabla)\mathbf{v}_{\mathrm{E}},\mathbf{u}^\varepsilon\bigr)_{L^2(\Omega_\varepsilon)} = \bigl((\mathbf{w}^\varepsilon\cdot\nabla)\mathbf{w}^\varepsilon,\mathbf{v}_{\mathrm{E}}\bigr)_{L^2(\Omega_\varepsilon)}
  \end{align}
  on $(0,\infty)$.
  Since we assume that $\mathbf{u}^\varepsilon$ satisfies the energy inequality \eqref{E:TSS_Ener}, and since $\mathbf{v}_{\mathrm{E}}$ satisfies the energy equality \eqref{E:Ext_Ener} by Proposition \ref{P:Ext_Ener}, we take the sum of them and then subtract \eqref{E:uv_pair} (with $T$ replaced by $t$) from the resulting inequality.
  Then,
  \begin{align} \label{Pf_DvE:weps}
    \begin{aligned}
      &\frac{1}{2}\|\mathbf{w}^\varepsilon(t)\|_{L^2(\Omega_\varepsilon)}^2+2\nu\int_0^t\|\mathbf{D}(\mathbf{w}^\varepsilon)\|_{L^2(\Omega_\varepsilon)}^2\,ds \\
      &\qquad \leq \frac{1}{2}\|\mathbf{w}_0^\varepsilon\|_{L^2(\Omega_\varepsilon)}^2+\int_0^t\langle\mathbf{h}^\varepsilon,\mathbf{w}^\varepsilon\rangle_{\mathcal{V}_\varepsilon}\,ds \\
      &\qquad\qquad +\int_0^t\bigl((\mathbf{w}^\varepsilon\cdot\nabla)\mathbf{w}^\varepsilon,\mathbf{v}_{\mathrm{E}}\bigr)_{L^2(\Omega_\varepsilon)}\,ds-\int_0^t\mathcal{R}_{\mathbf{v}}(\mathbf{w}^\varepsilon)\,ds
    \end{aligned}
  \end{align}
  for all $t\in(0,\infty)$, where we also used \eqref{Pf_DvE:TrSum}.
  Moreover, for any $\delta>0$,
  \begin{align} \label{Pf_DvE:Res}
    \begin{aligned}
      |\langle\mathbf{h}^\varepsilon,\mathbf{w}^\varepsilon\rangle_{\mathcal{V}_\varepsilon}| & \leq \delta\|\mathbf{w}^\varepsilon\|_{H^1(\Omega_\varepsilon)}^2+c\delta^{-1}\|\mathbf{h}^\varepsilon\|_{\mathcal{V}_\varepsilon^\ast}^2, \\
      \Bigl|\bigl((\mathbf{w}^\varepsilon\cdot\nabla)\mathbf{w}^\varepsilon,\mathbf{v}_{\mathrm{E}}\bigr)_{L^2(\Omega_\varepsilon)}\Bigr| &\leq \delta\|\mathbf{w}^\varepsilon\|_{H^1(\Omega_\varepsilon)}^2+c\delta^{-3}\|\mathbf{v}\|_{L^2(S^2)}^2\|\mathbf{v}\|_{H^1(S^2)}^2\|\mathbf{w}^\varepsilon\|_{L^2(\Omega_\varepsilon)}^2, \\
      |\mathcal{R}_{\mathbf{v}}(\mathbf{w}^\varepsilon)| &\leq \delta\|\mathbf{w}^\varepsilon\|_{H^1(\Omega_\varepsilon)}^2+c\delta^{-1}\varepsilon^3\eta_{\mathbf{v}}^2
    \end{aligned}
  \end{align}
  on $(0,\infty)$ by \eqref{Pf_ITr:Est}, \eqref{E:WFEx_Res}, and Young's inequality, where $c>0$ is a constant independent of $\varepsilon$, $t$, and $\delta$.
  We apply \eqref{Pf_DvE:Res} to \eqref{Pf_DvE:weps} and use \eqref{E:Korn_Thin} to the left-hand side to get
  \begin{align*}
    &\frac{1}{2}\|\mathbf{w}^\varepsilon(t)\|_{L^2(\Omega_\varepsilon)}^2+\frac{\nu}{2}\int_0^t\Bigl(\|\nabla\mathbf{w}^\varepsilon\|_{L^2(\Omega_\varepsilon)}^2-c\|\mathbf{w}^\varepsilon\|_{L^2(\Omega_\varepsilon)}^2\Bigr)\,ds \\
    &\qquad \leq \frac{1}{2}\|\mathbf{w}_0^\varepsilon\|_{L^2(\Omega_\varepsilon)}^2+c_1\delta\int_0^t\Bigl(\|\nabla\mathbf{w}^\varepsilon\|_{L^2(\Omega_\varepsilon)}^2+\|\mathbf{w}^\varepsilon\|_{L^2(\Omega_\varepsilon)}^2\Bigr)\,ds \\
    &\qquad\qquad +c\int_0^t\Bigl\{\delta^{-3}\|\mathbf{v}\|_{L^2(S^2)}^2\|\mathbf{v}\|_{H^1(S^2)}^2\|\mathbf{w}^\varepsilon\|_{L^2(\Omega_\varepsilon)}^2+\delta^{-1}\Bigl(\|\mathbf{h}^\varepsilon\|_{\mathcal{V}_\varepsilon^\ast}^2+\varepsilon^3\eta_{\mathbf{v}}^2\Bigr)\Bigr\}\,ds
  \end{align*}
  with constants $c,c_1>0$ independent of $\varepsilon$, $t$, $\nu$, and $\delta$.
  Then, we set $\delta:=\nu/4c_1$ in the above inequality and multiply both sides by two to deduce that
  \begin{multline*}
    \|\mathbf{w}^\varepsilon(t)\|_{L^2(\Omega_\varepsilon)}^2+\frac{\nu}{2}\int_0^t\|\nabla\mathbf{w}^\varepsilon\|_{L^2(\Omega_\varepsilon)}^2\,ds \\
    \leq \|\mathbf{w}_0^\varepsilon\|_{L^2(\Omega_\varepsilon)}^2+c\int_0^t\left\{\psi_{\mathbf{v}}\|\mathbf{w}^\varepsilon\|_{L^2(\Omega_\varepsilon)}^2+\frac{1}{\nu}\Bigl(\|\mathbf{h}^\varepsilon\|_{\mathcal{V}_\varepsilon^\ast}^2+\varepsilon^3\eta_{\mathbf{v}}^2\Bigr)\right\}\,ds
  \end{multline*}
  for all $t\in(0,\infty)$, where
  \begin{align*}
    \psi_{\mathbf{v}}(t) := \nu+\frac{1}{\nu^3}\|\mathbf{v}(t)\|_{L^2(S^2)}^2\|\mathbf{v}(t)\|_{H^1(S^2)}^2, \quad \psi_{\mathbf{v}} \in L_{\mathrm{loc}}^1([0,\infty)).
  \end{align*}
  Hence, we can apply Gronwall's inequality to obtain \eqref{E:Df_vE}.
\end{proof}

\begin{proof}[Proof of Theorem \ref{T:Diff_Est}]
  Let $\mathbf{w}^\varepsilon$, $\mathbf{w}_0^\varepsilon$, and $\mathbf{h}^\varepsilon$ be given by \eqref{E:Def_weps}.
  Then,
  \begin{align} \label{Pf_Dif:uv}
    \begin{aligned}
      \|\mathbf{u}^\varepsilon-\bar{\mathbf{v}}\|_{L^2(\Omega_\varepsilon)} &\leq \|\mathbf{w}^\varepsilon\|_{L^2(\Omega_\varepsilon)}+\|\mathbf{v}_{\mathrm{E}}-\bar{\mathbf{v}}\|_{L^2(\Omega_\varepsilon)} \\
      &\leq \|\mathbf{w}^\varepsilon\|_{L^2(\Omega_\varepsilon)}+c\varepsilon^{3/2}\|\mathbf{v}\|_{L^2(S^2)}
    \end{aligned}
  \end{align}
  by \eqref{E:DfEx_L2}.
  Similarly, since
  \begin{align*}
    \mathbf{w}_0^\varepsilon = (\mathbf{u}_0^\varepsilon-\bar{\mathbf{v}}_0)+(\bar{\mathbf{v}}_0-[\mathbf{v}_0]_{\mathrm{E}}), \quad \mathbf{h}^\varepsilon = (\mathbf{f}^\varepsilon-\bar{\mathbf{f}})+(\bar{\mathbf{f}}-\mathbf{f}_{\mathrm{E}}),
  \end{align*}
  it follows from \eqref{E:DfEx_L2} and \eqref{E:Func_Ext} that
  \begin{align} \label{Pf_Dif:wh}
    \begin{aligned}
      \|\mathbf{w}_0^\varepsilon\|_{L^2(S^2)} &\leq \|\mathbf{u}_0^\varepsilon-\bar{\mathbf{v}}_0\|_{L^2(\Omega_\varepsilon)}+c\varepsilon^{3/2}\|\mathbf{v}_0\|_{L^2(S^2)}, \\
      \|\mathbf{h}^\varepsilon\|_{\mathcal{V}_\varepsilon^\ast} &\leq \|\mathbf{f}^\varepsilon-\bar{\mathbf{f}}\|_{\mathcal{V}_\varepsilon^\ast}+c\varepsilon^{3/2}\|\mathbf{f}\|_{\mathcal{V}_0^\ast}.
    \end{aligned}
  \end{align}
  To calculate $\nabla\mathbf{w}^\varepsilon$, we see by $\mathbf{P}=\mathbf{I}_3-\mathbf{n}\otimes\mathbf{n}$ on $S^2$ that
  \begin{align*}
    \nabla\mathbf{u}^\varepsilon = \overline{\mathbf{P}}\nabla\mathbf{u}^\varepsilon+(\bar{\mathbf{n}}\otimes\bar{\mathbf{n}})\nabla\mathbf{u}^\varepsilon = \overline{\mathbf{P}}\nabla\mathbf{u}^\varepsilon+\bar{\mathbf{n}}\otimes[\partial_{\mathbf{n}}\mathbf{u}^\varepsilon], \quad \partial_{\mathbf{n}}\mathbf{u}^\varepsilon = (\bar{\mathbf{n}}\cdot\nabla)\mathbf{u}^\varepsilon
  \end{align*}
  in $\Omega_\varepsilon$.
  By this equality and \eqref{E:Ext_Grad}, we have
  \begin{align*}
    \nabla\mathbf{w}^\varepsilon = \mathbf{A}^\varepsilon+\bar{\mathbf{n}}\otimes\mathbf{a}^\varepsilon, \quad \mathbf{A}^\varepsilon := \overline{\mathbf{P}}\nabla\mathbf{u}^\varepsilon-\overline{\nabla_{S^2}\mathbf{v}}, \quad \mathbf{a}^\varepsilon := \partial_{\mathbf{n}}\mathbf{u}^\varepsilon-\bar{\mathbf{v}}.
  \end{align*}
  Moreover, since $\mathbf{P}\nabla_{S^2}\mathbf{v}=\nabla_{S^2}\mathbf{v}$, $\mathbf{P}^2=\mathbf{P}^{\mathrm{T}}=\mathbf{P}$, $\mathbf{P}\mathbf{n}=\mathbf{0}$, and $|\mathbf{n}|^2=1$ on $S^2$,
  \begin{align*}
    \mathbf{A}^\varepsilon = \overline{\mathbf{P}}\mathbf{A}^\varepsilon, \quad [\mathbf{A}^\varepsilon]^{\mathrm{T}}\bar{\mathbf{n}} = [\mathbf{A}^\varepsilon]^{\mathrm{T}}\,\overline{\mathbf{P}\mathbf{n}} = \mathbf{0}, \quad (\mathbf{a}^\varepsilon\otimes\bar{\mathbf{n}})(\bar{\mathbf{n}}\otimes\mathbf{a}^\varepsilon) = \mathbf{a}^\varepsilon\otimes\mathbf{a}^\varepsilon
  \end{align*}
  in $\Omega_\varepsilon$.
  Thus, $(\nabla\mathbf{w}^\varepsilon)^{\mathrm{T}}\nabla\mathbf{w}^\varepsilon=[\mathbf{A}^\varepsilon]^{\mathrm{T}}\mathbf{A}^\varepsilon+\mathbf{a}^\varepsilon\otimes\mathbf{a}^\varepsilon$ and we obtain
  \begin{align} \label{Pf_Dif:grad}
    |\nabla\mathbf{w}^\varepsilon|^2 = |\mathbf{A}^\varepsilon|^2+|\mathbf{a}^\varepsilon|^2 = \Bigl|\overline{\mathbf{P}}\nabla\mathbf{u}^\varepsilon-\overline{\nabla_{S^2}\mathbf{v}}\Bigr|^2+|\partial_{\mathbf{n}}\mathbf{u}^\varepsilon-\bar{\mathbf{v}}|^2 \quad\text{in}\quad \Omega_\varepsilon
  \end{align}
  by $|\mathbf{B}|^2=\mathrm{tr}[\mathbf{B}^{\mathrm{T}}\mathbf{B}]$ for $\mathbf{B}\in\mathbb{R}^{3\times3}$.
  We apply \eqref{Pf_Dif:uv}--\eqref{Pf_Dif:grad} to \eqref{E:Df_vE} and use
  \begin{align*}
    [\eta_{\mathbf{v}}(t)]^2 \leq c\Bigl(\nu^2+\|\mathbf{v}(t)\|_{L^2(S^2)}^2\Bigr)\|\mathbf{v}(t)\|_{H^1(S^2)}^2
  \end{align*}
  and $F_{\mathbf{v}}(t)\geq1$.
  Then, dividing the resulting inequality by $\varepsilon$, we get \eqref{E:Diff_Est}.
\end{proof}

Let us also show that the average of $\mathbf{u}^\varepsilon$ converges to $\mathbf{v}$ as $\varepsilon\to0$.

\begin{proof}[Proof of Theorem \ref{T:Ave_Conv}]
  For a function $\eta$ on $S^2$ and its constant extension $\bar{\eta}$,
  \begin{align*}
    \mathcal{M}_\varepsilon^0\bar{\eta}(y) = \eta(y), \quad \mathcal{M}_\varepsilon^1\bar{\eta}(y) = \frac{1}{2\varepsilon}\{(1+\varepsilon)^2-1\}\eta(y) = \left(1+\frac{\varepsilon}{2}\right)\eta(y), \quad y\in S^2.
  \end{align*}
  By these relations and \eqref{E:Ave_TGr} with $k=0$, we have
  \begin{align*}
    \mathcal{M}_\varepsilon^0\mathbf{u}^\varepsilon-\mathbf{v} &= \mathcal{M}_\varepsilon^0(\mathbf{u}^\varepsilon-\bar{\mathbf{v}}), \\
    \nabla_{S^2}\mathcal{M}_\varepsilon^0\mathbf{u}^\varepsilon-\nabla_{S^2}\mathbf{v} &= \mathcal{M}_\varepsilon^1\Bigl(\overline{\mathbf{P}}\nabla\mathbf{u}^\varepsilon-\overline{\nabla_{S^2}\mathbf{v}}\Bigr)+\frac{\varepsilon}{2}\nabla_{S^2}\mathbf{v}
  \end{align*}
  on $S^2$.
  Thus, it follows from \eqref{E:Ave_L2} that
  \begin{align} \label{Pf_AvC:Adif}
    \begin{aligned}
      \|\mathcal{M}_\varepsilon^0\mathbf{u}^\varepsilon-\mathbf{v}\|_{L^2(S^2)} &\leq c\varepsilon^{-1/2}\|\mathbf{u}^\varepsilon-\bar{\mathbf{v}}\|_{L^2(\Omega_\varepsilon)}, \\
      \|\nabla_{S^2}\mathcal{M}_\varepsilon^0\mathbf{u}^\varepsilon-\nabla_{S^2}\mathbf{v}\|_{L^2(S^2)} &\leq c\varepsilon^{-1/2}\Bigl\|\overline{\mathbf{P}}\nabla\mathbf{u}^\varepsilon-\overline{\nabla_{S^2}\mathbf{v}}\Bigr\|_{L^2(\Omega_\varepsilon)}+\frac{\varepsilon}{2}\|\nabla_{S^2}\mathbf{v}\|_{L^2(S^2)}.
    \end{aligned}
  \end{align}
  By these inequalities, the square root of \eqref{E:Diff_Est}, and \eqref{E:Def_FG}, we find that
  \begin{multline*}
    \|\mathcal{M}_\varepsilon^0\mathbf{u}^\varepsilon-\mathbf{v}\|_{L^\infty(0,T;L^2(S^2))}+\|\nabla_{S^2}\mathcal{M}_\varepsilon^0\mathbf{u}^\varepsilon-\nabla_{S^2}\mathbf{v}\|_{L^2(0,T;L^2(S^2))} \\
    \leq c_{\mathbf{v},\nu,T}\Bigl(\varepsilon^{-1/2}\|\mathbf{u}_0^\varepsilon-\bar{\mathbf{v}}_0\|_{L^2(\Omega_\varepsilon)}+\varepsilon^{-1/2}\|\mathbf{f}^\varepsilon-\bar{\mathbf{f}}\|_{L^2(0,T;\mathcal{V}_\varepsilon^\ast)}+\varepsilon\Bigr),
  \end{multline*}
  where $c_{\mathbf{v},\nu,T}>0$ is a constant depending on $\nu$, $T$, and
  \begin{align*}
    \|\mathbf{v}_0\|_{L^2(S^2)}, \quad \|\mathbf{v}\|_{L^\infty(0,T;L^2(S^2))}, \quad \|\mathbf{v}\|_{L^2(0,T;H^1(S^2))}, \quad \|\mathbf{f}\|_{L^2(0,T;\mathcal{V}_0^\ast)},
  \end{align*}
  but independent of $\varepsilon$.
  Thus, if \eqref{E:ACL_data} holds, then we have \eqref{E:ACL_sol}.
\end{proof}

\section{Global-in-time difference estimate} \label{S:Global}
This section is devoted to the proof of Theorem \ref{T:DEst_Gl}.

\subsection{Korn inequalities} \label{SS:Gl_Korn}
In the proof of the global difference estimate \eqref{E:DEst_Gl}, the following (uniform) Korn inequalities on $S^2$ and $\Omega_\varepsilon$ play a crucial role.
They were shown in \cite{JaOlRe18,Miu24_SNS} for closed surfaces and in \cite{LewMul11,Miu22_01} for curved thin domains, in which the proofs are written in terms of Killing vector fields on closed surfaces.
For the readers' convenience, we provide the proofs specialized for the unit sphere and the thin spherical shell below.

Let $\mathbf{a}\in\mathbb{R}^3$.
We define $\mathbf{r}_{\mathbf{a}}(x):=\mathbf{a}\times x$ for $x\in\mathbb{R}^3$, where $\times$ is the vector product in $\mathbb{R}^3$.
We use the same notation $\mathbf{r}_{\mathrm{a}}$ for the restriction of $\mathbf{r}_{\mathbf{a}}$ to $S^2$ and $\Omega_\varepsilon$.

\begin{lemma} \label{L:KoH1_S2}
  There exists a constant $c>0$ such that
  \begin{align} \label{E:KoH1_S2}
    \|\mathbf{v}\|_{H^1(S^2)} \leq c\|\mathbf{D}_{S^2}(\mathbf{v})\|_{L^2(S^2)}
  \end{align}
  for all $\mathbf{v}\in H_\tau^1(S^2)$ satisfying $(\mathbf{v},\mathbf{r}_{\mathbf{a}})_{L^2(S^2)}=0$ for all $\mathbf{a}\in\mathbb{R}^3$.
\end{lemma}

\begin{proof}
  By \eqref{E:Korn_S2}, it suffices to prove
  \begin{align} \label{Pf_KHS:goal}
    \|\mathbf{v}\|_{L^2(S^2)} \leq c\|\mathbf{D}_{S^2}(\mathbf{v})\|_{L^2(S^2)}.
  \end{align}
  Assume to the contrary that there exist some $\mathbf{v}_k\in H_\tau^1(S^2)$, $k\in\mathbb{N}$ such that
  \begin{align*}
    \|\mathbf{v}_k\|_{L^2(S^2)} > k\|\mathbf{D}_{S^2}(\mathbf{v}_k)\|_{L^2(S^2)}, \quad (\mathbf{v}_k,\mathbf{r}_{\mathbf{a}})_{L^2(S^2)} = 0 \quad\text{for all}\quad \mathbf{a}\in\mathbb{R}^3.
  \end{align*}
  Since $\|\mathbf{v}_k\|_{L^2(S^2)}>0$, we may replace $\mathbf{v}_k$ by $\mathbf{v}_k/\|\mathbf{v}_k\|_{L^2(S^2)}$ to assume that
  \begin{align*}
    \|\mathbf{v}_k\|_{L^2(S^2)} = 1, \quad \|\mathbf{D}_{S^2}(\mathbf{v}_k)\|_{L^2(S^2)} < \frac{1}{k} \quad\text{and thus}\quad \|\mathbf{v}_k\|_{H^1(S^2)} \leq c
  \end{align*}
  by \eqref{E:Korn_S2}.
  Then, since the embedding $H_\tau^1(S^2)\hookrightarrow L_\tau^2(S^2)$ is compact, we have
  \begin{align*}
    \lim_{k\to\infty}\mathbf{v}_k = \mathbf{v} \quad\text{weakly in $H_\tau^1(S^2)$ and strongly in $L_\tau^2(S^2)$}
  \end{align*}
  with some $\mathbf{v}\in H_\tau^1(S^2)$ up to a subsequence.
  Hence,
  \begin{align*}
    \|\mathbf{v}\|_{L^2(S^2)} = 1, \quad \|\mathbf{D}_{S^2}(\mathbf{v})\|_{L^2(S^2)} = 0, \quad (\mathbf{v},\mathbf{r}_{\mathbf{a}})_{L^2(S^2)} = 0 \quad\text{for all}\quad \mathbf{a}\in\mathbb{R}^3.
  \end{align*}
  Now, let $\mathbf{v}_{\mathrm{E}}$ be the extension of $\mathbf{v}$ given by \eqref{E:Def_Ext}.
  Then, we see by \eqref{E:Ext_str} that
  \begin{align*}
    \mathbf{D}(\mathbf{v}_{\mathrm{E}}) = \overline{\mathbf{D}_{S^2}(\mathbf{v})} = 0 \quad\text{a.e. in}\quad \mathbb{R}^3\setminus\{0\},
  \end{align*}
  and thus we can deduce that (see e.g. \cite[Lemma IV.7.5]{BoyFab13})
  \begin{align*}
    \mathbf{v}_{\mathrm{E}}(x) = \mathbf{a}\times x+\mathbf{b}, \quad x\in\mathbb{R}^3 \quad\text{with some}\quad \mathbf{a},\mathbf{b}\in\mathbb{R}^3.
  \end{align*}
  Moreover, since $\mathbf{v}\in H_\tau^1(S^2)$, we have
  \begin{align*}
    0 = \mathbf{v}(y)\cdot\mathbf{n}(y) = \mathbf{v}_{\mathrm{E}}(y)\cdot\mathbf{n}(y) = (\mathbf{a}\times y+\mathbf{b})\cdot y = \mathbf{b}\cdot y \quad\text{for a.a.}\quad y\in S^2.
  \end{align*}
  Hence, $\mathbf{b}=\mathbf{0}$ and $\mathbf{v}(y)=\mathbf{v}_{\mathrm{E}}(y)=\mathbf{a}\times y=\mathbf{r}_{\mathbf{a}}(y)$ for $y\in S^2$.
  This gives
  \begin{align*}
    \|\mathbf{v}\|_{L^2(S^2)}^2 = (\mathbf{v},\mathbf{v})_{L^2(S^2)} = (\mathbf{v},\mathbf{r}_{\mathbf{a}})_{L^2(S^2)} = 0,
  \end{align*}
  which contradicts $\|\mathbf{v}\|_{L^2(S^2)}=1$.
  Thus, \eqref{Pf_KHS:goal} is valid and we get \eqref{E:KoH1_S2}.
\end{proof}

\begin{lemma} \label{L:KoH1_TSS}
  There exist constants $\varepsilon_0\in(0,1)$ and $c>0$ such that
  \begin{align} \label{E:KoH1_TSS}
    \|\mathbf{u}\|_{H^1(\Omega_\varepsilon)} \leq c\|\mathbf{D}(\mathbf{u})\|_{L^2(\Omega_\varepsilon)}
  \end{align}
  for all $\varepsilon\in(0,\varepsilon_0]$ and $\mathbf{u}\in H^1(\Omega_\varepsilon)^3$ satisfying
  \begin{align*}
    \mathbf{u}\cdot\mathbf{n}_\varepsilon = 0 \quad\text{on}\quad \partial\Omega_\varepsilon, \quad (\mathbf{u},\mathbf{r}_{\mathbf{a}})_{L^2(\Omega_\varepsilon)} = 0 \quad\text{for all}\quad \mathbf{a}\in\mathbb{R}^3.
  \end{align*}
\end{lemma}

Note that \eqref{E:KoH1_TSS} is the uniform Korn inequality in the sense that we can take the constant $c>0$ independently of $\varepsilon$.
It was shown in \cite{LewMul11,Miu22_01} by a contradiction argument, but here we give the proof without contradiction.

\begin{proof}
  For any $\mathbf{a}\in\mathbb{R}^3$, we see by \eqref{E:CoV_Thin} and $\mathbf{r}_{\mathrm{a}}(x)=\mathbf{a}\times x$ that
  \begin{align*}
    (\mathbf{u},\mathbf{r}_{\mathbf{a}})_{L^2(\Omega_\varepsilon)} &= \int_{S^2}\int_1^{1+\varepsilon}\mathbf{u}(ry)\cdot\{\mathbf{a}\times(ry)\}r^2\,dr\,d\mathcal{H}^2(y) \\
    &= \int_{S^2}\left(\int_1^{1+\varepsilon}\mathbf{u}(ry)r^3\,dr\right)\cdot\mathbf{r}_{\mathbf{a}}(y)\,d\mathcal{H}^2(y) = \varepsilon(\mathcal{M}_\varepsilon^3\mathbf{u},\mathbf{r}_{\mathbf{a}})_{L^2(S^2)}.
  \end{align*}
  Moreover, since $\mathbf{r}_{\mathbf{a}}(y)\cdot\mathbf{n}(y)=(\mathbf{a}\times y)\cdot y=0$ for all $y\in S^2$, we have
  \begin{align*}
    (\mathbf{u},\mathbf{r}_{\mathbf{a}})_{L^2(\Omega_\varepsilon)} = \varepsilon(\mathcal{M}_\varepsilon^3\mathbf{u},\mathbf{r}_{\mathbf{a}})_{L^2(S^2)} = \varepsilon(\mathcal{M}_{\varepsilon,\tau}^3\mathbf{u},\mathbf{r}_{\mathbf{a}})_{L^2(S^2)}.
  \end{align*}
  Thus, if $\mathbf{u}\in H^1(\Omega_\varepsilon)^3$ satisfies the assumption of the lemma, then
  \begin{align*}
    \mathcal{M}_{\varepsilon,\tau}^3\mathbf{u} \in H_\tau^1(S^2), \quad (\mathcal{M}_{\varepsilon,\tau}^3\mathbf{u},\mathbf{r}_{\mathbf{a}})_{L^2(S^2)} = 0 \quad\text{for all}\quad \mathbf{a}\in\mathbb{R}^3,
  \end{align*}
  and we can apply \eqref{E:KoH1_S2} to $\mathcal{M}_{\varepsilon,\tau}^3\mathbf{u}$.
  By this fact, \eqref{E:VT_str}, and $|\mathbf{P}|=\sqrt{2}$ on $S^2$,
  \begin{align*}
    \|\mathcal{M}_{\varepsilon,\tau}^3\mathbf{u}\|_{H^1(S^2)} &\leq c\|\mathbf{D}_{S^2}(\mathcal{M}_{\varepsilon,\tau}^3\mathbf{u})\|_{L^2(S^2)} \\
    &\leq c\Bigl(\|\mathbf{D}_{S^2}(\mathcal{M}_\varepsilon^3\mathbf{u})\|_{L^2(S^2)}+\|\mathcal{M}_\varepsilon^3\mathbf{u}\cdot\mathbf{n}\|_{L^2(S^2)}\Bigr).
  \end{align*}
  Moreover, we see by \eqref{E:Ave_TGr} and $\mathbf{P}^2=\mathbf{P}^{\mathrm{T}}=\mathbf{P}$ on $S^2$ that
  \begin{align*}
    \mathbf{D}_{S^2}(\mathcal{M}_\varepsilon^3\mathbf{u}) = \mathbf{P}(\nabla_{S^2}\mathcal{M}_\varepsilon^3\mathbf{u})_{\mathrm{S}}\mathbf{P} = \mathbf{P}(\mathcal{M}_\varepsilon^4[\nabla\mathbf{u}])_{\mathrm{S}}\mathbf{P} = \mathbf{P}\mathcal{M}_\varepsilon^4[\mathbf{D}(\mathbf{u})]\mathbf{P}
  \end{align*}
  on $S^2$ (recall that $\mathbf{A}_{\mathrm{S}}=(\mathbf{A}+\mathbf{A}^{\mathrm{T}})/2$ for $\mathbf{A}\in\mathbb{R}^{3\times3}$).
  Hence,
  \begin{align} \label{Pf_KHT:MuL2}
    \begin{aligned}
      \|\mathcal{M}_{\varepsilon,\tau}^3\mathbf{u}\|_{H^1(S^2)} &\leq c\Bigl(\|\mathbf{P}\mathcal{M}_\varepsilon^4[\mathbf{D}(\mathbf{u})]\mathbf{P}\|_{L^2(S^2)}+\|\mathcal{M}_\varepsilon^3\mathbf{u}\cdot\mathbf{n}\|_{L^2(S^2)}\Bigr) \\
      &\leq c\Bigl(\varepsilon^{-1/2}\|\mathbf{D}(\mathbf{u})\|_{L^2(\Omega_\varepsilon)}+\varepsilon^{1/2}\|\mathbf{u}\|_{H^1(\Omega_\varepsilon)}\Bigr)
    \end{aligned}
  \end{align}
  by $|\mathbf{P}|=\sqrt{2}$ on $S^2$, \eqref{E:Ave_L2}, and \eqref{E:Ave_NC} (note that $\mathbf{u}\cdot\mathbf{n}_\varepsilon=0$ on $\partial\Omega_\varepsilon$).
  Now,
  \begin{align*}
    \|\mathbf{u}\|_{L^2(\Omega_\varepsilon)} &\leq \Bigl\|\mathbf{u}-\overline{\mathcal{M}_{\varepsilon,\tau}^3\mathbf{u}}\Bigr\|_{L^2(\Omega_\varepsilon)}+\Bigl\|\overline{\mathcal{M}_{\varepsilon,\tau}^3\mathbf{u}}\Bigr\|_{L^2(\Omega_\varepsilon)} \\
    &\leq c\Bigl(\varepsilon\|\mathbf{u}\|_{H^1(\Omega_\varepsilon)}+\varepsilon^{1/2}\|\mathcal{M}_{\varepsilon,\tau}^3\mathbf{u}\|_{L^2(S^2)}\Bigr) \\
    &\leq c\Bigl(\|\mathbf{D}(\mathbf{u})\|_{L^2(\Omega_\varepsilon)}+\varepsilon\|\mathbf{u}\|_{H^1(\Omega_\varepsilon)}\Bigr)
  \end{align*}
  by \eqref{E:CoEx_L2}, \eqref{E:Atan_Con}, and \eqref{Pf_KHT:MuL2}.
  Then, we further use \eqref{E:Korn_Thin} to get
  \begin{align*}
    \|\mathbf{u}\|_{L^2(\Omega_\varepsilon)} \leq c_1(1+\varepsilon)\|\mathbf{D}(\mathbf{u})\|_{L^2(\Omega_\varepsilon)}+c_2\varepsilon\|\mathbf{u}\|_{L^2(\Omega_\varepsilon)}
  \end{align*}
  with some constants $c_1,c_2>0$ independent of $\varepsilon$.
  Thus, setting
  \begin{align*}
    \varepsilon_0 := \min\left\{\frac{1}{2},\frac{1}{2c_2}\right\} \in (0,1),
  \end{align*}
  we have $c_2\varepsilon<1/2$ for $\varepsilon\in(0,\varepsilon_0]$ and deduce from the above inequality that
  \begin{align*}
    \|\mathbf{u}\|_{L^2(\Omega_\varepsilon)} \leq 2c_1(1+\varepsilon)\|\mathbf{D}(\mathbf{u})\|_{L^2(\Omega_\varepsilon)} \leq 4c_1\|\mathbf{D}(\mathbf{u})\|_{L^2(\Omega_\varepsilon)}.
  \end{align*}
  By this estimate and \eqref{E:Korn_Thin}, we obtain \eqref{E:KoH1_TSS}.
\end{proof}

\subsection{Preservation of orthogonality} \label{SS:Gl_Orth}
Next, we see that weak solutions to \eqref{E:NS_TSS} and \eqref{E:NS_S2} are orthogonal to the vector fields $\mathbf{r}_{\mathbf{a}}$ if given data satisfy the same condition.

Let us first consider a weak solution to \eqref{E:NS_TSS} on $\Omega_\varepsilon$.

\begin{lemma} \label{L:Rota_TSS}
  For each $\mathbf{a}\in\mathbb{R}^3$, the restriction of $\mathbf{r}_{\mathbf{a}}$ to $\Omega_\varepsilon$ satisfies
  \begin{align*}
    \mathbf{r}_{\mathbf{a}} \in \mathcal{V}_\varepsilon, \quad \mathbf{D}(\mathbf{r}_{\mathbf{a}}) = \mathbf{O}_3 \quad\text{in}\quad \Omega_\varepsilon, \quad \bigl((\mathbf{u}\cdot\nabla)\mathbf{u},\mathbf{r}_{\mathbf{a}}\bigr)_{L^2(\Omega_\varepsilon)} = 0 \quad\text{for all}\quad \mathbf{u}\in\mathcal{V}_\varepsilon.
  \end{align*}
\end{lemma}

\begin{proof}
  Since $\mathbf{r}_{\mathbf{a}}(x)=\mathbf{a}\times x$ for $x\in\mathbb{R}^3$, we see by direct calculations that
  \begin{align*}
    \mathbf{D}(\mathbf{r}_{\mathbf{a}}) = \mathbf{O}_3, \quad \mathrm{div}\,\mathbf{r}_{\mathbf{a}} = 0, \quad (\mathbf{b}\cdot\nabla)\mathbf{r}_{\mathbf{a}} = \mathbf{a}\times\mathbf{b} \quad (\mathbf{b}\in\mathbb{R}^3) \quad\text{in}\quad \mathbb{R}^3
  \end{align*}
  and $\mathbf{r}_{\mathbf{a}}(x)\cdot\mathbf{n}_\varepsilon(x)=(\mathbf{a}\times x)\cdot(\pm x/|x|)=0$ for $x\in\partial\Omega_\varepsilon$.
  Thus, $\mathbf{r}_{\mathbf{a}}\in\mathcal{V}_\varepsilon$.
  Also,
  \begin{align*}
    \bigl((\mathbf{u}\cdot\nabla)\mathbf{u},\mathbf{r}_{\mathbf{a}}\bigr)_{L^2(\Omega_\varepsilon)} = -\bigl(\mathbf{u},(\mathbf{u}\cdot\nabla)\mathbf{r}_{\mathbf{a}}\bigr)_{L^2(\Omega_\varepsilon)} = -(\mathbf{u},\mathbf{a}\times\mathbf{u})_{L^2(\Omega_\varepsilon)} = 0
  \end{align*}
  for all $\mathbf{u}\in\mathcal{V}_\varepsilon$ by \eqref{E:TrTS_AnS}.
\end{proof}

\begin{proposition} \label{P:TSS_Orth}
  For given $\mathbf{u}_0^\varepsilon\in\mathcal{H}_\varepsilon$ and $\mathbf{f}^\varepsilon\in L_{\mathrm{loc}}^2([0,\infty);\mathcal{V}_\varepsilon^\ast)$, let $\mathbf{u}^\varepsilon$ be a weak solution to \eqref{E:NS_TSS}.
  Also, let $\mathbf{a}\in\mathbb{R}^3$.
  Suppose that
  \begin{align*}
    (\mathbf{u}_0^\varepsilon,\mathbf{r}_{\mathbf{a}})_{L^2(\Omega_\varepsilon)} = 0, \quad \langle\mathbf{f}^\varepsilon(t),\mathbf{r}_{\mathbf{a}}\rangle_{\mathcal{V}_\varepsilon} = 0 \quad\text{for a.a.}\quad t\in(0,\infty).
  \end{align*}
  Then, $(\mathbf{u}^\varepsilon(t),\mathbf{r}_{\mathbf{a}})_{L^2(\Omega_\varepsilon)}=0$ for all $t\geq0$.
\end{proposition}

\begin{proof}
  Note that $(\mathbf{u}^\varepsilon(\cdot),\mathbf{r}_{\mathbf{a}})_{L^2(\Omega_\varepsilon)}\in C([0,\infty))$ by Proposition \ref{P:TSS_WeCo}.

  For any $\xi=\xi(t)\in C_c^\infty(0,\infty)$, let $\bm{\psi}:=\xi\mathbf{r}_{\mathbf{a}}\in C_c^\infty(0,\infty;\mathcal{V}_\varepsilon)$ in \eqref{E:WF_TSS}.
  Then,
  \begin{align*}
    \int_0^\infty \frac{d\xi}{dt}(t)(\mathbf{u}^\varepsilon(t),\mathbf{r}_{\mathbf{a}})_{L^2(\Omega_\varepsilon)}\,dt = 0
  \end{align*}
  by Lemma \ref{L:Rota_TSS} and the assumption on $\mathbf{f}^\varepsilon$.
  Hence,
  \begin{align*}
    \frac{d}{dt}(\mathbf{u}^\varepsilon(t),\mathbf{r}_{\mathbf{a}})_{L^2(\Omega_\varepsilon)} = 0 \quad\text{for a.a}\quad t\in(0,\infty),
  \end{align*}
  and we have $(\mathbf{u}^\varepsilon(t),\mathbf{r}_{\mathbf{a}})_{L^2(\Omega_\varepsilon)}=(\mathbf{u}_0^\varepsilon,\mathbf{r}_{\mathbf{a}})_{L^2(\Omega_\varepsilon)}=0$ for all $t\geq0$.
\end{proof}

Next, we deal with a weak solution to \eqref{E:NS_S2} on $S^2$.

\begin{lemma} \label{L:Rota_S2}
  For each $\mathbf{a}\in\mathbb{R}^3$, the restriction of $\mathbf{r}_{\mathbf{a}}$ to $S^2$ satisfies
  \begin{align*}
    \mathbf{r}_{\mathbf{a}} \in \mathcal{V}_0, \quad \mathbf{D}_{S^2}(\mathbf{r}_{\mathbf{a}}) = \mathbf{O}_3 \quad\text{on}\quad S^2, \quad (\nabla_{\mathbf{v}}\mathbf{v},\mathbf{r}_{\mathbf{a}})_{L^2(S^2)} = 0 \quad\text{for all}\quad \mathbf{v}\in\mathcal{V}_0.
  \end{align*}
\end{lemma}

\begin{proof}
  We have $\mathbf{r}_{\mathbf{a}}(y)\cdot\mathbf{n}(y)=(\mathbf{a}\times y)\cdot y=0$ for $y\in S^2$.
  Moreover,
  \begin{align*}
    \mathbf{D}_{S^2}(\mathbf{r}_{\mathbf{a}}) = \mathbf{P}(\nabla_{S^2}\mathbf{r}_{\mathbf{a}})_{\mathrm{S}}\mathbf{P} = \mathbf{P}(\nabla\mathbf{r}_{\mathbf{a}})_{\mathrm{S}}\mathbf{P} = \mathbf{P}[\mathbf{D}(\mathbf{r}_{\mathbf{a}})]\mathbf{P} = \mathbf{O}_3 \quad\text{on}\quad S^2
  \end{align*}
  by $\nabla_{S^2}\mathbf{r}_{\mathbf{a}}=\mathbf{P}\nabla\mathbf{r}_{\mathbf{a}}$ and $\mathbf{P}^2=\mathbf{P}^{\mathrm{T}}=\mathbf{P}$.
  We also have $\mathrm{div}_{S^2}\mathbf{r}_{\mathbf{a}}=0$ on $S^2$, since
  \begin{align*}
    \mathrm{tr}[\mathbf{D}_{S^2}(\mathbf{r}_{\mathbf{a}})] = \mathrm{tr}[\mathbf{P}(\nabla_{S^2}\mathbf{r}_{\mathbf{a}})\mathbf{P}] = \mathrm{tr}[\nabla_{S^2}\mathbf{r}_{\mathbf{a}}+\mathbf{r}_{\mathbf{a}}\otimes\mathbf{n}] = \mathrm{div}_{S^2}\mathbf{r}_{\mathbf{a}}+\mathbf{r}_{\mathbf{a}}\cdot\mathbf{n}
  \end{align*}
  by $\mathrm{tr}[\mathbf{A}_{\mathrm{S}}]=\mathrm{tr}[\mathbf{A}]$ for $\mathbf{A}\in\mathbb{R}^{3\times3}$ and \eqref{E:TGr_Dec}.
  Thus, $\mathbf{r}_{\mathbf{a}}\in\mathcal{V}_0$.

  Let $\mathbf{v}\in\mathcal{V}_0$.
  Then, $(\nabla_{\mathbf{v}}\mathbf{v},\mathbf{r}_{\mathbf{a}})_{L^2(S^2)}=-(\mathbf{v},\nabla_{\mathbf{v}}\mathbf{r}_{\mathbf{a}})_{L^2(S^2)}$ by \eqref{E:TrS2_AnS}.
  Moreover,
  \begin{align*}
    (\mathbf{v}\cdot\nabla_{S^2})\mathbf{r}_{\mathbf{a}} &= (\mathbf{v}\cdot\nabla)\mathbf{r}_{\mathbf{a}} = \mathbf{a}\times\mathbf{v}, \\
    \mathbf{v}\cdot\nabla_{\mathbf{v}}\mathbf{r}_{\mathbf{a}} &= \mathbf{v}\cdot(\mathbf{P}[(\mathbf{v}\cdot\nabla_{S^2})\mathbf{r}_{\mathbf{a}}]) = \mathbf{v}\cdot(\mathbf{a}\times\mathbf{v}) = 0
  \end{align*}
  on $S^2$, since $\mathbf{v}$ is tangential on $S^2$.
  Hence, $(\nabla_{\mathbf{v}}\mathbf{v},\mathbf{r}_{\mathbf{a}})_{L^2(S^2)}=0$.
\end{proof}

\begin{proposition} \label{P:S2_Orth}
  For given $\mathbf{v}_0\in\mathcal{H}_0$ and $\mathbf{f}\in L_{\mathrm{loc}}^2([0,\infty);\mathcal{V}_0^\ast)$, let $\mathbf{v}$ be a weak solution to \eqref{E:NS_S2}.
  Also, let $\mathbf{a}\in\mathbb{R}^3$.
  Suppose that
  \begin{align*}
    (\mathbf{v}_0,\mathbf{r}_{\mathbf{a}})_{L^2(S^2)} = 0, \quad \langle\mathbf{f}(t),\mathbf{r}_{\mathbf{a}}\rangle_{\mathcal{V}_0} = 0 \quad\text{for a.a.}\quad t\in(0,\infty).
  \end{align*}
  Then, $(\mathbf{v}(t),\mathbf{r}_{\mathbf{a}})_{L^2(S^2)}=0$ and also $(\mathbf{v}_{\mathrm{E}}(t),\mathbf{r}_{\mathbf{a}})_{L^2(\Omega_\varepsilon)}=0$ for all $t\geq0$.
\end{proposition}

\begin{proof}
  Recall that $\mathbf{v}_{\mathrm{E}}$ is the extension of $\mathbf{v}$ given by \eqref{E:Def_Ext} and
  \begin{align*}
    \mathbf{v} \in \mathbb{E}_T(\mathcal{V}_0) \subset C([0,T];\mathcal{H}_0), \quad \mathbf{v}_{\mathrm{E}} \in \mathbb{E}_T(\mathcal{V}_\varepsilon) \subset C([0,T];\mathcal{H}_\varepsilon)
  \end{align*}
  for all $T>0$ by Proposition \ref{P:S2_Cont} and Lemmas \ref{L:BS_Ipl} and \ref{L:Ext_Dt}.

  For any $\xi\in C_c^\infty(0,\infty)$, we set $\bm{\zeta}:=\xi\mathbf{r}_{\mathbf{a}}\in C_c^\infty(0,\infty;\mathcal{V}_0)$ in \eqref{E:WF_S2} to get
  \begin{align*}
    \int_0^\infty\frac{d\xi}{dt}(t)(\mathbf{v}(t),\mathbf{r}_{\mathbf{a}})_{L^2(S^2)}\,ds = 0
  \end{align*}
  by Lemma \ref{L:Rota_S2} and the assumption on $\mathbf{f}$.
  Thus,
  \begin{align*}
    \frac{d}{dt}(\mathbf{v}(t),\mathbf{r}_{\mathbf{a}})_{L^2(S^2)} = 0 \quad\text{for a.a.}\quad t\in(0,\infty)
  \end{align*}
  and $(\mathbf{v}(t),\mathbf{r}_{\mathbf{a}})_{L^2(S^2)}=(\mathbf{v}_0,\mathbf{r}_{\mathbf{a}})_{L^2(S^2)}=0$ for all $t\geq0$.
  Also,
  \begin{align*}
    (\mathbf{v}_{\mathrm{E}}(t),\mathbf{r}_{\mathbf{a}})_{L^2(\Omega_\varepsilon)} &= \int_{S^2}\int_1^{1+\varepsilon}[r\mathbf{v}(y,t)\cdot\{\mathbf{a}\times(ry)\}]r^2\,dr\,d\mathcal{H}^2(y) \\
    &= \frac{1}{5}\{(1+\varepsilon)^5-1\}(\mathbf{v}(t),\mathbf{r}_{\mathbf{a}})_{L^2(S^2)} = 0
  \end{align*}
  for all $t\geq0$ by \eqref{E:CoV_Thin}, \eqref{E:Def_Ext}, and $\mathbf{r}_{\mathbf{a}}(x)=\mathbf{a}\times x$ for $x\in\mathbb{R}^3$.
\end{proof}

\subsection{Global difference estimate} \label{SS:Gl_GlDF}
Now, we impose the assumptions of Theorem \ref{T:DEst_Gl}.
Let $\mathbf{u}^\varepsilon$ be a weak solution to \eqref{E:NS_TSS} satisfying the energy inequality \eqref{E:TSS_Ener}.
Also, let $\mathbf{v}$ be a weak solution to \eqref{E:NS_S2} and $\mathbf{v}_{\mathrm{E}}$ be its extension given by \eqref{E:Def_Ext}.
We derive the global difference estimate \eqref{E:DEst_Gl} after showing two auxiliary results.

\begin{proposition} \label{P:S2_GlEn}
  Under the assumptions of Theorem \ref{T:DEst_Gl}, we have
  \begin{align} \label{E:S2_GlEn}
    \|\mathbf{v}(t)\|_{L^2(S^2)}^2+\nu\int_0^t\|\mathbf{v}(t)\|_{H^1(S^2)}^2\,ds \leq cE_0 \quad\text{for all}\quad t\geq0,
  \end{align}
  where $E_0=\|\mathbf{v}_0\|_{L^2(S^2)}^2+\nu^{-1}\int_0^\infty\|\mathbf{f}\|_{\mathcal{V}_0^\ast}^2\,dt$ as in \eqref{E:Def_GlCon}.
\end{proposition}

\begin{proof}
  By the assumptions of Theorem \ref{T:DEst_Gl}, we can use Proposition \ref{P:S2_Orth} to get
  \begin{align*}
    (\mathbf{v}(t),\mathbf{r}_{\mathbf{a}})_{L^2(S^2)} = 0 \quad\text{for all}\quad t\geq0 \quad\text{and}\quad \mathbf{a}\in\mathbb{R}^3.
  \end{align*}
  Thus, we can apply \eqref{E:KoH1_S2} to $\mathbf{v}(t)$ for all $t\geq0$, and we have
  \begin{align*}
    \int_0^t\langle\mathbf{f},\mathbf{v}\rangle_{\mathcal{V}_0}\,ds &\leq \int_0^t\|\mathbf{f}\|_{\mathcal{V}_0^\ast}\|\mathbf{v}\|_{H^1(S^2)}\,ds \leq c\int_0^t\|\mathbf{f}\|_{\mathcal{V}_0^\ast}\|\mathbf{D}_{S^2}(\mathbf{v})\|_{L^2(S^2)}\,ds \\
    &\leq \nu\int_0^t\|\mathbf{D}_{S^2}(\mathbf{v})\|_{L^2(S^2)}^2\,ds+\frac{c}{\nu}\int_0^t\|\mathbf{f}\|_{\mathcal{V}_0^\ast}^2\,ds
  \end{align*}
  by Young's inequality.
  Using this estimate to \eqref{E:S2_Ener}, we find that
  \begin{align*}
    \frac{1}{2}\|\mathbf{v}(t)\|_{L^2(S^2)}^2+\nu\int_0^t\|\mathbf{D}_{S^2}(\mathbf{v})\|_{L^2(S^2)}^2\,ds \leq \frac{1}{2}\|\mathbf{v}_0\|_{L^2(S^2)}^2+\frac{c}{\nu}\int_0^t\|\mathbf{f}\|_{\mathcal{V}_0^\ast}^2\,ds \leq cE_0
  \end{align*}
  for all $t\geq0$.
  Hence, we again apply \eqref{E:KoH1_S2} to get \eqref{E:S2_GlEn}.
\end{proof}

\begin{theorem} \label{T:GlDf_vE}
  Under the assumptions of Theorem \ref{T:DEst_Gl}, let $\mathbf{w}^\varepsilon$, $\mathbf{w}_0^\varepsilon$, and $\mathbf{h}^\varepsilon$ be the vector fields given by \eqref{E:Def_weps}.
  Then, we have
  \begin{align} \label{E:GlDf_vE}
    \|\mathbf{w}^\varepsilon(t)\|_{L^2(\Omega_\varepsilon)}^2+\nu\int_0^t\|\mathbf{w}^\varepsilon\|_{H^1(S^2)}^2\,ds \leq cF_0\left(\|\mathbf{w}_0^\varepsilon\|_{L^2(\Omega_\varepsilon)}^2+\frac{1}{\nu}\int_0^t\|\mathbf{h}^\varepsilon\|_{\mathcal{V}_\varepsilon^\ast}^2\,ds+\varepsilon^3G_0\right)
  \end{align}
  for all $t\geq0$, where $F_0$ and $G_0$ are the constants given by \eqref{E:Def_GlCon}.
\end{theorem}

\begin{proof}
  Again, we can use Propositions \ref{P:TSS_Orth} and \ref{P:S2_Orth} under the assumptions of Theorem \ref{T:DEst_Gl}.
  Hence, for all $t\geq0$ and $\mathbf{a}\in\mathbb{R}^3$, we have
  \begin{align*}
    (\mathbf{w}^\varepsilon(t),\mathbf{r}_{\mathbf{a}})_{L^2(\Omega_\varepsilon)} = (\mathbf{u}^\varepsilon(t),\mathbf{r}_{\mathbf{a}})_{L^2(\Omega_\varepsilon)}-(\mathbf{v}_{\mathrm{E}}(t),\mathbf{r}_{\mathbf{a}})_{L^2(\Omega_\varepsilon)} = 0,
  \end{align*}
  and we can apply \eqref{E:KoH1_TSS} to $\mathbf{w}^\varepsilon(t)$ for all $t\geq0$.
  Now, we return to the proof of Theorem \ref{T:Df_vE}, and use \eqref{Pf_DvE:Res} and \eqref{E:KoH1_TSS} (instead of \eqref{E:Korn_Thin}) to \eqref{Pf_DvE:weps}.
  Then,
  \begin{align*}
    &\frac{1}{2}\|\mathbf{w}^\varepsilon(t)\|_{L^2(\Omega_\varepsilon)}^2+c_1\nu\int_0^t\|\mathbf{w}^\varepsilon\|_{H^1(\Omega_\varepsilon)}^2\,ds \\
    &\qquad \leq \frac{1}{2}\|\mathbf{w}_0^\varepsilon\|_{L^2(\Omega_\varepsilon)}^2+c_2\delta\int_0^t\|\mathbf{w}^\varepsilon\|_{H^1(\Omega_\varepsilon)}^2\,ds \\
    &\qquad\qquad +c\int_0^t\Bigl\{\delta^{-3}\|\mathbf{v}\|_{L^2(S^2)}^2\|\mathbf{v}\|_{H^1(S^2)}^2\|\mathbf{w}^\varepsilon\|_{L^2(\Omega_\varepsilon)}^2+\delta^{-1}\Bigl(\|\mathbf{h}^\varepsilon\|_{\mathcal{V}_\varepsilon^\ast}^2+\varepsilon^3\eta_{\mathbf{v}}^2\Bigr)\Bigr\}\,ds
  \end{align*}
  for any $\delta>0$, where $\eta_{\mathbf{v}}$ is given by \eqref{E:Def_eta} and $c_1,c_2,c>0$ are some constants independent of $\varepsilon$, $t$, $\nu$, and $\delta$.
  Thus, setting $\delta:=c_1\nu/2c_2$ (i.e. $c_2\delta=c_1\nu/2$), we get
  \begin{multline*}
    \|\mathbf{w}^\varepsilon(t)\|_{L^2(\Omega_\varepsilon)}^2+c_1\nu\int_0^t\|\mathbf{w}^\varepsilon\|_{H^1(\Omega_\varepsilon)}^2\,ds \\
    \leq \|\mathbf{w}_0^\varepsilon\|_{L^2(\Omega_\varepsilon)}^2+c\int_0^t\left\{\varphi_{\mathbf{v}}\|\mathbf{w}^\varepsilon\|_{L^2(\Omega_\varepsilon)}^2+\frac{1}{\nu}\Bigl(\|\mathbf{h}^\varepsilon\|_{\mathcal{V}_\varepsilon^\ast}^2+\varepsilon^3\eta_{\mathbf{v}}^2\Bigr)\right\}\,ds
  \end{multline*}
  for all $t\geq0$, where $\varphi_{\mathbf{v}}(t):=\nu^{-3}\|\mathbf{v}(t)\|_{L^2(S^2)}^2\|\mathbf{v}(t)\|_{H^1(S^2)}^2$.
  Therefore,
  \begin{multline*}
    \|\mathbf{w}(t)\|_{L^2(\Omega_\varepsilon)}^2+c_1\nu\int_0^t\|\mathbf{w}^\varepsilon\|_{H^1(S^2)}^2\,ds \\
    \leq \exp\left(c\int_0^t\varphi_{\mathbf{v}}\,ds\right)\left\{\|\mathbf{w}_0^\varepsilon\|_{L^2(S^2)}^2+\frac{c}{\nu}\int_0^t\Bigl(\|\mathbf{h}^\varepsilon\|_{\mathcal{V}_\varepsilon^\ast}^2+\varepsilon^3\eta_{\mathbf{v}}^2\Bigr)\,ds\right\}
  \end{multline*}
  for all $t\geq0$ by Gronwall's inequality.
  Now, we see that
  \begin{align*}
    \int_0^t\varphi_{\mathbf{v}}\,ds &= \frac{1}{\nu^3}\int_0^t\|\mathbf{v}\|_{L^2(S^2)}^2\|\mathbf{v}\|_{H^1(S^2)}^2\,ds \leq \frac{c}{\nu^4}E_0^2, \\
    \int_0^t\eta_{\mathbf{v}}^2\,ds &\leq c\int_0^t\Bigl(\nu^2+\|\mathbf{v}\|_{L^2(S^2)}^2\Bigr)\|\mathbf{v}\|_{H^1(S^2)}^2\,ds \leq c(\nu^2+E_0)\frac{E_0}{\nu}
  \end{align*}
  by \eqref{E:Def_eta} and \eqref{E:S2_GlEn}.
  Thus, we obtain \eqref{E:GlDf_vE}.
\end{proof}

\begin{proof}[Proof of Theorem \ref{T:DEst_Gl}]
  Let $\mathbf{w}^\varepsilon$, $\mathbf{w}_0^\varepsilon$, and $\mathbf{h}^\varepsilon$ be given by \eqref{E:Def_weps}.
  By \eqref{Pf_Dif:grad}, we have
  \begin{align*}
    \|\mathbf{w}^\varepsilon\|_{H^1(\Omega_\varepsilon)}^2 = \|\mathbf{w}^\varepsilon\|_{L^2(\Omega_\varepsilon)}^2+\Bigl\|\overline{\mathbf{P}}\nabla\mathbf{u}^\varepsilon-\overline{\nabla_{S^2}\mathbf{v}}\Bigr\|_{L^2(\Omega_\varepsilon)}^2+\|\partial_{\mathbf{n}}\mathbf{u}^\varepsilon-\bar{\mathbf{v}}\|_{L^2(\Omega_\varepsilon)}^2.
  \end{align*}
  We divide \eqref{E:GlDf_vE} by $\varepsilon$ and use the above equality, \eqref{Pf_Dif:uv}, and \eqref{Pf_Dif:wh} to get
  \begin{align} \label{Pf_Gl:Est}
    D_{\mathrm{sol}}^\varepsilon(t)+\frac{\nu}{\varepsilon}\int_0^t\|\mathbf{u}^\varepsilon-\bar{\mathbf{v}}\|_{L^2(\Omega_\varepsilon)}^2\,ds \leq cF_0\left\{D_{\mathrm{data}}^\varepsilon(t)+\varepsilon^2[G_0+A_0(t)]\right\}+c\varepsilon^2B_{\mathbf{v}}(t)
  \end{align}
  for all $t\geq0$, where $D_{\mathrm{sol}}^\varepsilon(t)$ and $D_{\mathrm{data}}^\varepsilon(t)$ are given by \eqref{E:Def_Diff} and
  \begin{align*}
    A_0(t) &:= \|\mathbf{v}_0\|_{L^2(S^2)}^2+\frac{1}{\nu}\int_0^t\|\mathbf{f}\|_{\mathcal{V}_0^\ast}^2\,ds, \\
    B_{\mathbf{v}}(t) &:= \|\mathbf{v}(t)\|_{L^2(S^2)}^2+\nu\int_0^t\|\mathbf{v}\|_{L^2(S^2)}^2\,ds.
  \end{align*}
  Now, we use \eqref{E:Def_GlCon}, $\|\mathbf{v}\|_{L^2(S^2)}\leq\|\mathbf{v}\|_{H^1(S^2)}$, and \eqref{E:S2_GlEn} to deduce that
  \begin{align*}
    A_0(t) \leq E_0 \leq G_0, \quad B_{\mathbf{v}}(t) \leq cE_0 \leq cG_0.
  \end{align*}
  Applying these estimates and $F_0\geq1$ to \eqref{Pf_Gl:Est}, we get \eqref{E:DEst_Gl}.

  Next, suppose that $\mathbf{f}^\varepsilon\in L^2(0,\infty;\mathcal{V}_\varepsilon^\ast)$.
  Then, for all $t\geq0$, we have
  \begin{align*}
    &\|[\mathcal{M}_\varepsilon^0\mathbf{u}^\varepsilon-\mathbf{v}](t)\|_{L^2(S^2)}^2+\nu\int_0^t\|\mathcal{M}_\varepsilon^0\mathbf{u}^\varepsilon-\mathbf{v}\|_{H^1(S^2)}^2\,ds \\
    &\qquad \leq cF_0\{D_{\mathrm{data}}^\varepsilon(t)+\varepsilon^2G_0\}+c\nu\varepsilon^2\int_0^t\|\nabla_{S^2}\mathbf{v}\|_{L^2(S^2)}^2\,ds \\
    &\qquad \leq cF_0\left\{\frac{1}{\varepsilon}\|\mathbf{u}_0^\varepsilon-\bar{\mathbf{v}}\|_{L^2(\Omega_\varepsilon)}^2+\frac{1}{\varepsilon\nu}\int_0^\infty\|\mathbf{f}^\varepsilon-\bar{\mathbf{f}}\|_{\mathcal{V}_\varepsilon^\ast}^2\,ds+\varepsilon^2G_0\right\}+c\varepsilon^2E_0
  \end{align*}
  by combining \eqref{E:DEst_Gl} and \eqref{Pf_AvC:Adif} and then using \eqref{E:S2_GlEn}.
  Hence, it follows that
  \begin{multline*}
    \|\mathcal{M}_\varepsilon^0\mathbf{u}^\varepsilon-\mathbf{v}\|_{L^\infty(0,\infty;L^2(S^2))}+\nu^{1/2}\|\mathcal{M}_\varepsilon^0\mathbf{u}^\varepsilon-\mathbf{v}\|_{L^2(0,\infty;H^1(S^2))} \\
    \leq c_0\Bigl\{\varepsilon^{-1/2}\|\mathbf{u}_0^\varepsilon-\bar{\mathbf{v}}_0\|_{L^2(\Omega_\varepsilon)}+(\varepsilon\nu)^{-1/2}\|\mathbf{f}^\varepsilon-\bar{\mathbf{f}}\|_{L^2(0,\infty;\mathcal{V}_\varepsilon^\ast)}+\varepsilon\Bigr\},
  \end{multline*}
  where $c_0>0$ is a constant depending only on $E_0$, $F_0$, and $G_0$ and thus independent of $\varepsilon$.
  This inequality implies \eqref{E:ACG_sol} provided that \eqref{E:ACG_data} is valid.
\end{proof}

\section*{Acknowledgments}
The work of the author was supported by JSPS KAKENHI Grant Number 23K12993.

\bibliographystyle{abbrv}
\bibliography{NS_Thin_Sphere_Ref}

\end{document}